\documentclass{amsart}

\usepackage[english]{babel}

\usepackage[letterpaper,top=2cm,bottom=2cm,left=3cm,right=3cm,marginparwidth=1.75cm]{geometry}

\usepackage{comment}

\usepackage{amsmath}
\usepackage{bbm}
\usepackage{graphicx}
\usepackage[colorlinks=true, allcolors=blue]{hyperref}
\usepackage{appendix}

\newtheorem{thm}{Theorem}
\newtheorem{lem}{Lemma}
\newtheorem{cor}{Corollary}

\newtheorem{open}{Open problem}
\newtheorem{prop}{Proposition}\makeatletter
\theoremstyle{remark}
\newtheorem{dfn}{Definition}[section]
\newtheorem{rem}[dfn]{Remark}

\numberwithin{equation}{section}

\newcommand{\red}[1]{{\color{red}#1}}
\newcommand{\blue}[1]{{\color{blue}#1}}

\newcommand{\ca}{{\rm Cap}}
\newcommand{\Z}{{\mathbb Z}}
\newcommand{\RR}{{\mathbb R}}

\newcommand{\R}{{\mathcal R}}

\newcommand{\nbd}{\text{Nbd}}

\title[Capacity of the random walk range]{Strassen's LIL and a Phase transition for the capacity of the random walk under diameter constraints }
\author{Arka Adhikari}
\address{Arka Adhikari\hfill\break
Department of Mathematics, University of Maryland- College Park, College Park, MD, USA}
\email{arkaa@umd.edu}
\author{Izumi Okada}
\address{Izumi Okada\hfill\break
Department of Mathematics, University of Tokyo, 3-8-1 Komaba, Meguro-ku, Tokyo 153-0041, Japan.}
\email{iokada@ms.u-tokyo.ac.jp}
\thanks{Research supported in part by JSPS KAKENHI Grant-in-Aid  for Early-Career Scientists (No.~JP24K16931) (I.O.).}
\thanks{Research partially supported by NSF grant DMS 2102842}

\begin{document}
\maketitle

\begin{abstract}
We discuss the relationship between the capacity and the geometry for the range of the random walk for $d=3$. 
In particular, we consider how efficiently the random walk moves or what shape it forms in order to maximize its capacity.
In one of our main results, we show a functional law for the capacity of the random walk. 
In addition, we find that there is a phase transition for the asymptotics of the capacity of the random walk when we condition the diameter of the random walk. 
\end{abstract}

\section{Introduction and Main results}
In this paper, we study the classical topic of simple random walks on $\Z^3$. 
Newtonian capacity 
is given for $d \ge 3$, by
\begin{align*}
\ca (A)
:= \sum_{x \in A} 
\mathbb{P}^x(T_A =\infty)
=\lim_{\|z\|\to \infty}\frac{\mathbb{P}^z(T_A <\infty)}{G(z)}, 
\end{align*}
where $G(x)$ denotes the Green's function of the walk, $T_A$ is the first hitting time to $A$ of a walk and $ \| \cdot \|$ is the Euclidean norm.
We let $R_n$ denote the capacity of the set of the random walk range up to the time $n$; that is, $R_n$ is a shorthand for $\ca(\{S_0,\ldots ,S_n\})$.  
The capacity of the random walk is a random variable 
that expresses the shape of the random walk 
using the hitting probabilities. 
The capacity of a random walk has been known to be an important quantity to understand when analyzing the random interlacement process (cf. \cite{SZ10}) and the uniform spanning tree (cf. \cite{HS20}). However, it is also an interesting quantity in its own right. 
The capacity is known to grow linearly in $n$ when $d\ge 5$, of the order 
$n/\log n$ when $d=4$, and $\sqrt{n}$ when $d=3$. The limit laws for the 
appropriately normalized fluctuations $R_n-\mathbb{E}R_n$ were established in the work of Chang \cite{Ch17}.
These limit laws have different forms in the three cases $d = 3$, $d = 4$, and $d\ge 5$. 
Importantly, in dimension three
$R_n$ is unconcentrated and has a limit law described by a non-trivial random variable.
In one of our main results, we express the shape of random walks in terms of functional spaces. 

To state the main results, for $d=3$, let
\begin{align*}
f_n(\cdot):=\frac{R_{n\cdot}}{h_3(n)}  \,, \qquad
g_n(\cdot):=\frac{R_{n\cdot}}{\hat{h}_3(n)}, \,
\end{align*}
where 
\begin{align*}
h_3(n):=\frac{\sqrt{6}\pi}{9}  (\log^{(3)} n)^{-1} \sqrt{n \log^{(2)} n}, \quad
\hat{h}_3(n):= \frac{\sqrt{6} \pi^2}{9} \sqrt{ n (\log^{(2)} n)^{-1}}
\end{align*}
with $\log^{(k)}:= \log \log^{(k-1)}$ with $\log^{(1)} =\log$. 
In \cite{DemboOkada}, Dembo and the second author showed  $\limsup_{n\to \infty} f_n(1)=1$ and $\liminf_{n\to \infty} g_n(1)=1$ a.s. 
Our result in this paper will give new insights into \cite{DemboOkada}. 
For $\ell_n \in C([0,1])$, $n \in \mathbb{N}$, let 
\begin{align*}
&C_1(\{\ell_n\}_{n=1}^{\infty}):=\{g \in C([0,1]): \exists (n_k) \text{ s.t. }  \lim_{k\to \infty} \sup_{x \in [0,1] }|\ell_{n_k}(x)-g(x)|=0\},
\end{align*}
and for $\ell_n \in \mathcal{M}$, $n \in \mathbb{N}$,
\begin{align*}
& C_2(\{\ell_n\}_{n=1}^{\infty}):=\{g \in \mathcal{M} : \exists (n_k) \text{ s.t. } \lim_{k\to \infty} |\ell_{n_k}(x)-g(x)|=0\text{ for continuity points }x \text{ of } g\}. 
\end{align*}
Here, $\mathcal{M}$ is the space of functions $\ell:[0, \infty)\to [0,\infty]$ such that $\ell(0)=0$, $\ell$ is right continuous on $(0,\infty)$, non-decreasing and $\lim_{t \to +\infty} \ell(t)=\infty$. 
$\{\ell_n\} \Rightarrow_i A$ means that $\{\ell_n\}$ is relatively compact and $C_i (\{\ell_n\})=A$ for $i=1,2$. 
We define the following functional sets:
\begin{align*}
  &\mathcal{S}:=\{\ell \in C([0,1]): \text{ absolutely continuous, }
  \int_0^1 \ell'(x)^{2} \text{d}x \le 1, 
  \ell(0)=0\} , \\ 
  & \mathcal{K}:=\{\ell\in \mathcal{M}:\int_0^\infty \ell(x)^{-2} \text{d}x\le 1 \}.
\end{align*}
These functional sets are the same as the sets found in the study of the functional LIL of the  standard  Brownian motion (cf. \cite{CKW, Strassen}).

To ease the presentation, we omit hereafter the integer-part symbol $\lceil \cdot \rceil$. 
If there exists $c\in (0,\infty)$ 
such that $c a_n \le b_n \le c^{-1}a_n$, 
we write $a_n \asymp b_n$. 
$a_n\gg b_n$ means $a_n/b_n \to \infty$ as $n\to \infty$ and $a_n\ll b_n$ means $b_n/a_n \to \infty$ as $n\to \infty$.

\begin{thm}\label{m1}
    For $d=3$, almost surely, as $n \to \infty$, 
\begin{align*}
\{f_n \} \Rightarrow_1 \mathcal{S},\, \qquad
\{g_n \} \Rightarrow_2 \mathcal{K}. \,
\end{align*}
\end{thm}
In general, extending Strassen’s LIL beyond the pointwise distribution of ordinary Brownian motion is of substantial theoretical importance, as it demonstrates that the functional law of the iterated logarithm is not merely a consequence of Gaussian path structure, but rather a manifestation of deeper probabilistic principles. 
By describing the set of almost sure limit points, it establishes a connection between probability theory and analysis. For example, $\mathcal{S}$ provides a link with the unit ball of the Sobolev space $H^1_0(0,1)$.  

{
Indeed,  in contrast to the previous work \cite{DemboOkada}, we had to develop an essentially new methodology to develop these connections with analysis and determine the functional properties of the capacity of the random walk range.} At a conceptual level, in order to capture fine asymptotic fluctuations for a wide class of random walks, it is necessary to perform a detailed analysis of the joint distributions associated with multiple points of the random walk range. This refined control of multipoint interactions goes beyond the techniques employed in previous studies \cite{DemboOkada} and forms a key component of our approach, especially since the formulas defining the capacity correlate the random walk over long time scales. 

Next, we discuss the relationship between the diameter and the capacity of a random walk range. 
Set $D_n:= \max _{0 \le i \le n}\|S_i\|$, that is, $D_n$ is the diameter of a random walk range. 
Let 
\begin{align*}
&j_3(n,k_n)=j_3(n):= k_n \sqrt{n \log^{(2)} n}(\log^{(3)} n)^{-1}\\
\text{ with } 
 &\sqrt{ n (\log^{(2)} n)^{-1}}(\log^{(3)} n)^2 \ll j_3(n) \le  \sqrt{ (2/3) n \log^{(2)} n}. 
\end{align*}

We let $\mathcal{B}(r)=\mathcal{B}_r$ denote the set of the points on $\Z^3$ included in the ball of radius $r$ centered at $0$.
Set
\begin{align*}
  &A_n^1 := \{ h_3(n)-R_n < \epsilon R_n\}, 
  \quad A_n^2:= \{\ca (\mathcal{B}_{j_3(n)})-R_n < \epsilon R_n\}, \\
  &B_n := \{j_3(n) <D_n < (1+\epsilon)j_3(n)  \}.
\end{align*}

\begin{thm}\label{m2}
Consider $d=3$. 
(i) If $k_n \to \infty$, 
for sufficiently small $\epsilon>0$, 
    \begin{align*}
\mathbb{P}( A_n^1 \cap B_n \quad \text{i.o.})=1, \quad 
\mathbb{P}( A_n^2 \cap B_n \quad \text{i.o.})=0.
\end{align*}
We also remark that $h_3(n)$ is the largest possible value that $R_n$ can take infinitely often, so we cannot improve on the lower bound of $R_n> (1-\epsilon) h_3(n)$ in the above statement.


(ii) If $k_n \asymp 1$, for  sufficiently small $\epsilon>0$, 
\begin{align*}
\mathbb{P}( A_n^1 \cap B_n \quad \text{i.o.})=0, \quad 
\mathbb{P}( A_n^2 \cap B_n \quad \text{i.o.})=0.
\end{align*}

(iii) If $k_n \to 0$, 
for sufficiently small $\epsilon>0$
   \begin{align*}
\mathbb{P}( A_n^1 \cap B_n \quad \text{i.o.})=0, \quad 
\mathbb{P}( A_n^2 \cap B_n \quad \text{i.o.})=1. 
\end{align*}
We also remark that under the condition that $D_n < j_3(n)$, $R_n$ can be no bigger than $\ca(\mathcal{B}_{j_3(n)})$, so we cannot improve on the lower bound of $R_n$ given in the above statement. 

\end{thm}
Another way to express this theorem is given by the following corollary.


\begin{cor}\label{main:cor}
    For $d=3$, almost surely,
    \begin{align*}
            1 + \delta(\epsilon) 
        =\begin{cases}
        \limsup_{n\to \infty} \frac{R_n}{h_3(n)} 1_{B_n} \quad &\text{ if } (i) \quad k_n \to \infty, \\
        \limsup_{n\to \infty} \frac{R_n}{ c_0 h_3(n)} 1_{B_n} \quad &\text{ if } (ii) \quad k_n \asymp 1, \\
        \limsup_{n\to \infty} \frac{R_n}{\ca (\mathcal{B}_{j_3(n)})} 1_{B_n} \quad &\text{ if } (iii) \quad k_n \to 0, 
        \end{cases}
    \end{align*}
    where $\delta(\epsilon)$ is a deterministic sequence with $\delta(\epsilon)\to 0$ as $\epsilon \to 0$ and in the case $(ii)$,   
    $$0 
    <c_0<1$$ with $c_0$ a deterministic constant (we know that $c_0$ is deterministic by the Hewitt-Savage zero-one law (see the explanation in Subsection 4.3 in \cite{DemboOkada})). 
    In the case (ii), we also have $\limsup_{n\to \infty} \frac{R_n}{ \ca (\mathcal{B}_{j_3(n)})} 1_{B_n} = c_1(1 + \delta(\epsilon))$, where $0<c_1<1$ with $c_1$ a deterministic constant. 
    
\end{cor}

    One of our original motivations in engaging in this study was to understand the large deviation behavior of the capacity of the random walk at multiple time points simultaneously. On a heuristic level, to generate large deviation behavior for the capacity of the random walk $S$ at a single time $n$, we have the ability to freely adjust the behavior of the random walk any time before time $n$. However, there is a different story if one wants to attain large deviation behavior at both times $t_1 n $ and $t_2 n$. One could adjust the random walk between times $0$ and $t_1n$ in order to obtain some desired large deviation behavior of $R_{t_1n}$. However, the part of the walk between times  $0$ and $t_1n$ now becomes a constraint that one has to overcome in order to obtain an appropriate large deviation behavior at { time $t_2n$.} Indeed, the same strategy that could have been used before in order to obtain large deviation behavior at time $t_2 n$ originally is no longer a viable strategy when the constraint at time $t_1n$ is added. 
    
    Furthermore, in contrast to the volume of a random walk in high dimensions, the capacity of a random walk is a highly complex formula that correlates over very long ranges of the random walk. This same fact means that many questions that could be readily answered about the volume do not necessarily have a corresponding explicit answer for the capacity. For example, it is not known what is the exact value of the capacity of the cube of side length one \cite{PS51}; thus, the determination of exact constants, as we seek to do here, is a significantly more challenging endeavor. 

    In this paper, our goal is to get a grasp of these questions to understand the multi-point large deviation behavior as well as trying to understand the capacity under constraints. In Theorem \ref{m1}, we establish Strassen's law of the iterated logarithm for both the $\limsup$ and $\liminf$ behavior of the capacity of the random walk. In Theorem \ref{m2}, we establish a result that derives the $\limsup$ behavior for the capacity of the random walk under the condition that we restrict the diameter of the walk. In this case, we observe a phase transition in the behavior of the random walk that attains the $\limsup$. We also aim to identify the optimal way in which a random walk can increase the capacity of its trajectory behind the proof. 

    In the first phase (i) of Corollary \ref{main:cor}, where we restrict the diameter of the walk to be above the critical scale of $\frac{\sqrt{n\log^{(2)} n}}{\log^{(3)} n},$ the $\limsup$  of the random walk is determined by the local behavior. Here, the goal is for the random walk to spread out as far as possible, while still being restricted to having a radius of $j_3(n)$. The random walk will still `look' like it stretches uniformly around the boundary of the sphere; however, one can attempt to interpret this behavior as being a union of a growing number of disjoint strings.

    In contrast, in phase (iii) of Corollary \ref{main:cor}, where we are below the critical scale, the $\limsup$ behavior is determined by the global shape of the random walk. Now, the walk will still wrap uniformly around the boundary of the sphere; the capacity of the walk will also start to match the capacity of the sphere. As such, one can interpret this behavior as the random walk starting to densely pack the boundary of the sphere. 

    Phase (ii) is the regime where the local and global contributions start to compete with each other. One can show that one will not get the disjoint string behavior of phase (i), nor will one get the dense packing of the sphere as in phase (iii). We interpret what occurs here as a sparse packing of the sphere. 

    In particular, the phase (ii) is not intuitively obvious, that is, there is no reason why the intermediate regime should  exist. In general, it is possible for a very sparse random walk to attain the full capacity of the surface it is embedded inside. In the course of our proof, we had to carefully determine the sparsity properties of the random walk while it is embedded inside the sphere, and then conclude that this level of sparsity was enough to cause a macroscopic reduction in the capacity. In particular, this required a very careful analysis of the intersection properties of two random walks and other delicate properties of its distribution.

In the future, we would like to evaluate the capacity of a random walk when it stays in a general convex region, rather than conditioning only on the diameter of the random walk. In the case where the diameter is conditioned, as in this study, symmetry implies that the optimal strategy to maximize capacity is for the local time of the random walk to be roughly uniform within the ball. However, when considering a general region, we expect that the optimal strategy is not uniform, making it difficult to directly extend this result.
We propose the open problem as follows: 
\begin{open}
 Consider some connected domain $W \subset  \RR^3$ and 
 $W_n:=   \Z^3 \cap \lceil j_3(n,1) \rceil W $. 
 Let $B'_n:=\{ S_i \in W_n, 1 \le i\le n \}$. 
 For each $W$, can one choose an appropriate $(a_n)_{n=0}^\infty$ with $a_n \to \infty$ such that almost surely, 
 \begin{align*}
     \limsup_{n\to \infty} a_n^{-1} R_n 1_{B'_n}=1?
 \end{align*}
\end{open}

\begin{open}
By the  Doob-Meyer decomposition, we have
\begin{align*}
R_n = A_n+M_n, 
\end{align*}
where $M_n$ is a Martingale with random walk filtration and $A_n$ is an  increasing predictable process. 
Let $\hat{f}_n(\cdot):=M_{n\cdot}/h_3(n)$. 
Can one choose an appropriate $(a_n)_{n=0}^\infty$ with $a_n \to \infty$ such that almost surely, 
    \begin{align}\label{functional:limit}
        \limsup_{n \to \infty} a_n
        \inf\{ \|\hat{f}_n(\cdot)-h \|_\infty :  h \in \mathcal{S}\}
        =1?
    \end{align}
    In the case of Brownian motion for $d=1$, 
    the question has been studied by a number of authors (cf. \cite{Bo78, Grill92, Tala92}). 
Especially, Gorn and Lifshitz \cite{GL99} expressed \eqref{functional:limit} 
    in terms of the smallest eigenvalue of the Sturm–Liouville problem if $h$ is included in $\mathcal{S}$ and $\int_0^1 (h'(t))^2\text{d}t=1$. 
We conjecture that a related formula appears in the case of the capacity of the random walk in $d=3$.  

In \cite{AdhikariOkada}, in $d=4$, we showed that the rate of moderate deviations for the capacity of the random walk can be expressed in terms of the best constant in the generalized Gagliardo-Nirenberg inequality. This was one surprising link to the field of analysis, and we are interested in finding further links to analysis in the background of this topic. 
\end{open}

Now, we will discuss some of the difficulties present in deriving a Strassen's LIL as well as determining the $\limsup$ behavior under constraints on the diameter.

Strassen's LIL is a result that controls the $\limsup$ or $\liminf$ behavior of the walk at multiple time points simultaneously. As we have mentioned earlier in our discussion of the multi-point large deviation principle, when we only consider the behavior at one time $n$, we have the ability to freely adjust the random walk from times $0$ to $n$ in order to get some desired $\limsup$ or $\liminf$ behavior. If we have to simultaneously derive $\limsup$ or $\liminf$ behavior at multiple points $t_1n, t_2n, \ldots, t_k n$, we have much more limited freedom in adjusting the random walk.

Namely, we are now concerned that the constraints that we impose for the random walk at earlier times will make it far more difficult to obtain some desired $\limsup$ behavior at later times. On a technical level, this manifests itself in the relatively difficult formulas to obtain accurate estimates of the capacity of the random walk. Indeed, for a set of $|S|$ points, one has to determine the capacity of the random walk by solving a system of $|S|$ linear equations. In general, without extra information, the best bounds are of the form,
\begin{equation*}
 \frac{|S|}{\max_{s \in S} \sum_{s' \in S} G(s-s')}\le\ca(S) \le \frac{|S|}{\min_{s \in S} \sum_{s' \in S} G(s-s')}.
\end{equation*}

To obtain bounds that are optimal up to a constant, using a bound of the above form, one would need to have very good control of both the maximum sum of the Green's function around a single vertex as well as the minimum sum of the Green's function entries. Indeed, even once we were able to make a heuristic guess regarding the types of structures that would obtain the Strassen's LIL in Section \ref{sec:limsup}, we have to be very careful in our technical analysis for the sum of the Green's functions, as in Lemmas \ref{thm:LowerboundofGreenspec}  and \ref{thm:LowerboundofGreengen}. In particular, it was essential to derive a lower bound on the sum of the Green's function for nearly every single point of the random walk.  This required a careful analysis and understanding of the geometry in order to ensure that minor effects from the fluctuations of the random walk deteriorate the desired estimates. This is quite in contrast to a proof of Strassen's LIL of the magnitude of the random walk in one dimension; in Strassen's case, he had direct formulas for multipoint distributions of the random walk.

These difficulties reappear when trying to determine the $\limsup$ behavior under constraints as in Theorem \ref{m2}. If all of these constraints were absent from the random walk, then, as shown in the paper \cite{DemboOkada}, the random walk would attempt to straighten out into a line. However, the constraint on the diameter causes the walk to fold in on itself. As before, to obtain estimates on the capacity, one would need to have very strong controls on the Green's function sums. Because the random walk starts to fold in on itself, it is no longer sufficient to only consider the local contribution to the Green's function sums; one must now consider the contribution from the local scale, the global scale, and mesoscopic scales. Indeed, one needs tight controls on the geometry of the random walk along individual `segments' in order to make sure that the mesoscopic contributions do not overpower the contributions of the other scales. This was the object of Theorem \ref{thm:lowerbndicap} and the preceding lemmas. 

We hope that the techniques developed in this paper would generalize to analyzing the capacities of random systems under arbitrary constraints and could reveal a deeper link between the geometry of the capacity of the random walk in $d$-dimensions and the volume of the random walk in $d-2$ dimensions.

\section{Strassen's LIL for the $\limsup$ of a Random Walk Range:Notation and Estimates } \label{sec:limsup}

\subsection{Key proposition}
Our first order of business is to derive the following multi-point estimates for $\limsup$ on the three-dimensional capacity, which yields the proof of $\limsup$ in Theorem \ref{m1} in the next section. 
Let $\{S_n\}$ denote a simple random walk in 3 dimensions, with $\{S^i_n\}$ for $i = 1,2,3$ denoting the coordinates of this random walk. We let $S_{[m,n]}$ denote the entire range of the random walk between times $m$ and $n$ and define $T_A:=\inf\{i \ge 1: S_i\in A \}$ for $A\subset \Z^3$. 
We write $\ca (S_{[m,n]})$ as $R[m,n]$. We set $\mathcal{A}_k=(a_1,\ldots, a_k)$ and $\mathcal{T}_k=(t_1,\ldots, t_k)$ for constants $0<a_1  , a_2 , \ldots, a_k <1$,  and times $0 = t_0 <t_1 \le t_2 \le \ldots \le t_k \le 1$, and fix  some small constant $\tilde{\delta}$. 
We sometimes omit the letter `k' in $\mathcal{A}_k$ and $\mathcal{T}_k$ for typographical reasons. 

Our goal in this section is to estimate the probability of the following event: 

\begin{equation*}
\begin{aligned}
 \mathcal{E}^{\mathcal{A},\tilde{\delta}}_{n,\mathcal{T}} 
 := \bigcap_{i=1}^k \{ a_i(1-\tilde{\delta}) h_3((t_i-t_{i-1}) n) 
\le R_{t_i n} - R_{t_{i-1} n}\le  a_i(1+\tilde{\delta}) h_3((t_i - t_{i-1})n ) \}.
\end{aligned}
\end{equation*}
We let $\phi(x)$ be the function
$$
\phi(x) = \sqrt{\frac{2}{3} x \log^{(2)} x}.
$$
What we will show is that the event $\mathcal{E}^{\mathcal{A},\tilde{\delta}}_{n,\mathcal{T}}$ will be implied by the following event,
\begin{equation}\label{def:F}
\mathcal{F}^{\mathcal{A},\delta}_{n,\mathcal{T}} := \bigcap_{i=1}^k 
\{a_i (1-\delta)\phi(n (t_i- t_{i-1})) \le S^1_{t_{i}n} - S^1_{t_{i-1}n} \le a_i(1+\delta) \phi(n(t_i - t_{i-1})) \},
\end{equation}
where $\delta$ will be chosen sufficiently small relative to $\tilde{\delta}$.


\begin{prop}\label{mainpro1}
For $\tilde{\delta}>0$, there is sufficiently small $\delta>0$ such that for any $k\in N$, $0<a_1  , a_2 , \ldots, a_k <1$, $0<t_1\le t_2\le \ldots \le t_k$, 
\begin{align*}
\mathbb{P}(
\mathcal{E}^{\mathcal{A},\tilde{\delta}}_{n,\mathcal{T}}| 
\mathcal{F}^{\mathcal{A},\delta}_{n,\mathcal{T}})
=1-o(1).
\end{align*}
\end{prop}

In what proceeds, we will prove the upper and lower bounds of Proposition \ref{mainpro1} in the following two sections.

\subsection{Deriving Upper Bounds on the Capacity in Proposition \ref{mainpro1}}

In this subsection, we show the upper bound for the capacity of the random walk on the event $\mathcal{F}^{\mathcal{A},\delta}_{n,\mathcal{T}}$ in order to prove Proposition \ref{mainpro1}. To obtain this bound, our goal in this section is to show Proposition \ref{prop:upperboundcap}. 

 \begin{prop} \label{prop:upperboundcap}
    
Conditioned on  the event 
$\Lambda_{a,\delta}:= \{ a \phi(n) \le S_n^1 \le a(1+ \delta) \phi(n) \}$, we have with probability  $1- o(1)$ that,
\begin{equation*}
R_n \le a (1+ \tilde{\delta}) h_3(n) ,
\end{equation*}
where $\tilde{\delta}$ can be any constant strictly greater than $\delta$. 

\end{prop}
After finishing Proposition \ref{prop:upperboundcap}, we can easily obtain Proposition \ref{mainpro1} as follows: 

 \begin{proof}[Proof of Proposition \ref{mainpro1} given Proposition \ref{prop:upperboundcap}]
We first begin with the remark that the event,
\begin{equation}\label{firstrn}
\{R_{t_{i+1}n} - R_{t_i n} \le a_{i+1}( 1+ \tilde{\delta}) h_3((t_{i+1} - t_i)n)\},
\end{equation}
would be entailed by the event,
\begin{equation*}
\{R[t_in+1, t_{i+1}n] \le a_{i+1}(1 + \tilde{\delta}) h_3((t_{i+1} -t_i)n)\},
\end{equation*}
since $R_{t_{i+1}n} - R_{t_i n} \le R[t_in+1, t_{i+1}n]$, as derived from the subadditivity of the capacity (cf. \cite{LA91}). 
Hence, to show the upper bound of the event \eqref{firstrn}, we should consider the new collection of events
\begin{equation*}
\mathcal{G}^{\mathcal{A},\tilde{\delta}}_{n,\mathcal{T}}:= \bigcap_{i=0}^{k-1} \left\{R[t_in+1, t_{i+1}n] \le a_{i+1}(1+\tilde{\delta}) h_3((t_{i+1} -t_i)n) \right\}.
\end{equation*}
If we show that,
\begin{equation*}
\mathbb{P}(\mathcal{G}^{\mathcal{A},\tilde{\delta}}_{n,\mathcal{T}}| \mathcal{F}^{\mathcal{A},\delta}_{n,\mathcal{T}}) = 1- o(1),
\end{equation*}
then we have
\begin{equation*}
\mathbb{P}\bigg(\bigcap_{i=0}^{k-1}\{R_{t_{i+1}n} - R_{t_i n} \le a_{i+1}(1+\tilde{\delta}) h_3((t_{i+1} - t_i)n)\} | \mathcal{F}^{\mathcal{A},\delta}_{n,\mathcal{T}}\bigg) = 1-o(1).
\end{equation*}
Now, all the events $\{R[t_in+1, t_{i+1}n] \le a_{i+1}(1+ \tilde{\delta}) h_3((t_{i+1} -t_i)n)\}$ deal with different segments of the random walk and are independent of each other. 
Hence, we have the upper bound for the capacity, as is necessary. 
\end{proof}


To reach Proposition \ref{prop:upperboundcap}, the argument proceeds through several layers of auxiliary results. Lemmas \ref{est:lem1} and \ref{lem:TransverseFluc} first give precise probabilistic bounds showing that the random walk stays close to a narrow tube around its main direction and that lateral deviations are small. These lemmas lead to Corollaries \ref{cor:Inmiddle} and \ref{cor:Bmiddle}, which ensure that such controlled behavior holds uniformly along the path. Lemma \ref{thm:LowerboundofGreenspec} then uses this structure to obtain a quantitative lower bound for the sum of Green’s functions, capturing how different parts of the walk interact. Lemma \ref{thm:LowerboundofGreengen} extends this estimate to nearby segments, and together these results make it possible to conclude the upper bound on the capacity stated in Proposition \ref{prop:upperboundcap}.


We now begin to prove Lemma \ref{lem:TransverseFluc}. 
The following definition introduces notions that will allow us to bound the projections of the random walk below and above a specific line.
\begin{dfn}

Given some constants $a,\delta, \kappa$,  we define the line,
\begin{equation*}
U(l):=l \frac{a (1+\delta) \phi(n) }{(1-\kappa) n}, \quad 
L(l):= l \frac{a \phi(n) }{(1 +\kappa) n}.
\end{equation*}

For $r(n):= \frac{n (\log^{(3)} n)^{3/2}}{\log^{(2)} n }$, we define the events,
\begin{equation*}
\begin{aligned}
& U^{P,k}_a:= \bigcap_{ n - k \ge l \ge r(n)} \{  S^1_{k+l}-  S^1_{k} \le U(l) \} ,\quad
 L^{P,k}_a := 
\bigcap_{n -k  \ge l \ge r(n) } \{  S^1_{ l+k} -S^{1}_{k} \ge  L(l)\} ,\\
& U^{N,k}_a
:=\bigcap_{-k \le l \le -r(n)} \{  S^1_{l+k} -S^1_{k} \le  L(l)\}, \quad
L^{N,k}_a:= \bigcap_{  -k \le l \le -r(n)} \{  S^1_{l+k}  -S^1_k \ge  U(l)\} .
\end{aligned}
\end{equation*}
\end{dfn}

\begin{rem} In what proceeds, we will choose constants that $\delta, \kappa$ are sufficiently small (at least less than 1) and $\kappa \ge \delta$. With this choice of parameters, we will eventually get a later cutoff choice
$$
 n C(\delta,\kappa) \le \frac{3}{4}n, 
 \quad C(\delta,\kappa):= \frac{\delta + \kappa + \delta \kappa}{\delta + 2 \kappa + \delta \kappa} (1-\kappa).
$$
Furthermore, we have that $U(l) \le a \phi(n)$ for $l \le \frac{3n}{4} $. 
\end{rem}

What one should imagine in this picture is that $S^1_n - S^1_0$ has a slope of approximately $\frac{a \phi(n)}{n}$. Then, we can bound the fluctuation of $S$ (at least for times sufficiently far from $0$) by lines going through $S^1_0$ whose slopes slightly differ from $\frac{a \phi(n)}{n}$. Furthermore, these approximations should hold even if we consider the walk $S^1$ around some intermediate point, $S_k$; namely, we would have that $S^1_{m+k} - S^1_{k} \approx \frac{m}{n} [S^1_n -S^1_0]$.

\begin{lem}\label{est:lem1}
Fix some constant $a>0$ and some small parameters $\delta>0$ and $\kappa >0$.
 Then, for $\zeta_n = O( \exp[- (\log^{(3)} n)^{3/2-\epsilon}])$ and $\epsilon >0$, we have that,
\begin{equation} \label{eq:lowprobestimate}
\mathbb{P} \left( U^{P,0}_a \cap L^{P,0}_a|\Lambda_{a,\delta}\right)  = 1- \zeta_n.
\end{equation}
\end{lem}

\begin{proof}

The proof of $P(L^{P,0}_a|\Lambda_{a,\delta})=1-\zeta_n$  is very similar to the proof of \cite[Lemma  3.2]{DemboOkada}. First of all, the event $L^{P,0}_a$ is the event that was exactly dealt with in said Lemma.  The only point we must acknowledge is that, by the Local Central Limit Theorem, (see \cite[Theorem 1.2.1]{LA91}), 
\begin{equation*}
\mathbb{P}(S^1_n \ge a(1+\delta) \phi(n)) \le C\exp[- (3/2) a^2 \delta \log^{(2)} n] \mathbb{P}(S^1_n \ge a \phi(n))  \ll \exp[- (\log^{(3)} n)^{3/2-\epsilon}] \mathbb{P}(S^1_n \ge a \phi(n)). 
\end{equation*}
Thus we have $P(L^{P,0}_a |\Lambda_{a,\delta})=1-\zeta_n$. 

Next, we will show $P(U^{P,0}_a |\Lambda_{a,\delta})=1-\zeta_n$.
The key idea in the proof of the lower bound is that any time the walk $S^1_l$ goes below the lower bound $L(l)$, it must be followed by a rare increase from $S^1_l$ to $S_n^1$. The argument then follows from a stopping time argument based on the location of the crossing below $L(l)$, and all of these locations can be treated uniformly. In the analysis for the upper bound, we have to distinguish cases based on whether $S^1_l$ first exceeds $U(l)$ when $l$ is relatively close to $n$ or relatively far away from $n$. In what follows, $\textbf{Case 1}$ will first treat the case where $l$ is relatively close to $n$. This follows from a symmetry argument from the estimate on $P(L^{P,0}_a|\Lambda_{a,\delta})$. \textbf{Case 2} treats those points extremely close to $n$ that do not follow from the symmetry argument. \textbf{Case 3} will treat points where $l$ is close to $0$ and uses a stopping time argument similar to that appearing in the proof of the lower bound.   

In what follows, we will assume that $\kappa \ge \delta$ and $\kappa,\delta$ are chosen to be sufficiently small. In particular, the threshold that appears in $\textbf{Case 1}$, will satisfy  $C(\delta,\kappa)n \le \frac{3}{4}n$. Furthermore, we will also have that $U(l) < \frac{5}{6}a \phi(n)$ for $l \le \frac{3n}{4}$. 

 \textbf{Case 1: $ n - r(n)  \ge l \ge  C(\delta, \kappa)n $ }

In this case, we can use the estimate for  $L^{P,0}_a$ combined with a symmetry argument. Namely, by considering the random walk, a lower bound event around the point $S_0$ becomes an upper bound event around the point $S_n$. 
Consider the following transformation: if $(s_1,s_2,\ldots,s_n)$ is a sequence of increments of the random walk such that the event $L^{P,0}_a$ holds, then the event $U^{N,n}_a$ holds for the reverse walk with increments $(s_n,s_{n-1},\ldots,s_2,s_1)$. 
 Indeed, we can now say that $S^1_n - S^1_{l} \ge L(n-l)$ for the reversed walk.  Thus, we have
$$
S^1_l \le S^1_n - L(n-l) \le a\phi(n)(1+ \delta) - \frac{a \phi(n) (n- l)}{n(1+ \kappa)}.
$$
Notice, that above we used that $S^1_n$ must be bounded by $a\phi(n) (1+\delta)$ on the event $\Lambda_{a,\delta}$. 

If we then use our assumption on the value of $l$, we notice that the difference above is less than $U(l)$ by direct computation. This will finish the proof of the upper bound for this range of $l$.


\textbf{Case 2}: $|l- n| \le r(n)$:

 If $S^1_l$ is greater than $U(l)$ for at least one value of $l$ in the regime $|l - n| \le r(n)$, then the reversed random walk $S^1_{n-1}- S^1_n, S^1_{n-2}-S^1_n, \ldots ,S^1_{n-r(n)} - S^1_n$, must attain a maximum value greater than,
 \begin{equation*}
\frac{a(1+ \delta) (n- r(n) )\phi(n)}{n(1- \kappa)} - a(1+\delta) \phi(n),
 \end{equation*}
 for some $S^1_{n-l} - S^1_n$ with $ l \le r(n)$. 
 Indeed, we must have $S^1_n \le a(1+\delta) \phi(n)$. Furthermore, if there exists $S^1_l \ge U(l)$ for $|l -n| \le r(n)$, then for this $l$ we must have $S^1_l \ge U(l) \ge U( n -r(n)). $ Then, we would also have $S^1_l - S^1_n \ge U(n -r(n)) - a(1+\delta) \phi(n) = \frac{a(1+\delta) (n- r(n)) \phi(n)}{n(1-\kappa)} - a(1+\delta)\phi(n)$. 
 For $n$ large, the difference above is of order $\phi(n)$, and the reflection principle for random walks shows that the chance that such a fluctuation could occur is bounded from above by 
 $$
 O\left( \exp \left[ - \frac{C' (\phi(n))^2 }{r(n)} \right] \right) = O\left(\exp\bigg[- C'  \frac{(\log^{(2)} n)^2}{(\log^{(3)} n)^{3/2}}\bigg]\right),
 $$
 where $C'$ is some constant that depends on $\delta$ and $\kappa$.
 
 This is significantly smaller than  $\exp[-(\log^{(3)} n)^{3/2-\epsilon}] \mathbb{P}(S^1_n \ge a \phi(n)).$

\textbf{Case 3:} $l \le   C(\delta,\kappa) n. $

This final case will allow us to finish the proof that $\mathbb{P} ((U^{P,0}_a )^c|\Lambda_{a,\delta})=\zeta_n$. 
We introduce the stopping time $\tau:= \min \left\{  l \ge r(n) : S^1_l \ge U(l)\right\}.$ 
We consider the interval $J_j := \left( \frac{jn}{\log^{(2)} n}, \frac{(j+1)n}{\log^{(2)} n} \right].$ 
Note that 
$$
\begin{aligned}
\mathbb{P}((U^{P,0}_a )^c \cap \Lambda_{a,\delta})
&\le \sum_{j\le (\log ^{(2)} n) C(\delta,\kappa)}  \mathbb{P}(\{\tau\in J_j \}\cap \Lambda_{a,\delta}) 
+ \mathbb{P}\left(\left\{\tau= r(n) \right\}\cap \Lambda_{a,\delta}\right) \\&+ 
\mathbb{P}\left( \{\tau \ge n C(\delta,\kappa)\} \cap \Lambda_{a,\delta}\right).
\end{aligned}
$$
By \textbf{Cases 1 and 2}, the second line can be shown to be bounded by $\exp[-(\log^{(3)} n)^{3/2- \epsilon} ] \mathbb{P}(\Lambda_{a,\delta})$. In what proceeds, we will bound the terms in the first line above. 

We can write the first event as,
\begin{equation} \label{eq:splitJj}
\begin{aligned}
\mathbb{P}(\{\tau \in J_j\}\cap  \Lambda_{a,\delta})& = \sum_{k \in J_j} \mathbb{P}(a \phi(n) \le S_n^1 \le a(1+ \delta) \phi(n)| \tau= k) \mathbb{P}(\tau = k) \\
& \le \max_{k \in J_j}\mathbb{P}(S^1_n - S^1_k \ge a \phi(n) - U(k)) \sum_{k \in J_j} \mathbb{P}(\tau = k).
\end{aligned}
\end{equation}

Note that in the above equation, we used that $U(k) \le \frac{5}{6} a \phi(n)$  for $k \le \frac{3}{4}n$. Thus, when $S^1_k$ exceeds $U(k)$, $S^1_n$ still has to increase to attain the value $a\phi(n)$ or greater. 
Now, notice that for any $D>0$,
$$\sum_{k \in J_j}\mathbb{P}(\tau = k) \le \mathbb{P}\left(S^1_{ \frac{nj}{\log^{(2)} n}} \ge U\left(\frac{nj}{\log^{(2)} n}\right) - D\right) + \mathbb{P}\left(\sup_{k \in J_j} \left|S^1_k - S^1_{\frac{nj}{\log^{(2)} n}}\right| \ge D\right).$$

In other words, if the stopping time $\tau$ occurs in the interval $J_j$, then it either must be the case that $S^1_k$ is large at the beginning of the interval (corresponding to the first term), or $S^1$ has a large fluctuation inside the interval $J_j$ (corresponding to the second term). 
We can choose $D = \sqrt{\frac{nj \log^{(3)} n}{\log^{(2)} n}}$ in order to ensure both that $D \ll U\left(\frac{nj}{\log^{(2)} n}\right)$ and that,
$$
\mathbb{P}\left(\sup_{k \in J_j} \left|S^1_k - S^1_{\frac{nj}{\log^{(2)} n}}\right| \ge D\right) \le \mathbb{P}\left(S^1_{ \frac{nj}{\log^{(2)} n}} \ge U\left(\frac{nj}{\log^{(2)} n}\right) - D\right).
$$

Indeed, note that $$
U\left( \frac{nj}{\log^{(2)}n}\right) = \frac{(1+\delta) a\phi(n)}{n(1-\kappa)} \frac{nj}{\log^{(2)}n} = O\left( \frac{\sqrt{n} j}{ \sqrt{\log^{(2)}}n} \right) \gg \sqrt{\frac{n j \log^{(3)}n}{\log^{(2)}n}},
$$
since we assume $j \ge ( \log^{(3)} n)^{3/2}$.
This shows that $D \ll U\left( \frac{nj}{\log^{(2)}n} \right).$

Now, by using the reflection principle, we see that,
\begin{equation}
\mathbb{P}\left( \sup_{k \in J_j}\left| S^1_k - S^1_{ \frac{nj}{\log^{(2)}n}}\right| \ge D \right) \le 2 \exp\left[ - \frac{D^2}{ \frac{n}{\log^{(2)}n}}\right] \le 2 \exp\left[- j \log^{(3)} n\right]. 
\end{equation}

Furthermore, we have,
\begin{equation}
\mathbb{P}\left(S^1_{ \frac{nj}{\log^{(2)} n}} \ge U\left(\frac{nj}{\log^{(2)} n}\right) - D\right) \asymp \exp \left[ (1- o(1))^2 \frac{(1 +\delta)^2}{(1-\kappa)^2} \frac{a^2\phi(n)^2}{ \frac{2}{3}n} \frac{j}{\log^{(2)}n} \right] = \exp[ - O(j)].
\end{equation}
This will give us for large enough $n$ that,
$$
\mathbb{P}\left(\sup_{k \in J_j} \left|S^1_k - S^1_{\frac{nj}{\log^{(2)} n}}\right| \ge D\right) \le \mathbb{P}\left(S^1_{ \frac{nj}{\log^{(2)} n}} \ge U\left(\frac{nj}{\log^{(2)} n}\right) - D\right).
$$

Now, using this, we return to equation \eqref{eq:splitJj} and observe that,
$$
\mathbb{P}(\{\tau \in J_j\}\cap  \Lambda_{a,\delta}) \le 2 \mathbb{P}\left(S^1_n - S^1_{\frac{(j+1) n}{\log^{(2)}n }} = a \phi(n) - U\left( \frac{(j+1) n}{\log^{(2)}n} \right)\right) \mathbb{P}\left(S^1_{ \frac{nj}{\log^{(2)} n}} \ge (1- o(1)) U\left(\frac{nj}{\log^{(2)} n}\right) \right) .
$$
To obtain an estimate on the quantity on the right-hand side above, we can use estimates obtained by the local central limit theorem.

Observe that the right-hand side above can be bounded by,
\begin{equation}
\begin{aligned}
&\exp\left[ - \frac{(a\phi(n))^2\left[ 1-  \frac{1+\delta}{1-\kappa}\frac{(j+1)}{\log^{(2)}n} \right]^2 }{ \frac{2n}{3} \left[ 1- \frac{j+1}{\log^{(2)}n}\right]}\right] \exp\left[ -\frac{(a \phi(n))^2}{\frac{2n}{3}} (1- o(1))^2 \frac{(1+\delta)^2}{(1- \kappa)^2}  \frac{j}{\log^{(2)}n}\right]\\
& =\exp\left[ - \frac{(a\phi(n))^2\left[ 1-  \frac{1+\delta}{1-\kappa}\frac{(j+1)}{\log^{(2)}n} \right]^2 }{ \frac{2n}{3} \left[ 1- \frac{j+1}{\log^{(2)}n}\right]}\right] \exp\left[ -\frac{(a \phi(n))^2}{\frac{2n}{3}} \frac{(1+\delta)^2}{(1- \kappa)^2}  \frac{j+1}{\log^{(2)}n}\right] \exp\left[ o(1) \frac{a^2 \phi(n)^2}{\frac{2n}{3}} \frac{j}{\log^{(2)}n} \right]\\
& =\exp\left[- \frac{(a \phi(n))^2}{\frac{2n}{3}} - \frac{(a\phi(n))^2}{\frac{2n}{3}}\left(  1- \frac{1+\delta}{1-\kappa} \right)^2  \frac{ \frac{(j+1)}{\log^{(2)} n}}{1 - \frac{j+1}{\log^{(2)}n}}  + o(1) \frac{(a \phi(n))^2}{ \frac{2n}{3}} \frac{(j+1)}{\log^{(2)}n }\right].
\end{aligned}
\end{equation}
For $j \ge (\log ^{(3)}n)^{3/2}$, we see that the quantity above will be less than $\exp\left[ -\frac{(a \phi(n))^2}{\frac{2n}{3}} - C (\log^{(3)} n)^{3/2}\right]$. Meanwhile, a lower bound for $P(\Lambda_{a,\delta})$ is of the form,
$$
\sum_{j =0}^{\sqrt{\frac{n}{\log^{(2)} n}}} \frac{1}{\sqrt{2 \pi \frac{2n}{3}}} \exp\bigg[ - \frac{(a \phi(n) + j)^2}{\frac{2n}{3}}\bigg] \ge C_1 \exp\left[ - \frac{(a\phi(n))^2}{\frac{2n}{3}} -  C_2\ a\log^{(3)}n  \right],
$$
where $C_1$  and $C_2$ are universal constants.

Thus, we see that $\mathbb{P}(\{ \tau \in J_j \} \cap \Lambda_{a,\delta}) \le C\exp[-(\log^{(3)} n)^{3/2- \epsilon } ] \mathbb{P}(\Lambda_{a,\delta})$. Furthermore, we have at most $\log^{(2)} n $ choices of $J_j$. Applying a union bound would only increase our error estimate by a factor of $\log ^{(2)} n$, which can be absorbed into $\exp[-(\log ^{(3)} n)^{3/2 - \epsilon}]$.

Thus, we have that,
$$
 \sum_{j\le (\log ^{(2)} n) C(\delta,\kappa)}  \mathbb{P}(\{\tau\in J_j \}\cap \Lambda_{a,\delta})  \le \mathbb{P}(\Lambda_{a,\delta}) \exp[-(\log^{(3)} n)^{3/2 - \epsilon}].
$$

The remaining term we have to consider is $\mathbb{P}(\{\tau=r(n)\} \cap \Lambda_{a,\delta})$. In this instance, we separate the walks depending on the value of $S^1_{r(n)}$. We let $I_j:[ \frac{a \phi(n) }{n} \frac{j n}{\log^{(2)} n}, \frac{a \phi(n) }{n} \frac{(j+1) n}{\log^{(2)} n}]$. The probability of the event that at time $\tau=r(n)$, $S^1_\tau$ is at the position $I_j$ is bounded by,
\begin{equation*}
\mathbb{P}\left(S^1_{r(n)} \ge \frac{a \phi(n) }{n} \frac{j n}{\log^{(2)} n}  \right) \mathbb{P}\left( S^1_n - S^1_{r(n)} \ge a \phi(n) - \frac{a \phi(n) }{n} \frac{(j+1) n}{\log^{(2)} n} \right).
\end{equation*}

We thus have that,
\begin{equation*}
\begin{aligned}
&\mathbb{P}\left( \{\tau = r(n)\}  \cap \Lambda_{a,\delta} \right) \le \mathbb{P}\left(S^1_{r(n)} \ge a \phi(n)  \right)  \\&+ \sum_{j = (1- \kappa)^{-1} (\log^{(3)} n)^{3/2}}^{\log^{(2)} n} \mathbb{P}\left(S^1_{r(n)} \ge \frac{a \phi(n) }{n} \frac{j n}{\log^{(2)} n}  \right) \mathbb{P}\left( S^1_n - S^1_{r(n)} \ge a \phi(n) - \frac{a \phi(n) }{n} \frac{(j+1) n}{\log^{(2)} n} \right).
\end{aligned}
\end{equation*}

These estimates are also obtained by the Local Central Limit Theorem, as was done in treating the case $\tau \in J_j$. Thus, we see that $\mathbb{P}\left( \{\tau = r(n)\}  \cap \Lambda_{a,\delta} \right) \le \mathbb{P}(\Lambda_{a,\delta}) \exp[- (\log^{(3)} n)^{3/2 - \epsilon}]$. Again, we can now combine all our estimates to finally derive that $\mathbb{P}( (U^{P,0}_{a})^c \cap \Lambda_{a,\delta}) \le \exp[- (\log^{(3)} n)^{3/2 - \epsilon}] \mathbb{P}(\Lambda_{a,\delta})$. 

 


\end{proof}

Using the previous lemma, we can derive the following conclusion. 
\begin{cor}\label{cor:Inmiddle}
For any time $0\le j \le n$ and any $a,\delta,\kappa$, we have that,
\begin{equation} \label{eq:boundbylines}
\begin{aligned}
\mathbb{P}\bigg( &U^{P,j}_a \cap L^{P,j}_a \cap U^{N,j}_a\cap L^{N,j}_a | \Lambda_{a,\delta} \bigg) = 1 - \zeta_n.
\end{aligned}
\end{equation}
\end{cor}
\begin{proof}
Let $s_i$ be a vector representing the difference between the random walk at time $i+1$ and time $i$. That is, $s_i:= S_{i+1} - S_i$. 
Notice that if $(s_0,s_1,\ldots,s_{n-1})$ is a series of increments for the random walk $S$ that satisfies the event $U^{P,0}_a \cap L^{P,0}_a$, then the increment sequence 
$(s_{n - j }, s_{n-j+1},\ldots, s_{n-1},s_{0},s_{1},s_{2},\ldots,s_{n-j -1})$ for a new walk $\tilde{S}$ will satisfy the event $U^{P,j}_a \cap L^{P,j}_a$. 
Indeed, by computation,  one notices that we have $\tilde{S}(i) = S(n-j + i) - S(n-j)$ for times $0 \le i \le j$, and $\tilde{S}(i) =[S(n) - S(n-j)] + S(i- j)$ for times $j \le i \le n $. Thus, we notice that $\tilde{S}(i) - \tilde{S}(j)= S(i-j)$ for times $j \le i \le n$. Let $\{\tilde{S}^1_n\}$ be the first coordinate of this random walk $\tilde{S}$.  
If we have the upper and lower bounds $\frac{a \phi(n) }{(1+\kappa) n} (i-j) \le S^1(i-j) \le \frac{a (1+\delta) \phi(n) }{(1-\kappa) n}  (i-j)$  for $ r(n) \le i - j \le n $, then $\tilde{S}$ must satisfy the bounds 
$$
\frac{a \phi(n) }{(1+\kappa) n} (i-j)\le \tilde{S}^1(i) - \tilde{S}^1(j)  \le \frac{a (1+\delta) \phi(n) }{(1-\kappa) n}  (i-j),
$$ for times $n \ge i\ge j + r(n)$, as desired for $\tilde{S}$. 

Since this map on increment sequences is a bijection and conserves probability, we see that, 
by Lemma \ref{est:lem1}, 
$$
\begin{aligned}
\mathbb{P}\bigg( &(U^{P,j}_a \cap L^{P,j}_a )^c| \Lambda_{a,\delta}\bigg) \le  \zeta_n.
\end{aligned}
$$

Furthermore, we see that if $(s_{0},s_{1},\ldots,s_{n -1})$ is an increment sequence  for a walk $S$ in $U^{P,j}_a \cap L^{P,j}_a$, then the increment sequence, $( s_{n-1}, s_{n-2},\ldots, s_{j+1},s_j,s_0,s_1,\ldots,s_{j-1})$ for a new walk $\tilde{S}$ 
would satisfy the events $U^{N,n-j}_a$ and $L^{N,n-j}_a$. 
The purpose behind this increment sequence is to ensure that the first $n-j$ steps taken by $\tilde{S}$ are the reverse of the  $n-j$ steps of $S$ taken after time $j$. Thus, we have that $\tilde{S}_{i} - \tilde{S}_{n-j} = -[S_{n- i} - S_{j}]$. When $S$ satisfies,
\begin{equation*}
\frac{a \phi(n) }{(1+\kappa) n} (n- i - j)\le S^1_{n-i} - S^1_{j} \le  \frac{a (1+\delta) \phi(n) }{(1-\kappa) n}  (n- i -j),
\end{equation*} as it would under the event $U^{P,j}_a \cap L^{P,j}_a $, for $ r(n) +j  \le  n-i  \le n$, then we must have that,
\begin{equation*}
\frac{a (1+\delta) \phi(n) }{(1-\kappa) n}  (i-(n-j)) \le \tilde{S}^1_i - \tilde{S}^1_{n-j} \le \frac{a \phi(n) }{(1+\kappa) n} (i- (n-j)),
\end{equation*}
and $ 0 \le i \le n -j- r(n)$, as desired for $\tilde{S}$ to be part of $U^{N,n-j}_a \cap L^{N,n-j}_a$.

Therefore, we see that, by Lemma \ref{est:lem1}, 
$$
\begin{aligned}
\mathbb{P}\bigg( &(U^{N,n-j}_a \cap L^{N,n-j}_a)^c| \Lambda_{a,\delta} \bigg) \le  \zeta_n.
\end{aligned}
$$
By taking a union bound, we have our assertion \eqref{eq:boundbylines}.

\end{proof}

Now that we have bounded the behavior of $S^1_n$ in the first coordinate direction, we also need to assert that the fluctuations of $S$ in the other coordinate directions are relatively small.

\begin{lem} \label{lem:TransverseFluc}
For some small $\theta>0$, 
    $\zeta_n = O(\exp[-(\log^{(3)} n)^{3/2 - \epsilon} ])$, we have that 
    \begin{equation}\label{eq:desiredbnd}
\mathbb{P}\left(\bigcup_{l= r(n)}^{n} \left\{ |S^2_l| \ge \frac{\theta l \phi(n)}{n}\right\} \bigg| \Lambda_{a,\delta} \right) = \zeta_n.
    \end{equation}
\end{lem}
\begin{proof}


First, we make the following remark. Starting from the 3-d random walk, we construct a 1-d random walk by removing any steps that took place in the first or third coordinate direction. Call this modified random walk $\hat{S}$. If by time $n$, the random walk $\hat{S}$ takes $m$ steps in the second coordinate direction, then this will uniquely define $\hat{S}$ until time $m$. After time $m$, one randomly extends $\hat{S}_m$ as a 1-dimensional random walk starting at the point $\hat{S}_m$.

Now, if it were true that $|S^2_l | \ge \frac{\theta l \phi(n)}{n}$ for our 3-d random walk, then there must be some time $k<l$ such that $|\hat{S}_k |\ge \frac{\theta l \phi(n)}{n} \ge \frac{\theta k \phi(n)}{n}$. This is because the step corresponding to $S^2_l$ corresponds to $\hat{S}_k$ for some $k<l$. Furthermore, the walk $\tilde{S}$ is independent of the event $\Lambda_{a,\delta}$ since the steps that our original 3-d walk takes in the other two coordinate directions do not affect the behavior of $S^1$.   We let 
$$M_k :=\left\{ |\hat{S}_k| \ge \frac{\theta k \phi(n)}{n} \right\},$$ 
and
$$
\tilde{M}_{r(n)}:= \left\{ \sup_{1 \le l \le r(n)} |\hat{S}_l| \ge \frac{\theta r(n) \phi(n)}{n} \right \}. 
$$
Notice that with these definitions, we see that,
\begin{equation}
\mathbb{P}\left(\bigcup_{l= r(n)}^{n} \left\{ |S^2_l| \ge \frac{\theta l \phi(n)}{n}\right\} \bigg| \Lambda_{a,\delta} \right) \le \mathbb{P}\left(\bigcup_{l = r(n)+1}^n M_l \cup \tilde{M}_{r(n)}\right).
\end{equation}
To see the above equation, if the point $S^2_l$  where $|S^2_l| \ge \frac{\theta l \phi(n)}{n} $corresponds to a point $\hat{S}_k$ with $k > r(n)$, then we will belong to the event $M_k$. Otherwise, if $l$ corresponds to a point $k$ with $k <r(n)$, then we will belong to the event $\tilde{M}_{r(n)}$.

Now, it suffices to show that 
$\mathbb{P}\left(\bigcup_{l= r(n)}^{n} M_l\right) \le \zeta_n$. 
On any interval $[1,l]$, we can use the reflection principle to say that,
\begin{equation*}
\mathbb{P}\left(\max_{i\in [1,k]} |\hat{S}_i| \ge \lambda \right) 
\le 4 \mathbb{P}( \hat{S}_k  \ge \lambda).
\end{equation*}
The quantity on the left-hand side above can be bounded readily. First observe that,
\begin{equation} \label{eq:maxprinciple}
\mathbb{P}\left( |\hat{S}_k | \ge \frac{\theta k\phi(n)}{n}\right) \le C\exp\left[ - \frac{\theta^2 k^2\phi(n)^2}{2k n^2} \right] \frac{\sqrt{k} n}{\theta \phi(n) k}.
\end{equation}
Here, $C$ is some universal constant.
The right-hand side is larger when $k$ is smaller. The smallest value of $k$ that we need to consider is $k = r(n)$. In this case, the bound is $ C\exp[-\theta^2 ({  \log ^{(3)}} n)^{3/2}/3 ] \frac{1}{\theta (\log^{(3)} n)^{3/4}}$.  This will be an upper bound on the probability of $\mathbb{P}(\tilde{M}_{r(n)})$.

Now, consider dividing the time interval $[0,n]$ dyadically. Namely, we consider the intervals $[n/2,n]$, $[n/4, n/2]$, $[n/8,n/4]$, $\ldots$, $[n2^{-(j+1_)},n2^{-j}]$, where $j$ is the number such that the last interval contains $r(n)$. By substituting the previous maximal principle bound from \eqref{eq:maxprinciple} with $\theta \to \frac{\theta}{2}$ and noting that the right hand side is maximized with $k = r(n)$, we can apply the maximum principle to the interval $[n2^{-(k+1)}, n2^{-k}]$ and obtain, 
\begin{equation*}
\mathbb{P}\left(\max_{[n2^{-(k+1)},n 2^{-k}]} |\hat{S}_i| \ge \frac{\frac{\theta}{2} n 2^{-k}\phi(n)}{n}\right) \le C \exp\left[- \frac{\theta^2(\log^{(3)} n)^{3/2}}{ 12}\right] \frac{1}{\theta(\log^{(3)} n)^{3/4}}.
\end{equation*}
We can thus take the union over all dyadic intervals of interest; the number of this is $\log\frac{n}{r(n)}\le \log^{(3)} n$. By taking a union bound, we get that the probability of $$
\mathbb{P}\left(\bigcup_{l= r(n)}^{n} M_l \right)  \le C(\log^{(3)} n) \exp\left[- \frac{\theta^2(\log^{(3)} n)^{3/2}}{12}\right] \frac{2}{\theta(\log^{(3)} n)^{3/4}} \le \zeta_n,
$$
as desired.
\end{proof}
Similar to how we derived Corollary \ref{cor:Inmiddle} from Lemma \ref{est:lem1}, we can derive the following corollary for fluctuation of $S^2$ in the interior of the random walk from Lemma \ref{lem:TransverseFluc}.

\begin{cor} \label{cor:Bmiddle}
Fix times $j, n$ and some small $\theta>0$. Define the event,
\begin{equation*}
B^{j,n}_{\theta}:= \bigcap_{i=2,3}\left\{ |S^i_{j + l} - S^i_{j}| \le \frac{\theta l\phi(n)}{ n}:  |l| \ge r(n), -j \le  l \le n-j    \right\}.
\end{equation*}
Then,
\begin{equation*}
\mathbb{P}(B_{\theta}^{j,n}|\Lambda_{a,\delta}) \ge 1- \zeta_n,
\end{equation*}
with $\zeta_n = \exp[- (\log^{(3)} n)^{3/2-\epsilon}]$.
\end{cor}
Using these lemmas, we can provide a lower bound for the Green's function that holds with high probability conditioned on the event $\Lambda_{a,\delta}$.
We can leverage this lower bound on the Green's functions to get an upper bound on the capacity. This is the object of the following lemma. Before we proceed, we introduce the notation,

\begin{equation}\label{Ej}
\mathcal{E}_j:= U^{P,j}_a \cap L^{P,j}_a \cap U^{N,j}_a \cap L^{N,j}_a \cap B^{j,n}_{\theta}.
\end{equation}

From Corollaries \ref{cor:Inmiddle} and \ref{cor:Bmiddle}, 
we see that $\mathcal{E}_j$ strongly controls the behavior of the walk around the point $S_j$.

\begin{lem} \label{thm:LowerboundofGreenspec}

Let $j$ satisfy $\min(j, n-j) \ge \frac{n}{(\log^{(2)} n)^{\chi}}$ with $\chi>0$.
Assume that the event $\mathcal{E}_j$ holds. Then, we have that,
\begin{equation*}
\sum_{i \in[1,n]} G(S_i - S_{j}) \ge   2 \frac{(3+ o(1)) n \log^{(3)} n}{2 \pi a\phi(n)} \frac{1-\chi}{\frac{(1+\delta)}{1-\kappa} + \frac{2\theta}{a}}.
\end{equation*}

\end{lem}

\begin{proof}
Notice that when $\|x\|$ is large, we have the asymptotic,
\begin{equation*}
G(x) = \frac{3 + o(1)}{2\pi\|x\|} \ge \frac{3 + o(1)}{2\pi [|x^1| + |x^2| + |x^3|]}
\end{equation*}
(cf. \cite{LA91}). 
From Corollary \ref{cor:Inmiddle}, on the event $\mathcal{E}_j$, we have that,
\begin{equation*}
|S^1_{j + y} - S^1_{j}| \le \frac{a(1+ \delta) \phi(n){|y| }}{(1-\kappa)n}
\end{equation*}
when $|y| \ge  r(n)$. 
Furthermore, from Corollary \ref{cor:Bmiddle}, we also have that on the event $\mathcal{E}_j$,
\begin{equation*}
|S^i_{j+ y} - S^i_j | \le \frac{\theta |y| \phi(n)}{n},
\end{equation*}
for $i=2,3$. 
Under these conditions, we have that,
\begin{equation*}
G(S_{j +y} - S_{j}) \ge \frac{3 + o(1)}{2\pi \left[ \frac{a(1+ \delta)}{1-\kappa} + 2 \theta \right] \frac{y \phi(n)}{n}},
\end{equation*}
as long as $|y| \ge r(n)$. 

Thus, by summation, we derive that, once $\min(j, n-j) \ge \frac{n}{(\log^{(2)} n)^{\chi}}$,
\begin{equation*}
\begin{aligned}
\sum_{i \in [1,n]} G(S_i - S_{j})& \ge 2 \frac{(3+o(1)) n}{2\pi\left[ \frac{a(1+\delta)}{1-\kappa} + 2 \theta\right]\phi(n)} \log \frac{[\log^{(2)} n]^{1-\chi}}{[\log^{(3)} n]^{3/2}}\\& \ge  2 \frac{(3+ o(1)) n \log^{(3)} n}{2 \pi a\phi(n)} \frac{1-\chi}{\frac{(1+\delta)}{1-\kappa} + \frac{2\theta}{a}}.
\end{aligned}
\end{equation*}
This is our desired result. 
\end{proof}

We also have the following lemma regarding the bound for the sum of Green's functions for times $l$ that are merely close to those times $j$ that satisfy the event $\mathcal{E}_j$. 

Namely,
\begin{lem} \label{thm:LowerboundofGreengen}
Fix some constants $3a<b$ and $\delta>0$.
On the time interval $[0, n]$, let $J_i$ denote the time interval $\left[ i \frac{n}{(\log^{(2)} n)^{3}},  (i+1) \frac{n}{(\log^{(2)} n)^{3}}\right]$. Let $\Omega_i$ be the event that 
$$\sup_{k \in J_i}\|S_k -S_{ i \frac{n}{(\log^{(2)} n)^{3}}}\| \ge b \frac{\phi(n)}{(\log^{(2)} n)^{3/2}},$$
then $$
\mathbb{P}\left( \bigcap_{i=0}^{(\log^{(2)} n)^{3}} \Omega_i^c \bigg| \Lambda_{a,\delta}\right)=1-o(1).
$$

Now, let $j$ be a time satisfying the event $\mathcal{E}_j$ and such that $\min(j, n-j) \ge \frac{n}{(\log^{(2)} n)^{\chi}}$. Also, we consider $|y| \le \frac{n}{(\log^{(2)} n)^{3}}$. Then, on the event $\bigcap \Omega_i^c$, we also have the estimate,  
\begin{equation*}
\sum_{k=1}^{n} G(S_k - S_{j +y}) \ge 2 \frac{(3+ o(1)) n \log^{(3)} n}{2 \pi a\phi(n)} \frac{1-\chi}{\frac{(1+\delta)}{1-\kappa} + \frac{2\theta}{a}}.
\end{equation*}
\end{lem}
\begin{proof}

To establish the probability bound on $\Omega_i$, we just have to apply the reflection principle carefully. Notice that if $\sup_{k \in J_i}\|S_k -S_{  i \frac{n}{(\log^{(2)} n)^{ 3}}}\| \ge b \frac{\phi(n)}{(\log^{(2)} n)^{3/2}}$, then at least one coordinate direction $d$ must satisfy,
$\sup_{k\in J_i} \left|S^d_k -S^d_{ i \frac{n}{(\log^{(2)} n)^{3}}}\right| \ge \frac{b}{3}\frac{\phi(n)}{(\log^{(2)} n)^{3/2}} $. We can apply the reflection principle in said coordinate direction to bound this probability on the supremum by,
\begin{equation} \label{eq:boundOmegai}
2\mathbb{P}\left( \left|S^d_{ (i+1) \frac{n}{(\log^{(2)} n)^{3}}} -S^d_{ i \frac{n}{(\log^{(2)} n)^{3}}}\right| \ge \frac{b}{3}\frac{\phi(n)}{(\log^{(2)} n)^{3/2}}\right).
\end{equation}
\eqref{eq:boundOmegai} can be bounded by
\begin{equation*}
O\left( \frac{\sqrt{n}}{\phi(n)}
\exp\left[- \frac{3}{2}\frac{b^2}{9} \frac{(\phi(n))^2}{n}\right] \right).
\end{equation*}
Hence, if we take a union bound, the probability of $\bigcup \Omega_i$ is at most 
\begin{equation} \label{eq:standardevent}
O\left((\log^{(2)} n )^{3} \frac{\sqrt{n}}{\phi(n)}\exp\left[- \frac{3}{2} \frac{b^2}{9} \frac{(\phi(n))^2}{n}\right] \right).
\end{equation}

Now, the probability of the event $\mathbb{P}( \Lambda_{a,\delta})$ is 
\begin{align}\label{eq:standardevent*}
    O\left(\frac{\sqrt{n}}{\phi(n)} \exp\left[ -a^2 \frac{3(\phi(n))^2}{2n}\right]\right).
\end{align}
The term inside the exponential is larger than the order of $\log^{(2)} n$. Any small factor $\exp[-\epsilon \log^{(2)}(n)]$ will be far smaller than the factor $(\log^{(2)} n)^3$. Now, by assumption, we have that $\frac{b^2}{9}> a^2$, so we can reduce $b$ slightly to compensate for the union bound factor of $(\log^{(2)} n)^3$ in equation \eqref{eq:standardevent} and still obtain a term that is significantly smaller than the term in equation \eqref{eq:standardevent*}.  

Now, let us turn to the second part of this lemma and let $j$ be a point satisfying the conditions of Lemma  \ref{thm:LowerboundofGreenspec}. From the proof, we know that,
\begin{equation*}
|S^1_{j+ l} - S^1_{j}| \le \frac{a(1+ \delta) \phi(n) l}{(1-\kappa)n},
\end{equation*}
as long as $l \ge r(n)$.

Now, if $|y| \le \frac{n}{(\log^{(2)} n)^3}$, then $j+y$ and $j$ are either in the same block $J_i$ or adjacent blocks $J_i$ and $J_{i+1}$. Either way, on the event $\bigcap \Omega_i^c,$ one can use the triangle inequality to assert that $\|S_{j} - S_{j +y}\| \le 4b\frac{\phi(n)}{(\log^{(2)} n)^{3/2}} $. 
For $|l| \ge r(n) $, this latter term is much smaller than $\frac{a(1+ \delta) \phi(n)l}{(1-\kappa)n}$. 
Indeed, notice that we have 
$$
\frac{a(1+\delta) \phi(n) l}{(1- \kappa) n} \ge C \frac{\phi(n) (\log^{(3)}n)^{3/2}}{\log^{(2)}n} \gg 4b\frac{\phi(n)}{ (\log^{(2)}n)^{3/2}}.
$$
Thus, for $n$ large, we still have,
$$
|S^1_{j + l} - S^1_{j+y}| \le \frac{a(1+ \delta) \phi(n)l}{(1-\kappa)n}.
$$

Again, by summing up the Green's functions, we obtain the estimate,
\begin{equation*}
\sum_{k=1}^{n} G(S_k - S_{j +y}) \ge 2 \frac{(3+ o(1)) n  \log^{(3)} n}{2 \pi a\phi(n)} \frac{1-\chi}{\frac{(1+\delta)}{1-\kappa} + \frac{2\theta}{a}},
\end{equation*}
as desired.
\end{proof}

By appropriately applying this lower bound on the Green's functions, we can now derive the following upper bound on the capacity of the range of the random walk. 
We can now return to the proof to Proposition \ref{prop:upperboundcap}.



\begin{proof}[Proof of Proposition \ref{prop:upperboundcap}]
Recall the events $\mathcal{E}_j$ in \eqref{Ej}.
By applying Corollary \ref{cor:Inmiddle}, we have that,
\begin{equation*}
\sum_{i=1}^{n}\mathbb{E}[ \mathbbm{1}[\mathcal{E}_i^c]|\Lambda_{a,\delta}] \le n \zeta_n ,
\end{equation*}
with $\zeta_n = O(\exp[-(\log^{(3)} n)^{3/2 - \epsilon}]).$ Now, for any $\epsilon'> \epsilon$, we can use Markov's inequality to deduce that, with probability
$1-o(1)$ conditioned on $\Lambda_{a,\delta}$, there are no more than
$n \exp[-(\log^{(3)} n)^{3/2 - \epsilon'}]$ terms $S_i$ that do not satisfy the event $\mathcal{E}_i$. Since this number is much smaller than $\frac{n}{(\log^{(2)}n)^{3}}$, we notice that any time $S_i$ in the interval $[0,n]$ must be within distance $\frac{n}{(\log^{(2)} n)^3}$ of some time $S_k$ that satisfies the estimate $\mathcal{E}_k$.  (Notice that the only way that there exists a time that is not within distance $\frac{n}{(\log^{(2)} n)^3}$ away from some time that satisfies the event $\mathcal{E}_i$ is if there is an interval of size $\frac{2n}{(\log^{(2)} n)^3}$ of times, all of which do not satisfy $\mathcal{E}_i$. Thus, there must be at least $\frac{2n}{(\log^{(2)}n)^3}$ times that do not satisfy $\mathcal{E}_i$.) Thus, we can apply Lemma \ref{thm:LowerboundofGreengen} to derive the lower bound \begin{equation}\label{eq:uniflowerbnd}\sum_{l= 1}^{n} G(S_i - S_l)\ge 2 \frac{(3+ o(1)) n \log^{(3)} n}{2 \pi a\phi(n)} \frac{1-\chi}{\frac{(1+\delta)}{1-\kappa} + \frac{2\theta}{a}},
\end{equation} 
to the sum of Green's functions for any $i$ satisfying $\min(i, n-i) \ge \frac{n}{(\log^{(2)} n)^{\chi}}$.


We now split the walk $S_{[1,n]}$ into two parts. We consider 
$W:= S\left[1, \frac{n}{(\log^{(2)} n)^{\chi}}\right]  \cup S \left[n -  \frac{n}{(\log^{(2)} n)^{\chi}}, n\right]$ and $V:= S_{[1,n]} \setminus W$. On the set $V$, we have the uniform lower bound on the Green's function \eqref{eq:uniflowerbnd}. 
Thus, we can apply \cite[Lemma 2.4]{DemboOkada} (see also Lemma \ref{lem:uprsumcap}), to deduce that,
\begin{equation} \label{eq:capuppr}
R_n \le \ca(W) + \frac{n}{\min_{x \in V} \sum_{l=1}^{n} G(S_l - x) } \le \ca(W) + a h_3(n) \frac{\frac{1+\delta}{1-\kappa} + \frac{2 \theta}{a}}{1- \chi}. 
\end{equation}

It suffices to bound $\ca(W)$. Now, $S{\left[1, \frac{n}{(\log^{(2)} n)^{\chi}}\right]}$ is a portion of a random walk of length $\frac{n}{(\log^{(2)} n)^{\chi}}$. 
On the level of heuristics, the capacity of this segment of the random walk can be bounded by $O\left(h_3\left( \frac{n}{(\log^{(2)} n)^{\chi}} \right) \right)$. For a formal proof of this assertion, one can refer to the later Lemma \ref{lem:Uuprbnd}. This shows that the probability that the capacity of ${S{\left[ 1, \frac{n}{(\log^{(2)} n)^{\chi}}\right]}}$  is much larger than $b h_3\left(  \frac{n}{(\log^{(2)} n)^{\chi}}\right)$ is less than $C \exp\left[ - b^2 \log^{(2)} \left( \frac{n}{(\log^{(2)} n)^{\chi}} \right) \right]$. By setting $b^2 > a^2$, this probability will be significantly smaller than $\mathbb{P}(\Lambda_{a,\delta}) \ge  \frac{c}{\sqrt{\log^{(2)}n }}\exp[- a^2 \log^{(2)} n] $. In addition, it will also be the case that $b h_3\left(  \frac{n}{(\log^{(2)} n)^{\chi}}\right)$  is much smaller than $a h_3(n) \frac{\frac{1+\delta}{1-\kappa} + \frac{2 \theta}{a}}{1- \chi} $ due to an extra factor of $\sqrt{(\log^{(2)} n)^{\chi}}$ appearing in the denominator of $h_3\left(  \frac{n}{(\log^{(2)} n)^{\chi}}\right).$

By applying this bound to parts of $W$ that are contiguous random walks and substituting this into the inequality \eqref{eq:capuppr}, we see with probability $1-o(1)$ conditioned on the event $\Lambda$ that 
$$
R_n \le  a h_3(n)\left[ \frac{\frac{1+\delta}{1-\kappa} + \frac{2 \theta}{a}}{1- \chi} + \varepsilon\right],
$$
for any arbitrarily small $\epsilon$. 
Since we are allowed to choose $\kappa,\theta,\chi$ and $\epsilon$ as small as we want, the constant on the right hand side above can be replaced with $(1+ \tilde{\delta})$ for any $\tilde{\delta} >\delta,$ as desired.
\end{proof}


\subsection{Deriving Lower Bounds in Proposition \ref{mainpro1}}
In this subsection, we show the lower bound on the capacity for the random walk as in Proposition \ref{mainpro1}. Namely, we establish the following lemma.
Let
\begin{equation*}
\begin{aligned}
\hat{\mathcal{E}}^{\mathcal{A}}_{n,\mathcal{T}} 
= \bigcap_{i=1}^k \{ R_{t_i n} - R_{t_{i-1} n} \le a_i(1-\delta)h_3 ((t_i-t_{i-1}) n)   \}.
\end{aligned}
\end{equation*}

\begin{lem}
Recall the definition $\mathcal{F}^{\mathcal{A},\delta}_{n,\mathcal{T}}$ in \eqref{def:F}. 
For any $k\in N$, $0< a_1 , a_2, ..., a_k \le 1$, $0<t_1\le t_2\le \ldots \le t_k$, 
\begin{align*}
\mathbb{P}(
\hat{\mathcal{E}}^{\mathcal{A}}_{n,\mathcal{T}}| 
\mathcal{F}^{\mathcal{A},\delta}_{n,\mathcal{T}})
=1-o(1).
\end{align*}
\end{lem}

\begin{proof}
For $0\le i\le k-1$, let 
\begin{align*}
&\Lambda_{i,1} := \Big\{j\in [t_in+1,t_{i+1}n]: \sum_{\ell=t_in+1}^{t_{i+1}n} G(S_j-S_\ell) \le (1+ \delta/2) \frac{(t_{i+1} -t_i) n}{a_{i+1}h_3((t_{i+1}-t_i)n)} 
\Big\},  \\
&\Lambda_{i,2} := \Big\{j\in ([1,n]\setminus [t_in+1, t_{i+1}n]): \sum_{\ell \in {[t_i n, t_{i+1} n]}} G(S_j-S_\ell) 
\le C {\frac{\sqrt{n}}{\sqrt{\log^{(2)} n}}}\Big\}
\end{align*}

One can now adapt the bounds we used to establish the lower bound in Lemma \ref{thm:LowerboundofGreenspec} to establish upper bound on the Green's functions. Notice that one can bound $\|S_i - S_j\|$ from below by the lower bounds present in the events $L_a^{P,k}$ and $U_a^{N,k}$ that appear in $\mathcal{E}_k$. Then one can apply Markov's inequality as in the first equation in the proof of  Proposition \ref{prop:upperboundcap} to show that for any $\epsilon >0$ that:
\begin{align*}
    &\mathbb{P}(|\Lambda_{i,1} |\ge (1-\epsilon) (t_{i+1}-t_i)n|\mathcal{F}^{\mathcal{A},\delta}_{n,\mathcal{T}})=1-o(1),\\
     & \mathbb{P}(|\Lambda_{i,2} |\ge (1-\epsilon) (n-(t_{i+1}-t_i)n)|\mathcal{F}^{\mathcal{A},\delta}_{n,\mathcal{T}})=1-o(1).
\end{align*}
We remark that this is the same as the proof as that of \cite[Lemma 3.1]{DemboOkada}. 


Let $\Lambda:=\cap_{i=0}^{k-1} \Lambda_{i,1} \cup \Lambda_{i,2}$, 
$\tilde{\Lambda}_{i,1}:=\Lambda \cap [t_in+1, t_{i+1}n]$, 
$\tilde{\Lambda}_{i,2}:=\Lambda \cap ([1,n]\setminus[t_in+1, t_{i+1}n])$ 
and $S_{\Lambda}:=\{S_m\}_{m\in \Lambda}$. 
We define $\hat{T}_A$ as the hitting time to $A$ for an independent random walk from $(S_n)$. 
By the decomposition on $S_{\Lambda}$, we have that for $0\le i\le k-1$,  for any $\epsilon>0$, 
\begin{align*}
    (1-\epsilon)(t_{i+1}-t_i)n 
    \le  |\tilde{\Lambda}_{i,1}| = &\sum_{j\in \tilde{\Lambda}_{i,1}} 
    \sum_{\ell\in \tilde{\Lambda}_{i,1} }
    G(S_\ell-S_j)\mathbb{P}^{S_j}(\hat{T}_{S_{\Lambda}}=\infty)
    1\{S_j\not\in S_{[1,j-1]\cap \tilde{\Lambda}_{i,1}}\}\\
    + &\sum_{j\in  \tilde{\Lambda}_{i,2}} \sum_{\ell\in \tilde{\Lambda}_{i,1}}
    G(S_\ell-S_j)\mathbb{P}^{S_j}(\hat{T}_{S_{\Lambda}}=\infty)
    1\{S_j\not\in S_{[1,j-1]\cap \tilde{\Lambda}_{i,2}}\}.
\end{align*}

On the event $\mathcal{F}^{\mathcal{A},\delta}_{n,\mathcal{T}}$, for any 
$j\in \tilde{\Lambda}_{1,i}$ and $0\le i \le k-1$,  with high probability, 
\begin{align*}
\sum_{\ell\in \tilde{\Lambda}_{i,1}}
    G(S_\ell-S_j)
\le  \sum_{\ell \in  [t_in+1,t_{i+1}n] }
    G(S_\ell-S_j) \le (1+\delta/2)\frac{(t_{i+1} -t_i) n}{a_{i+1}h_3((t_{i+1}-t_i)n)} 
\end{align*}
and  for any $j\in \tilde{\Lambda}_{2,i}$
\begin{align*}
\sum_{\ell\in \tilde{\Lambda}_{i,1} }
    G(S_\ell-S_j)
\le  \sum_{\ell\in  [t_in+1,t_{i+1}n]}
    G(S_\ell-S_j) \le C { \frac{\sqrt{n}}{\sqrt{\log^{(2)} n}}}
\end{align*}
Then, it yields that with high probability, 
\begin{align*}
    (1-\epsilon)(t_{i+1}-t_i)n \le & \sum_{j\in \tilde{\Lambda}_{i,1} } 
    \mathbb{P}^{S_j}(\hat{T}_{S_\Lambda}=\infty)
    1\{S_j\not\in S_{[1,j-1]\cap \tilde{\Lambda}_{i,1}}\}
    (1+\delta/2)\frac{(t_{i+1} -t_i) n}{a_{i+1}h_3((t_{i+1}-t_i)n)}  \\
    + &C \sum_{j\in \tilde{\Lambda}_{i,2}}
    \mathbb{P}^{S_j}(\hat{T}_{S_\Lambda}=\infty)
    1\{S_j\not\in S_{[1,j-1]\cap \tilde{\Lambda}_{i,2}}\}
     \frac{\sqrt{n}}{\sqrt{\log^{(2)} n}}.
\end{align*}
Since $R_n \le Ch_3(n)$ with high probability by \cite{DemboOkada} (also see Corollary \ref{cor:halfLipschitz}) and $\sum_{j} \mathbb{P}^{S_j}(\hat{T}_{S_\Lambda}=\infty)
1\{S_j\not\in S_{[1,j-1]\cap \Lambda} \}\le \ca(S_\Lambda) \le R_n$ by the definition of the capacity, 
the second term is bounded by 
\begin{align*}
    &C R_n \frac{\sqrt{n}}{\sqrt{\log^{(2)} n}}
    \le C\frac{n}{\log^{(3)} n} .
\end{align*}
Therefore, with high probability,
\begin{align}\label{green1:lower}
    (1-\epsilon)(1+\delta/2)^{-1}a_{i+1} h_3((t_{i+1}-t_i)n)
    \le \sum_{j\in \tilde{\Lambda}_{i,1}} 
    \mathbb{P}^{S_j}(\hat{T}_{S_\Lambda}=\infty)
    1\{S_j\not\in S_{[1,j-1]\cap \tilde{\Lambda}_{i,1}}\}. 
\end{align}

We can sum up this estimate over all $0 \le i \le k-1$ to determine that,
\begin{equation} \label{green2:lower}
R_{t_k n} \ge \sum_{i=0}^{k-1} \sum_{j\in \tilde{\Lambda}_{i,1}} \mathbb{P}^{S_j}(\hat{T}_{S_{\Lambda}} = \infty)
1\{S_j\not\in S_{[1,j-1]\cap \tilde{\Lambda}_{i,1}}\} 
\ge (1-\epsilon)(1+\delta/2)^{-1}\sum_{i=0}^{k-1} a_{i+1} h_3((t_{i+1} - t_{i})n).
\end{equation}

Since $R_{t_{k-1}n}=\sum_{j=1}^{t_{k-1}n} 
    \mathbb{P}^{S_j}(\hat{T}_{S_{[1,t_{k-1}n]}}=\infty)
    1\{S_j\not\in S_{[1,j-1] }\}
    \le \sum_{i=0}^{k-2} \sum_{j=t_in+1}^{t_{i+1}n} 
    \mathbb{P}^{S_j}(\hat{T}_{S_{[t_in+1,t_{i+1}n]}}=\infty)
    1\{S_j\not\in S_{[t_in+1,j-1] }\}
    \le 
    (1+\hat{\delta})\sum_{i=0}^{k-2}
    a_{i+1} h_3((t_{i+1}-t_i)n)$ on $\mathcal{F}^{\mathcal{A},\delta}_{n,\mathcal{T}}$ with probability $1-o(1)$ by Proposition \ref{mainpro1}, 
    combining \eqref{green1:lower} and \eqref{green2:lower}, we have that 
    on $\mathcal{F}^{\mathcal{A},\delta}_{n,\mathcal{T}}$ with probability $1-o(1)$ 
    \begin{equation*}
\begin{aligned}
 a_{i+1}(1-\hat{\delta}) h_3 ((t_{i+1}-t_i) n)
 \le R_{t_{i+1} n} - R_{t_i n}
\end{aligned}
\end{equation*}
as desired. 
\end{proof}

\section{Strassen's LIL for the $\limsup$ of a Random Walk Range: Proofs} \label{sec:Strassenlimsuplil} 

Now that we have established appropriate prerequisite estimates in the previous section, we can now complete the proof of Strassen's LIL. 

\subsection{Elements not in $\mathcal{S}$ are not points of $C(\{f_n\})$ }
Recall our definition $f_n(t) = \frac{R_{nt}}{h_3(n)}$ and the set

\begin{align}\label{def:E1}
 \mathcal{E}^{\mathcal{A},\delta}_{n,\mathcal{T}} 
 = \bigcap_{i=1}^k \{  a_i h_3((t_i-t_{i-1})n)  
 \le R_{t_i n} - R_{t_{i-1} n} 
 \le  a_i(1+\delta)h_3((t_i-t_{i-1})n) \}.
\end{align}

We mention that since $R[t_{i-1}n,t_i n] \ge R_{t_i n} - R_{t_{i-1} n}$,  this is contained in the set,
\begin{equation*}
 \mathcal{U}^{\mathcal{A}}_{n,\mathcal{T}} 
 := \bigcap_{i=1}^k 
 \{ R[t_{i-1}n,t_i n] \ge a_i h_3((t_i-t_{i-1})n)    \}.
\end{equation*}

The following lemma estimates the probability of such events.
\begin{lem} \label{lem:Uuprbnd}
Consider the event $ \mathcal{U}^{\mathcal{A}}_{n,\mathcal{T}}$ as above. For any $\kappa >0$ and n sufficiently large (depending on $\kappa$, $a_1,\ldots,a_k$, $t_1,\ldots,t_k$), we have,
\begin{equation*}
\mathbb{P}( \mathcal{U}^{\mathcal{A}}_{n,\mathcal{T}}) \le  \exp\left[ - \frac{(1-\kappa)^2}{(1+\kappa)^{2}}\sum_{i=1}^k a_i^2  \log^{(2)}((t_i -t_{i-1})n)\right].
\end{equation*}
\end{lem}

As a corollary, we can obtain the following estimate.
\begin{cor} \label{cor:halfLipschitz}
Consider any $\delta>0$. Eventually, the event,
\begin{equation*}
\{|f_n(x) - f_n(y)| \le 8 \sqrt{|x-y|} + \delta, \quad \forall x,y \in [0,1] \},
\end{equation*}
occurs almost surely.
\end{cor}

\begin{proof}[Proof of Lemma  \ref{lem:Uuprbnd}]
We first define $$
b^i_n:= 2 \eta  \log^{(2)} ((t_i - t_{i-1}) n), 
\quad b_n:=b^1_n, \quad
r^i_n := \frac{a_i \phi((t_i - t_{i-1})n)}{b^i_n}, \quad r_n:=r^1_n.
$$
$\eta$ is a parameter that will be chosen later. 
We define the event $C_i$ to be the event,
$$
\begin{aligned}
C_i&:= \{ S_{[t_{i-1}n,t_i n]}  \text{ is not contained in the union of at most } b^i_n   \text{ balls of radius }  r^i_n,\\
&\text{ with the centers of the balls at most distance } r^i_n \text{ from each other.} \}.
\end{aligned}
$$
From \cite[Lemma 2.1]{DemboOkada}, we see that the event $ \mathcal{U}^{\mathcal{A}}_{n,\mathcal{T}}$ is  contained in the event,
\begin{equation*}
\hat{\mathcal{U}}^{\mathcal{A}}_{n,\mathcal{T}} :=\bigcap_{i=1}^k C_i.
\end{equation*}
Notice that the events $C_i$ are all independent of each other, and thus,
\begin{equation*}
\mathbb{P}(\hat{\mathcal{U}}^{\mathcal{A}}_{n,\mathcal{T}}) = \prod_{i=1}^k \mathbb{P}(C_i).
\end{equation*}

Now, it suffices to find a bound on the probability of each individual event $C_i$. The bound here is similar to the proof of \cite[Lemma 3.4]{DemboOkada}, but we will give the appropriate modifications. For simplicity in notation, we consider only the event $C_1.$ 

We define the stopping time $T_i$ as follows:
\begin{equation*}
T_i:= \inf\{ t> T_{i-1}: \|S_t- S_{T_{i-1}}\| \ge r_n-1 \}.
\end{equation*}
This is the first exit time of the ball of radius $r_n -1$ centered at $S_{T_{i-1}}$. 
We see that $C$ is included in the event that $T_{b_n} \le  t_1  n $. Furthermore, we also notice that $T_i - T_{i-1}$ form a collection of independent identically distributed random variables. 
Fix some $\kappa>0$ (set to be small), and let $\overline{T}_1$ be the first exit time of the Brownian motion for the ball of radius $\frac{r_n}{\sqrt{1+\kappa}}$ and $\hat{T}_1$ be the first exit time of a ball of radius 1. 
Via the Skorokhod coupling with a Brownian motion, one can assert that there exist some constants $c$ and $\epsilon$ (depending on $\kappa$) such that for $n$ large enough, we have,
\begin{equation} \label{eq:comparT}
\mathbb{P}(T_1 \le u) \le c\exp[- n^{\epsilon}] + \mathbb{P}(3\overline{T}_1 \le u) \le c \exp[-n^{\epsilon}] + \mathbb{P}\left(\frac{3 r_n^2\hat{T}_1}{1+\kappa} \le u\right).
\end{equation}
Let $\hat{T}'_1,\ldots,\hat{T}'_{b_n}$ be a collection of independent random variables with the same distribution as $\hat{T}_1$. Equation \eqref{eq:comparT} allows us to say that we can find a coupling of the variables $\frac{3 r_n^2\hat{T}'_i}{1+\kappa}$ and $T_{i}-T_{i-1}$ such that, except for an event of probability $c \exp[-n^{\epsilon}]$, we have that $T_i -T_{i-1} \ge\frac{3  r_n^2\hat{T}'_i}{1+\kappa}. $ By a union bound, we see that, aside from an event of probability at most $cb_n \exp[-n^{\epsilon}] $, we have that 
$$
T_{b_n}= \sum_{i=1}^{b_n} (T_i - T_{i-1}) \ge \sum_{i=1}^{b_n} \frac{3 r_n^2 \hat{T}'_i}{1+\kappa}.
$$ 
By combining these estimates, we see that we have,
\begin{equation*}
\begin{aligned}
\mathbb{P}(T_{b_n} \le t_1 n) &\le c b_n \exp[-n^{\epsilon}] + \mathbb{P}\left( \frac{1}{b_n} \sum_{i=1}^{b_n} \hat{T}_i' \le \frac{t_1 n(1+\kappa)}{3 b_n r_n^2} \right) \\
&= c b_n \exp[-n^{\epsilon}]  +  \mathbb{P} \left(  \frac{1}{b_n} \sum_{i=1}^{b_n} \hat{T}_i' \le \frac{\eta(1+\kappa)}{a_1^2}\right).
\end{aligned}
\end{equation*}
We can apply the estimate \cite[equation (3.20)]{DemboOkada} to assert that, for any $\lambda>0$, 
\begin{equation*}
\mathbb{P}\left( \frac{1}{b_n} \sum_{i=1}^{b_n} \hat{T}'_i \le\frac{\eta(1+\kappa)}{a_1^2} \right) \le (c_{\kappa} \exp[\lambda^2 (1+\kappa) \eta/(2a_1^2)  - \lambda(1-\kappa)])^{b_n}.
\end{equation*}
If we optimize over the choice of $\lambda$, we would set $\lambda =\frac{1-\kappa}{1+\kappa}\frac{a_1^2}{\eta}$. Thus, we see that,
\begin{equation*}
\begin{aligned}
& (c_{\kappa} \exp[\lambda^2 (1+\kappa) \eta/(2a_1^2)  - \lambda(1-\kappa)])^{b_n} \le (c_{\kappa} \exp[ -(1-\kappa)^2 a_1^2/(2 \eta(1+\kappa)) ])^{b_n}\\
& \le \exp[ - b_n(1-\kappa)^2 a_1^2/(2 \eta(1+\kappa)^2) ] ,
\end{aligned}
\end{equation*}
provided $\eta$ is sufficiently small (depending only on $\kappa$ and $a_i$). 
Indeed, we see that we want $\eta \le \frac{a_i^2 (1-\kappa)^2 \kappa}{2(1+\kappa)^2 \log c_{\kappa}}$.

We now set $\eta$ sufficiently small. Upon substituting $b_n = 2 \eta \log^{(2)} (t_1 n)$, we see that,
for $n$ large enough depending on $a_1$ and $\kappa$, we have
\begin{equation*}
\mathbb{P}(C_1) \le \exp[- a_1^2 (1-\kappa)^2/(1+\kappa)^2 \log^{(2)}(t_1 n)].
\end{equation*}
Thus, for any $\kappa>0$, we eventually have that,
\begin{equation*}
\mathbb{P}( \mathcal{U}^{\mathcal{A}}_{n,\mathcal{T}}) 
\le \mathbb{P}(\hat{\mathcal{U}}^{\mathcal{A}}_{n,\mathcal{T}})
\le \prod_{i=1}^k \exp\left[ - \frac{(1-\kappa)^2}{(1+\kappa)^{2}}\sum_{i=1}^k a_i^2 \log^{(2)}((t_i - t_{i-1})n) \right].
\end{equation*}
This is our desired result.


 

\end{proof}

\begin{proof}[Proof of Corollary \ref{cor:halfLipschitz} ]

Let $4>q>1$ and suppose that we have shown that for the collection of times $t_m =q^m$, that, almost surely, 
\begin{equation} \label{eq:induceq}
|f_{t_m}(x) - f_{t_m}(y)| \le 2 \sqrt{x-y} + \frac{\delta}{4}, \quad \forall x,y \in [0,1],\quad  x \ge y
\end{equation}
once $k$ is sufficiently large.

Now, suppose that $n$ is a time satisfying $q^m \le n \le q^{m+1}$. Then, we have that,
$$
\begin{aligned}
|f_n(x) - f_n(y)|& = |f_{q^{m+1}}(n q^{-(m+1)}x) - f_{q^{m+1}}(n q^{-(m+1)}y)| \frac{\sqrt{q^{m+1} \log^{(2)} q^{m+1}} \log^{(3)} n}{\sqrt{n \log^{(2)} n} \log^{(3)} q^{m+1}}\\&
\le 2 \sqrt{q}|f_{q^{m+1}}(n q^{-(m+1)}x) - f_{q^{m+1}}(n q^{-(m+1)}y)| \le 2\sqrt{q}[2 \sqrt{nq^{-(m+1)}}\sqrt{x-y} + \delta/4]\\
& \le 4\sqrt{q}\sqrt{x-y} +  \sqrt{q}\delta/2  \le 8 \sqrt{x-y} + \delta.
\end{aligned}
$$

To derive the first inequality, we used the fact that $$\log^{(2)} (q^{m+1}) = \log^{(2)} q + \log(m+1) \le 2 \log^{(2)} q + 2 \log m = 2 \log^{(2)} q^m \le 2 \log^{(2)} n,$$ $\log^{(3)} n \le \log^{(3)} q^{m+1}$ and $\sqrt{q^{m+1}} \le \sqrt{q}n.$ 
To get the second inequality, we used  the assumption found in \eqref{eq:induceq}. 

Now, it suffices to establish the inequalities shown in equation \eqref{eq:induceq}. Let $K$ be large enough so that $\sqrt{\frac{8}{K}} < \frac{\delta}{4}$. If we then show that 
\begin{equation}\label{eq:discretftm}
W_m:=\{|f_{t_m}(i K^{-1}) -f_{t_m}(j K^{-1})| \le 2 \sqrt{|i K^{-1}- j K^{-1}|},\quad  \forall i,j\in \{1,\ldots,K\}\}
\end{equation}
holds for all sufficiently large $m$ almost surely, 
then
we indeed have that \eqref{eq:induceq}. 
Indeed, let $i$ be the smallest integer such that $i/K \ge x$ and $j$ be the largest integer such that $y \ge j/K$. We have that $|i K^{-1} - j K^{-1}| \le |x-y| + \frac{2}{K}$ and, thus, by the monotonicity of $f_{t_m}$, we have,
$$
\begin{aligned}
|f_{t_m}(x) - f_{t_m}(y)|&\le |f_{t_m}(i K^{-1}) - f_{t_m}(j K^{-1})| \le 2\sqrt{|iK^{-1} - j K^{-1}| }\\
&\le 2\sqrt{x-y+ \frac{2}{K}} 
\le 2 \sqrt{x-y} + \frac{\frac{4}{K}}{\sqrt{x-y} + 
\sqrt{x-y+ \frac{2}{K}}}\\
& \le 2 \sqrt{x-y} + \sqrt{\frac{8}{K}} \le 2 \sqrt{x-y} + \frac{\delta}{4}.
\end{aligned}
$$

At this point, it suffices to establish equation \eqref{eq:discretftm}. We see that  $|f_{t_m}(i K^{-1}) - f_{t_m}(j K^{-1})| \ge 2 \sqrt{{|i K^{-1}-j K^{-1}|}}$ is included in the event $\mathcal{U}^{0,2}_{t_m, jK^{-1},iK^{-1} }.$ By Lemma \ref{lem:Uuprbnd}, we see that the probability of this event is less than $\exp[- 3 \log^{(2)} t_m]$ when $m$ is sufficiently large (we set $\kappa$ small enough so that $4 \frac{(1-\kappa)^2}{(1+\kappa)^2}<3)$. 
We have that
$$
\mathbb{P}(W_m^c)
\le  \mathbb{P}(\cup_{i,j \in \{1,\ldots,K\}}\mathcal{U}^{0,2}_{t_m, jK^{-1},iK^{-1}})\le\sum_{m=0}^{\infty} K^2 \exp[- 3 \log^{(2)} q^m] < \infty,
$$
so by the Borel-Cantelli lemma, for $m$ large enough,  $W_m^c$ will eventually stop occurring. This establishes equation \eqref{eq:discretftm}.

\end{proof}

Our final theorem for this section will establish that any function that is not found in $\mathcal{S}$ cannot be a limit point using Corollary \ref{cor:halfLipschitz} and Lemma \ref{lem:Uuprbnd}.  

\begin{thm}
Let $F$ be any function $\not \in \mathcal{S}$. Namely, it satisfies
\begin{equation*}
\int_{0}^1 (F'(t))^2 \text{d}t >1,
\end{equation*}
or it is either discontinuous or not monotone.
Define $\mathcal{N}^{f}_{\epsilon}$ to be,
$$
\mathcal{N}^f_{\epsilon}:= \{g:[0,1]\to [0,\infty), \|g-f\|_{\infty}< \epsilon \}. 
$$

There exists some $\epsilon>0$ such that the event $f_n \not \in \mathcal{N}^{F}_{\epsilon}$ will occur a.s. for all sufficiently large $n$. 
\end{thm}

\begin{proof}
By applying Corollary \ref{cor:halfLipschitz}, we can quickly see that the function $F$ must be continuous and monotone increasing if we want $f_n \in \mathcal{N}^F_{\epsilon}$ for some appropriate $\epsilon$. Since $f_n$ is monotone increasing, the supposed limit $F$ must be monotone increasing as well. (Indeed, if there were $x$,$y$ such that $x>y$ but $F(x) <F(y)$, we could let $\epsilon = \frac{F(x)-F(y)}{3}$, and it would be trivial that no $f_n \in \mathcal{N}^{F}_{\epsilon}.$) Furthermore, we may also assume that $F$ is continuous. (We show it by contradiction. If there were a point $p$ such that  $\lim_{x \to p^{+}} F(x) \ne \lim_{x \to p^-} F(x) $, we could let $\epsilon =\frac{\lim_{x \to p^{+}} F(x)- \lim_{x \to p^-} F(x)}{3}. $ The continuity properties detailed in Corollary \ref{cor:halfLipschitz} ensure that, eventually, no function $f_n(x)$ could lie in $\mathcal{N}^{F}_{\epsilon}$. (If $f_n(x) \in \mathcal{N}^{F}_{\epsilon}$, then 
$$|\lim_{x \to p^{+}}f_n(x) - \lim_{x \to p^-}f_n(x)| \le |\lim_{x \to p^{+}}F(x) - \lim_{x \to p^{-}}F(x)| - 2 \epsilon/3 >\epsilon/3. $$
In Corollary \ref{cor:halfLipschitz}, we could set $\delta<\epsilon/3$ and arrive at a contradiction.

At this point, we only need to demonstrate that the function $F$ must be absolutely continuous and, provided it is absolutely continuous, we have $\int_0^1 (F'(t))^2 \text{d}t \le  1$. The proof that $F$ must be absolutely continuous is very similar to the proof of $\int_0^1 (F'(t))^2 \text{d}t \le  1.$ We will present the details for the proof that  $\int_0^1 (F'(t))^2 \text{d}t \le  1,$ and then present the changes necessary to derive that $F$ must be absolutely continuous.

Set $q>1$ and let $\tau_m = q^m$. 
We will first argue that, almost surely, for sufficiently large $m$, it must be the case that $f_{\tau_m} \not \in \mathcal{N}_{\epsilon}^F$. We first remark that if $\int_0^1 (F'(t))^2 \text{d} t>1$, there must be some discretization and  $\kappa >0,\epsilon>0$ such that,
\begin{equation} \label{eq:consfintge1}
\sum_{i=1}^K \frac{(1-\kappa)^2}{(1+\kappa)^2}(t_{i} - t_{i-1}) \left(\frac{\max(F(t_i) - F(t_{i-1}) -\epsilon,0)}{t_i - t_{i-1}} \right)^2 >1.
\end{equation}
We remark first that if it were the case that $f_n \in \mathcal{N}^{F}_{\epsilon}$, then it necessarily must be the case that,
\begin{equation*}
f_n(t_i) - f_n(t_{i-1}) \ge F(t_i) - F(t_{i-1}) -  2\epsilon.
\end{equation*}
As such, it must be the case that,
\begin{equation} \label{eq:Diffca}
\begin{aligned}
R[1,t_i \tau_m] - R[1,t_{i-1}\tau_m] &\ge [F(t_i) - F(t_{i-1}) - 2\epsilon] h_3(\tau_m)\\& \ge \frac{F(t_i) - F(t_{i-1}) - 2 \epsilon}{\sqrt{t_i - t_{i-1}}}
h_3((t_i - t_{i-1})\tau_m) .
\end{aligned}
\end{equation}
The last inequality holds for $m$ sufficiently large (where we use the fact that $\frac{\sqrt{\log^{(2)} x}}{\log^{(3)} x}$ is increasing when $x$ is sufficiently large).

We now define $a_i$ to be $\frac{\max(F(t_i) - F(t_{i-1}) - 2 \epsilon,0)}{\sqrt{t_i - t_{i-1}}}$ and consider the event $\mathcal{U}^{\mathcal{A}}_{\tau_m,\mathcal{T}}$. Due to equation \eqref{eq:Diffca}, $f_{\tau_m} \in \mathcal{N}^f_{\epsilon}$ would imply that $f_{\tau_m} \in \mathcal{U}^{\mathcal{A}}_{\tau_m,\mathcal{T}}$. 
From Lemma \ref{lem:Uuprbnd}, we see that,
$$
\mathbb{P}(\mathcal{U}^{\mathcal{A}}_{\tau_m,\mathcal{T}}) \le \exp\left[- \sum_{i=1}^k \frac{(1-\kappa)^2}{(1+\kappa)^2} a_i^2 \log^{(2)}((t_i -t_{i-1})\tau_m) \right].
$$
By our definition of $a_i$, we see that $\sum_{i=1}^k \frac{(1-\kappa)^2}{(1+\kappa)^2}a_i^2 >1$.  Thus, we would have that 
$$
\sum_{i=1}^{\infty } \exp\left[- \sum_{i=1}^k \frac{(1-\kappa)^2}{(1+\kappa)^2} a_i^2 \log^{(2)}((t_i -t_{i-1})\tau_m) \right] < \infty.
$$
By the Borel-Cantelli lemma, we can ensure that, eventually, the events $\mathcal{U}^{\mathcal{A}}_{\tau_m,\mathcal{T}}$ will no longer occur.

Now, we will deal with general times $n$.  Note that since $F$ is a continuous function, by appealing to uniform continuity, there exists some $q>1$ such that $|F(t) - F(yt)| < \epsilon/4$ for all $t\in [0,1]$ and all $y \in [1,q]$. In addition, we remark that there is $q>1$ sufficiently close to $1$ such that,
\begin{equation} \label{def:Rntiln}
L(n,\tilde{n}):=1 - \frac{\frac{\sqrt{n \log^{(2)} n}}{\log^{(3)} n}}{\frac{\sqrt{\tilde{n} \log^{(2)} \tilde{n}}}{\log^{(3)} \tilde{n}}} ,
\end{equation}
satisfies $|L(n,\tilde{n})|< \epsilon/8$, whenever $n q^{-1} \le\tilde{n}< q n$. 
Furthermore, by \cite[Theorem 1.1]{DemboOkada} on the $\limsup$, we may assume that we have the bound $f_n(t) \le 2$ for any $t\in[0,1]$. 

For a general time $n$, let $\tau_m$ correspond to the largest $m$ such that $\tau_m <n$. We see that we have, for $q$ sufficiently close to $1$,
\begin{equation*}
\begin{aligned}
|f_n(t) - F(t)| &= \left|f_{\tau_m}\left(\frac{n}{\tau_m}t\right) + f_{n}(t)L(n,\tau_m) - F\left(\frac{n}{\tau_m}t\right) +\left(F\left(\frac{n}{\tau_m}t\right) - F(t)\right)\right|\\
& \ge \left|f_{\tau_m}\left(\frac{n}{\tau_m}t\right) - F\left(\frac{n}{\tau_m}t\right) \right| -f_n(t)|L(n,\tau_m)| - \left|F\left(\frac{n}{\tau_m}t\right) - F(t) \right|\\
& \ge \left|f_{\tau_m}\left(\frac{n}{\tau_m}t\right) - F\left(\frac{n}{\tau_m}t\right) \right| - \frac{\epsilon}{2}.
\end{aligned}
\end{equation*}

Here, we used the fact that $\frac{n}{\tau_m}<q$ , so that we could apply the estimates in equation \eqref{def:Rntiln} and the line that precedes it.  Thus, we see that if, for some $n$, we have $f_n \in \mathcal{N}^F_{\epsilon/2}$, then we would have $f_{\tau_m} \in \mathcal{N}^{F}_{\epsilon}$.

We now mention some of the modifications necessary to show that $F$ must be an absolutely continuous function. Notice that if $F$ is not absolutely continuous, then we must be able to find some $\delta >0$ such that for any $\epsilon$, there will be a union of disjoint intervals $\cup_{i=1}^k [x_i, y_i]$ such that we have,
\begin{equation*}
\begin{aligned}
& \delta \le \sum_{i=1}^k |F(x_i) - F(y_i)| \\
& \epsilon \ge \sum_{i=1}^k |x_i - y_i|.
\end{aligned}
\end{equation*}
By the Cauchy-Schwartz inequality, we see
\begin{equation*}
\delta^2 \le (\sum_{i=1}^k |F(x_i) - F(y_i)|)^2 \le \left(\sum_{i=1}^k \frac{|F(x_i) - F(y_i)|^2}{|x_i - y_i|} \right) (\sum_{i=1}^k |x_i - y_i|).
\end{equation*}
Thus, we have that,
\begin{equation*}
\frac{\delta^2}{\epsilon} \le \sum_{i=1}^k |x_i -y_i| \frac{(F(x_i) - F(y_i))^2}{(x_i - y_i)^2}.
\end{equation*}

By taking $\epsilon$ sufficiently small, the right-hand side will become greater than 1. This will now resemble the condition found in equation \eqref{eq:consfintge1} and we can proceed along the same argument to derive a contradiction.
\end{proof}

\subsection{Elements in $\mathcal{S}$ are points in $C(\{f_n\})$}

In this subsection, we will establish the following main theorem using Corollary \ref{cor:halfLipschitz} and Proposition \ref{mainpro1}.

\begin{thm} \label{thm:neinfoft}
Let $F \in \mathcal{S}$. Thus, $F$ is a continuous monotone increasing function such that  $\int_{0}^1 (F'(t))^2 \text{d}t \le 1$. Recall the neighborhood,

\begin{equation*}
\mathcal{N}^{F}_{\epsilon}:= \{g:[0,1]\to [0,\infty), \|g- F\|_{\infty} \le \epsilon\}.
\end{equation*}
For any $\epsilon>0$, it must be the case that $f_n \in \mathcal{N}^F_{\epsilon}$ infinitely often.

\end{thm}

\begin{proof}
We first consider the sequence of times $\tau_m = m^m$ and argue that if it is the case that,
$$
x_{\tau_m}(t):= 
\frac{R[\tau_{m-1}, \tau_{m-1}+(\tau_m - \tau_{m-1})t]}{
h_3(\tau_{m}-\tau_{m-1}) },
$$
belongs to $\mathcal{N}^{F}_{\epsilon}$
then it must also be the case that $f_{\tau_m}$ belongs to $\mathcal{N}^{F}_{2 \epsilon}$ for $m$ sufficiently large. 
We define,
\begin{equation*}
y_m(t):=\frac{R_{\tau_m t} }{h_3(\tau_{m}-\tau_{m-1})}-x_{\tau_m}(t).
\end{equation*}
First, notice that, 
\begin{equation*}
|f_{\tau_m}(t) - F(t)| = |x_{\tau_m}(t) -F(t) + y_m(t) + f_{\tau_m}(t) L(\tau_m, \tau_m - \tau_{m-1}) |,
\end{equation*}
where $L(\cdot,\cdot)$ is from equation \eqref{def:Rntiln} and
\begin{equation*}
\begin{aligned}
&|y_m(t)| = \frac{|(-R_{[\tau_{m-1},\tau_m t + (1-t) \tau_{m-1}]} + R_{\tau_{m}t +(1-t) \tau_{m-1}}) +(R_{\tau_m t} - R_{\tau_m t +(1-t) \tau_{m-1}}) |}{h_3(\tau_m - \tau_{m-1})}\\
& \le \frac{|-R_{[\tau_{m-1},\tau_m t + (1-t) \tau_{m-1}]} + R_{\tau_{m}t +(1-t) \tau_{m-1}}|}{h_3(\tau_m - \tau_{m-1})} + \frac{|R_{\tau_m t} - R_{\tau_m t +(1-t) \tau_{m-1}} |}{h_3(\tau_m - \tau_{m-1})}\\
& \le \frac{R_{\tau_{m-1}}}{h_3(\tau_m)} \frac{h_3(\tau_m)}{h_3(\tau_m - \tau_{m-1})} + \frac{|R_{\tau_mt} - R_{\tau_m t + (1-t) \tau_{m-1}}|}{h_3(\tau_m)} \frac{h_3(\tau_m)}{h_3(\tau_m - \tau_{m-1})}.
\end{aligned}
\end{equation*}
By using Corollary \ref{cor:halfLipschitz}, we see that eventually, almost surely, we see that the ratios $\frac{R_{\tau_{m-1}}}{h_3(\tau_m)}$ and $\frac{|R_{\tau_m t} - R_{\tau_m t +(1-t) \tau_{m-1}}|}{h_3(\tau_m)}$ can be made arbitrarily small. (Here, we remark that the ratio $\frac{R_{\tau_m t}}{h_3(\tau_m)} = f_{\tau_m}(t)$ ,$\frac{R_{\tau_m t +(1-t) \tau_{m-1}}}{h_3(\tau_m)} = f_{\tau_m}(t + (1-t) \frac{\tau_{m-1}}{\tau_m}) $, and $ \frac{R_{\tau_{m-1}}}{h_3(\tau_m)} = f_{\tau_m}(\frac{\tau_{m-1}}{\tau_m})$.) Furthermore, the ratio $\frac{h_3(\tau_m)}{h_3(\tau_m -\tau_{m-1})}$ will go to $1$ as $m$ is taken to $\infty$. Here, we use the fact that $\lim_{m \to \infty} \frac{\tau_m}{\tau_m - \tau_{m-1}} = 1$. Furthermore, Corollary \ref{cor:halfLipschitz} would also tell us that almost surely for $m$ large enough, we have that $f_{\tau_m}(t) \le 2.$ 

In addition, using the logic following equation \eqref{def:Rntiln}, we see furthermore that $L(\tau_m,\tau_m - \tau_{m-1}) \to 0$. Thus, for $m$ sufficiently large, we have both,
$$
|y_m(t)| \le \frac{\epsilon}{2},\quad |f_{\tau_m}(t)L(\tau_m,\tau_m - \tau_{m-1})| \le \frac{\epsilon}{2}.
$$
Hence, we have for $m$ large enough that,
$$
|x_{\tau_m}(t) -F(t) + y_m(t) + f_{\tau_m}(t) L(\tau_m, \tau_m - \tau_{m-1}) |\le \epsilon + \frac{\epsilon}{2} + \frac{\epsilon}{2} \le 2 \epsilon.
$$

At this point, we will try to analyze the properties of the function $x_{\tau_m}$ instead of $f_{\tau_m}$. We remark that due to the assumption that $\int_{0}^1 (F'(t))^2 \text{d}t \le 1$, for any interval $I=[c,d]$ of size $\ell$, we have that 
$$
|I| 
\ge |I|\int_{I} (F'(t))^2 \text{d}t 
\ge \left(\int_{I} F'(t) \text{d}t\right)^2,
$$
and, thus,
\begin{equation}\label{eq:diffF}
F(d) - F(c) = \int_{I} F'(t)\text{d}t \le \sqrt{|I|}.
\end{equation}
We choose $K$ large enough such that $\frac{1}{K} \le \frac{\epsilon^2}{64^2}$ and consider the sequence of times $t_i = \frac{i}{K}$. We consider
$$
a_i = \frac{F(t_{i}) - F(t_{i-1}) - \frac{\epsilon}{4K}}{\sqrt{t_i -t_{i-1}}}.
$$
Notice that $a_i \le 1$ by equation \eqref{eq:diffF}. 

On the event 
$\mathcal{E}^{\mathcal{A},\frac{\epsilon}{8K}}_{\tau_m-\tau_{m-1},\mathcal{T}}(S_{[\tau_{m-1},\tau_m]})$ in \eqref{def:E1},  we have that, 
\begin{equation*}
\begin{aligned}
& R[\tau_{m-1},(K-J)K^{-1}\tau_{m-1} + J K^{-1}\tau_m]\\
= &\sum_{i=1}^J 
R[\tau_{m-1},(K-i)K^{-1}\tau_{m-1} + i K^{-1}\tau_m] - R[\tau_{m-1},(K-i-1)K^{-1}\tau_{m-1} + (i-1) K^{-1}\tau_m] \\
\le &  \sum_{i=1}^J a_i \left(1 + \frac{\epsilon}{8K}\right) 
h_3((t_i - t_{i-1}) (\tau_m - \tau_{m-1}))
\le  \sum_{i=1}^J h_3( \tau_m-\tau_{m-1}) (F(t_i) - F(t_{i-1}))
=  F(t_J) h_3(n),
\end{aligned}
\end{equation*}
for $m$ sufficiently large, as well as,

\begin{equation*}
\begin{aligned}
R[\tau_{m-1},(K-J)K^{-1}\tau_{m-1} + J K^{-1}\tau_m]
& \ge \sum_{i=1}^J  a_i \left(1- \frac{\epsilon}{8K}\right) h_3((t_i - t_{i-1})(\tau_m -\tau_{m-1})) \\
& \ge \sum_{i=1}^J h_3(\tau_m - \tau_{m-1}) \left(F(t_i) - F(t_{i-1}) - \frac{\epsilon}{2K} \right)\\
&=h_3(\tau_m - \tau_{m-1}) \left(F(t_J) - \frac{J \epsilon}{2K} \right),
\end{aligned}
\end{equation*}
again, for $m$ sufficiently large.

We see that along the times $t_i$, we must have that $|x_{\tau_m}(t_i) - F(t_i)| \le \frac{\epsilon}{2}$. 
For any time $t$ between $t_i < t <t_{i-1}$, the control from equation \eqref{eq:diffF} shows that,
$|F(t) - F(t_i)| \le \frac{\epsilon}{64}$. Furthermore, Corollary  \ref{cor:halfLipschitz} with $\delta = \epsilon/8$ shows that $|x_{\tau_{m}}(t) - x_{\tau_m}(t_{i-1})| \le 8 \epsilon/64+ \epsilon/8 = \epsilon/4.$ 
Thus, we must have that $|x_{\tau_m}(t) - F(t)|\le \epsilon$ provided the event $\mathcal{E}^{\mathcal{A},\frac{\epsilon}{8K}}_{\tau_m-\tau_{m-1},\mathcal{T}}(S_{[\tau_{m-1},\tau_m]})$ holds. 


By the Skorokhod Embedding (see Lemma 3.1 in \cite{LA91}), the probability of  $\mathcal{F}^{\mathcal{A},\frac{\epsilon}{8K}}_{\tau_m-\tau_{m-1},\mathcal{T}}(S_{[\tau_{m-1},\tau_m]})$ is

$$
\begin{aligned}
\exp\left[ -\sum_{i=1}^{K}  a_i^2 \log^{(2)}( (\tau_m -\tau_{m-1})K^{-1}) \right] &= \exp \left[ -\log^{(2)}((\tau_m - \tau_{m-1}) K^{-1}) \sum_{i=1}^{K} \frac{(F(t_i)- F(t_{i-1}) - \epsilon/(2K))^2}{t_i - t_{i-1}}  \right]\\& > \exp\left[-\log^{(2)}( m^m K^{-1}) (1-\chi) \int_0^1 (F'(t))^2 \text{d}t\right] > \frac{C}{m^{1-\chi}},
\end{aligned}
$$
where $\chi$ is some constant $>0$ and $C$ is some constant$>0$. 
Thus, by Proposition \ref{mainpro1}, 
\begin{align*}
\sum_{m=1}^{\infty}\mathbb{P}
(|x_{\tau_m}(t) - F(t)|\le \epsilon)
\ge &\sum_{m=1}^{\infty}\mathbb{P}
( \mathcal{E}^{\mathcal{A},\frac{\epsilon}{8K}}_{\tau_m-\tau_{m-1},\mathcal{T} }(S_{[\tau_{m-1},\tau_m]}))\\
\ge & c 
\sum_{m=1}^{\infty}\mathbb{P}(\mathcal{F}^{\mathcal{A},\frac{\epsilon}{8K}}_{\tau_m-\tau_{m-1},\mathcal{T}}(S_{[\tau_{m-1},\tau_m]}))=\infty.
\end{align*}
Noting that the events $\{|x_{\tau_m}(t) - F(t)|\le \epsilon\}$ are independent for different $m$, by the second Borel-Cantelli lemma, the event $\{|x_{\tau_m}(t) - F(t)|\le \epsilon\}$ will occur infinitely often in $m$ almost surely, that is, $f_n \in \mathcal{N}_\epsilon^F$ will occur infinitely often  almost surely.
\end{proof}

\section{$\liminf$ estimates for $R_n$ }

In this section, we consider the following type of sets, 
for $0\le t_1\le t_2\le \ldots \le t_k \le 1$, 
\begin{equation}\label{eq:deftildE}
 \tilde{E}^{\mathcal{A},\delta}_{n,\mathcal{T}} 
 = \bigcap_{i=1}^k  \{ (a_i- \delta) \hat{h}_3(n)  \le R_{t_i n} \le a_i \hat{h}_3(n)  \},
\end{equation}
where $a_1 \le a_2 \le \ldots  \le a_k$. 
We only have to show the main result for $\mathcal{K}$ restricted by $h:[0,1]\to [0, \infty)$. 
Indeed, if we can show it for $h:[0,1]\to [0, \infty)$, then the same proof holds for a function $h:[0,T] \to [0,\infty)$ for any arbitrarily large, but finite time $T$. Once, we have achieved this step, we can easily extend it to $h:[0,\infty) \to [0,\infty)$ by noting for any increasing function $h$ with limit $\infty$ at $\infty$, there is $T$ sufficiently large such that $\left|\frac{h(t)}{1+h(t)} - 1\right| \le \frac{\epsilon}{2}$ for $t \ge T$. Thus, if we know that $f$ is a function such that the distance from $f$ to $h$ is less than $\epsilon/2$ in the topology defined on functions from $[0,T] \to [0,\infty)$, then the distance from $f$ to $h$ will be less than $\epsilon$ on  the topology defined on functions from $[0,\infty) \to [0,\infty)$.


\subsection{Lower bound estimates on $\mathbb{P}( \tilde{E}^{\mathcal{A},\delta}_{n,\mathcal{T}})$}

The Brownian capacity of $D \subset \RR^3$ can be defined by
\begin{align*}
\ca_B(D)^{-1}
:= \inf \bigg\{ \int\int G_B(x-y) \mu(dx) \mu(dy): \mu(D)=1 \bigg\}, \,
\end{align*}
where $G_B(\cdot)$ is the Green's function for the Brownian motion. 
The capacity of a random walk can be compared to the Brownian capacity of a Brownian motion. From \cite[Theorem 1.7]{Ch17} and  from \cite[Lemma 3.6]{DemboOkada}, there exists a coupling between the random walk and the Brownian motion for $d=3$ such that the following holds almost surely,

\begin{equation} \label{eq:coupl1}
\lim_{n \to \infty} \frac{R_n}{\ca_B(B[0,n/3])} = \frac{1}{3},
\end{equation}
If we are only dealing with a finite collection of points $t_1,\ldots, t_k$, we may also assume that,
\begin{equation*}
\lim_{n \to \infty} \frac{R_{t_i n}}{\ca_B(B[0,t_in/3])} = \frac{1}{3},
\end{equation*}
uniformly. Using this coupling, we can instead deal with the Brownian motion when trying to deduce properties of the $\liminf$ LIL for the random walk.
If we write $\mathcal{A}=(a_1,\ldots, a_k)$ and $\mathcal{T}=(t_1,\ldots, t_k)$, we introduce
$$
\mathcal{G}_{n,i, \mathcal{T}}^{\mathcal{A},\hat{\delta}}:= \{\sup_{s \in [t_{i-1}n/3,t_i n/3]}{\|B_s\|} \le  a_i \psi(n), \quad \|B_{t_{i}n/3 }\| \le b_{i,n}^{\hat{\delta}}\},
$$
for $\psi(n):= \pi \sqrt{n / (6 \log^{(2)} n)}$, 
$b_{i,n}^{\delta}:=2^i \hat{\delta} \pi \psi(n)$ 
and
$$
\mathcal{B}_{n,i,\mathcal{T}}^{\mathcal{A},\delta}:=\{|\nbd(B[0,t_i n/3], n^{1/2}/(\log^{(2)}n)^{3/2- \kappa})| \le  \omega_3 \left((a_i - \delta/2) \psi(n)\right)^3 \},
$$
where $\kappa>0$ is any small parameter and given the value $k$, $\hat{\delta}$ will be chosen to be extremely small relative to $\delta$. 
We also define,
$$
\mathcal{TG}_{n,\mathcal{T}}^{\mathcal{A},\hat{\delta}}:= \bigcap_{i=1}^k \mathcal{G}_{n,i,\mathcal{T}}^{\mathcal{A},\hat{\delta}}
$$
and
$$
\mathcal{BU}_{n,i,\mathcal{T}}^{\mathcal{A},\delta,\hat{\delta}}:= \mathcal{TG}_{n,\mathcal{T}}^{\mathcal{A},\hat{\delta}} \bigcap \mathcal{B}_{n,i,\mathcal{T}}^{\mathcal{A},\delta}.
$$
We consider the following event:
\begin{equation}\label{eq:defF}
\begin{aligned}
F^{\mathcal{A},\hat{\delta},\delta}_{n,\mathcal{T}} 
:=
\mathcal{TG}_{n,\mathcal{T}}^{\mathcal{A},\hat{\delta}}
\setminus 
\bigcup_{i=1}^k \mathcal{BU}_{n,i,\mathcal{T}}^{\mathcal{A},\delta,\hat{\delta}}.
\end{aligned}
\end{equation}
We will estimate the lower bound for the probability of $F^{\mathcal{A},\hat{\delta},\delta}_{n,\mathcal{T}}$. In words, the event $F^{\mathcal{A},\hat{\delta},\delta}_{n,\mathcal{T}} $ ensures that the Brownian motion stays within a ball in order to ensure the capacity cannot be too large (and indeed within $a_i \psi(n)$ by time $t_i n/3$. In addition, it also removes places where the volume taken up by the Wiener sausage is too small; this prevents the capacity from being smaller than $(a_i-\delta)\psi(n)$ by time $t_in/3$.

In the following lemma, we establish that the events $F^{\mathcal{A},\hat{\delta},\delta}_{n,\mathcal{T}}$ would imply the events $\tilde{E}^{\mathcal{A},\delta}_{n,\mathcal{T}}$. Consequently, if we want to show that $\tilde{E}^{\mathcal{A},\delta}_{n,\mathcal{T}}$ occurs infinitely often, it would suffice to show that the event $F^{\mathcal{A},\hat{\delta},\delta}_{n,\mathcal{T}} $ would also occur infinitely often. 
\begin{lem} \label{lem:lwrbndtildE}
The event $F^{\mathcal{A},\hat{\delta},\delta}_{n,\mathcal{T}}$ described above would imply the event $ \tilde{E}^{\mathcal{A},\delta}_{n,\mathcal{T}}.$  As such, we have,
\begin{equation*}
\mathbb{P}(F^{\mathcal{A},\hat{\delta},\delta}_{n,\mathcal{T}}) \le \mathbb{P}( \tilde{E}^{\mathcal{A},\delta}_{n,\mathcal{T}}) - k \exp[-(\log^{(3)} n)^{\eta} \log^{(2)}n],
\end{equation*}
where $\eta>0$ on the right-hand side above is some small universal constant.
\end{lem}
\begin{rem}
Combined with the lower bound on $\mathbb{P}(F^{\mathcal{A},\hat{\delta},\delta}_{n,\mathcal{T}})$ from the proceeding Lemma \ref{lem:LwrbndF}, we can interpret the main result of the lemma above as saying
$$
(1+o(1))\mathbb{P}(F^{\mathcal{A},\hat{\delta},\delta}_{n,\mathcal{T}}) \le \mathbb{P}( \tilde{E}^{\mathcal{A},\delta}_{n,\mathcal{T}}). 
$$
\end{rem}
\begin{proof}

Recall that we let $\mathcal{B}(r)=\mathcal{B}_r$ denote the ball of radius $r$ centered at $0$.
Note that if we have $\sup_{s\le t_i n/3}\|B_s\| \le a_i \psi(n)$, then the set $B[0,t_i n/3]$ will be contained in $\mathcal{B}(a_i \psi(n))$, and, by monotonicity of the capacity (cf. \cite{LA91}), we can deduce that on the event $F^{\mathcal{A},\hat{\delta},\delta}_{n,\mathcal{T}}$, we would have that $\ca_B(B_{[0,t_in/3]}) \le \ca_B(\mathcal{B}(a_i \psi(n)) =a_i 2 \pi \psi(n).$

What we have left to argue is that on the event $F^{\mathcal{A},\hat{\delta},\delta}_{n,\mathcal{T}}$, we cannot have that $\ca(B_{[0,t_i n/3]}) \le (a_i - \delta) \hat{h}_3(n).$  We will prove this by contradiction. 

Now, an argument from \cite{DemboOkada+} demonstrates that if $\ca_{B}(B[0,t_i n/3]) \le (a_i - \delta) \hat{h}_3(n)$, then one would be able to find some subset $V_i$ of $\nbd(B[0,t_in/3], n^{1/2}/(\log^{(2)}n)^{3/2- \kappa})$ such that $$|V_i| = (1- o(1)) | \nbd(B[0,t_in/3], n^{1/2}/(\log^{(2)}n)^{3/2- \kappa})|$$ and $\ca(V) \le (a_i-\delta) \hat{h}_3(n)$ with probability at least $1 - \exp[- \log^{(2)}n (\log^{(3)} n)^{\eta}]$ for some small fixed $\eta>0$. Assume now that these events hold, and we can find such good subsets $V_1,\ldots,V_k$ for all times $t_1 n/3,\ldots , t_k n/3$. On the event $F^{\mathcal{A},\hat{\delta},\delta}_{n,\mathcal{T}}$, since we have removed the events $\mathcal{BU}_{n,i,\mathcal{T}}^{\mathcal{A},\delta,\hat{\delta}}$, we must have that $|V_i| \ge \omega_3 \left((a_i - 2\delta/3) \psi(n)\right)^3.$

Now, we observe  that,
\begin{equation*}
 \ca_B(V_i) \ge \inf_{|A| = \omega_3 \left((a_i - 2\delta/3) \psi(n)\right)^3} \ca_B(A) \ge \ca_{B}(\mathcal{B}((a_i - 2\delta/3) \psi(n))) \ge (a_i - 2\delta/3) 2 \pi \psi(n).
\end{equation*}
The last inequality used the Poincar{\'e}-Carleman-Szeg{\"o} Theorem \cite{PS51}.

Thus, under the event that we can construct the sets $V_1,\ldots,V_k$ the event $F^{\mathcal{A},\hat{\delta},\delta}_{n,\mathcal{T}}$ must imply the event $\tilde{E}^{\mathcal{A},\delta}_{n,\mathcal{T}}$. It is only with at most probability $ k \exp[-(\log^{(3)} n)^{\eta} \log^{(2)}n]$ that we cannot construct the events $V_1,\ldots,V_k$.



\end{proof}

\begin{lem} \label{lem:LwrbndF}
There exists some small constant $c>0$ (that could depend on the values of $\delta,\hat{\delta},k, a,t$ but not $n$) such that for $n$ large enough we have
\begin{equation*}
\mathbb{P}(F^{\mathcal{A},\hat{\delta},\delta}_{n,\mathcal{T}}) \ge  c \exp\left[-\sum_{i=1}^k \frac{t_i - t_{i-1}}{(a_k - 2^{i+1} \hat{\delta})^2} \log^{(2)}n \right].
\end{equation*}
\end{lem}

\begin{proof}
Our goal is to show for all $i$ that $\mathbb{P}(\mathcal{BU}_{n,i,\mathcal{T}}^{\mathcal{A},\delta,\hat{\delta}}) \ll \mathbb{P}(\mathcal{TG}_{n,\mathcal{T}}^{\mathcal{A},\hat{\delta}})$. 
Consider the event,
\begin{equation*}
\tilde{\mathcal{B}}_{n,i,\mathcal{T}}^{\mathcal{A},\delta} 
:=\{|\nbd(B[t_{i-1}n/3,t_i n/3], n^{1/2}/(\log^{(2)}n)^{3/2-\kappa})| \le  \omega_3 \left((a_i - \delta/2) \psi(n)\right)^3 \},
\end{equation*}
as well as the slightly modified bad union event,
\begin{equation*}
\tilde{\mathcal{BU}}_{n,i,\mathcal{T}}^{\mathcal{A},\delta,\hat{\delta}}:=\mathcal{TG}_{n,\mathcal{T}}^{\mathcal{A},\hat{\delta}} \bigcap \tilde{\mathcal{B}}_{n,i,\mathcal{T}}^{\mathcal{A},\delta}. 
\end{equation*}

We will show instead that $\mathbb{P}(\tilde{\mathcal{BU}}_{n,i,\mathcal{T}}^{\mathcal{A},\delta,\hat{\delta}}) \ll\mathbb{P}(\mathcal{TG}_{n,\mathcal{T}}^{\mathcal{A},\hat{\delta}}) $.  (We also remark that $\mathbb{P}(\mathcal{BU}^{\mathcal{A},\delta,\tilde{\delta}}_{n,i,\mathcal{T}}) \le \mathbb{P}(\tilde{\mathcal{BU}}^{\mathcal{A},\delta,\tilde{\delta}}_{n,i,\mathcal{T}})$.) The benefit now is that all of the events in $\mathcal{BU}_{n,i,\mathcal{T}}^{\mathcal{A},\delta,\hat{\delta}}$ involve the disjoint parts $B_{[t_{i-1}n/3,t_i n/3]}$. Indeed, we first notice that the condition 
$$\{\sup_{s \in [t_{i-1}n/3, t_i n/3]}\|B_s\| \le a_i \psi(n), \quad \|B_{t_{i-1}n/3}\| \le b^{\hat{\delta}}_{i,n}\}$$ implies
$$
\sup_{s\in[t_{i-1}n/3, t_i n/3]}\|B_s - B_{t_{i-1}n/3}\| \le a_i \psi(n) +b^{\hat{\delta}}_{i,n}.
$$

Thus, we have that,
\begin{equation}
\begin{aligned}\label{first**}
\mathbb{P}(\tilde{\mathcal{BU}}_{n,i,\mathcal{T}}^{\mathcal{A},\delta, \hat{\delta}}) &\le \prod_{j \ne i} \mathbb{P}(\sup_{s\in[t_{i-1}n/3, t_i n/3]}\|B_s - B_{t_{i-1}n/3}\| 
\le a_i \psi(n) +b^{\hat{\delta}}_{i,n}) \\
& \times\mathbb{P}(|\nbd(B[t_{i-1}n/3,t_i n/3],n^{1/2}/(\log^{(2)}n)^{3/2-\kappa})| \le  \omega_3 \left((a_i - \delta) \psi(n)\right)^3)\\
&\le C^k \exp\left[-\sum_{j \ne i} \frac{t_j -t_{j-1}}{(a_j + 2^{i+1} \hat{\delta})^2} \log^{(2)} n \right] \times \exp\left[- \frac{t_i -t_{i-1}}{(a_i - \delta/2})^2 \log^{(2)} n\right],
\end{aligned}
\end{equation}
where $C$ is some deterministic constant. To estimate the bound on the volume of the neighborhood of the Wiener sausage, we used the celebrated large deviation principle \cite[Equation (1.4)]{BBH01}. The estimate that the supremum of a Brownian motion stays within a ball of radius $r$ for time $t$ is a result of \cite{Kent80}.

Now, we consider the probability of the event $\mathcal{TG}_{n,\mathcal{T}}^{\mathcal{A},\hat{\delta}}$. By taking a conditioning, we observe that,

\begin{equation*}
\begin{aligned}
\mathbb{P}(\mathcal{TG}_{n,\mathcal{T}}^{\mathcal{A},\hat{\delta}})
= \mathbb{P}(\mathcal{TG}_{n,\mathcal{T}_{k-1}}^{\mathcal{A}_{k-1},\hat{\delta}}) \mathbb{P}\bigg( \sup_{s \in [t_{k-1}n/3,t_k n/3]} \|B_s\| \le a_{k}\psi(n),\quad  \|B_{t_k n/3}\| \le b_{k,n}^{\hat{\delta}} \bigg| 
 \|B_{t_{k-1}n/3}\| \le b_{k-1,n}^{\hat{\delta}} \bigg).
\end{aligned}
\end{equation*}

Conditional on $\|B_{t_{k-1}n/3}\| \le b_{k-1,n}^{\hat{\delta}}$, the event $\{\sup_{s \in [t_{k-1}n/3,t_k n/3]}\|B_s\| \le a_k \psi(n) \} \cap \{\|B_{t_{k}n/3}\| \le b_{k,n}^{\hat{\delta}}\}$ will certainly be implied by the event 
\begin{equation*}
\left\{ \sup_{s \in [t_{k-1}n/3, t_k n/3]} \|B_s - B_{t_{k-1}n/3}\| \le a_k\psi(n)- b_{k-1,n}^{\hat{\delta}}, \quad \|B_{t_kn/3}-B_{t_{k-1}n/3}\| \le b_{k-1,n}^{\hat{\delta}}  \right\}.
\end{equation*}
$\|B_s - B_{t_{k-1}n/3}\|$ on the time region $[t_{k-1}n/3, t_{k}n/3]$ behaves like a Brownian motion of time $(t_{k}- t_{k-1})n/3$. 
Thus, we have that,
\begin{equation} \label{eq:inducR}
\begin{aligned}
&\mathbb{P}(\mathcal{TG}_{n,\mathcal{T}}^{\mathcal{A},\hat{\delta}})
\ge \mathbb{P}(\mathcal{TG}_{n,\mathcal{T}_{k-1}}^{\mathcal{A}_{k-1},\hat{\delta}})
 \times \mathbb{P}\left( \sup_{s \in [0, (t_k-t_{k-1}) n/3]}\|B_s \| \le a_k\psi(n)- b_{k-1,n}^{\hat{\delta}},\quad  \|B_{(t_k-t_{k-1})n/3}\| \le b_{k-1,n}^{\hat{\delta}} \right).
\end{aligned}
\end{equation}

Now, observe the following decomposition,
\begin{equation*}
\begin{aligned}
 &\mathbb{P}\left( \sup_{s \in [0, (t_k-t_{k-1}) n/3]} \|B_s \| \le a_k\psi(n)- b_{k-1,n}^{\hat{\delta}}  \right)\\
 =& \sum_{j^1,j^2,j^3 =-\sqrt{3}a_k/(2^{k} \hat{\delta})}^{\sqrt{3}a_k/(2^{k} \hat{\delta})} 
 \mathbb{P}\bigg( \sup_{s \in [0, (t_k-t_{k-1}) n/3]} \|B_s \| \le a_k\psi(n)- b_{k-1,n}^{\hat{\delta}}, \quad
 {|B^i_{(t_k-t_{k-1})n/3} +j^i b_{k,n}^{\hat{\delta}}| \le \sqrt{3}^{-1} b_{k-1,n}^{\hat{\delta}} } \bigg)\\
 & \le \left[\frac{2 \sqrt{3} a_k}{2^k \hat{\delta}}+1 \right]^3 \mathbb{P}\bigg( \sup_{s \in [0, (t_k-t_{k-1}) n/3]} \|B_s \| \le a_k\psi(n)- b_{k-1,n}^{\hat{\delta}}, \quad
 \|B_{(t_k-t_{k-1})n/3}\| \le b_{k-1,n}^{\hat{\delta}}  \bigg)
\end{aligned}
\end{equation*}
In the last line, we used Anderson's inequality to deduce that,
\begin{equation*}
\begin{aligned}
&\mathbb{P}\bigg( \sup_{s \in [0, (t_k-t_{k-1}) n/3]} \|B_s \| \le a_k\psi(n)- b_{k-1,n}^{\hat{\delta}},\quad  
 |B^i_{(t_k-t_{k-1})n/3} +j^i b_{k,n}^{\hat{\delta}}| 
 \le \sqrt{3}^{-1}b_{k-1,n}^{\hat{\delta}}  \bigg)\\
 & \le \mathbb{P}\bigg( \sup_{s \in [0, (t_k-t_{k-1}) n/3]} \|B_s \| \le a_k\psi(n)- b_{k-1,n}^{\hat{\delta}}, \quad
 \|B_{(t_k-t_{k-1})n/3} \| \le b_{k-1,n}^{\hat{\delta}}  \bigg).
\end{aligned}
\end{equation*}
Thus, we have that,
\begin{equation*}
\begin{aligned}
&\mathbb{P}\left( \sup_{s \in [0, (t_k-t_{k-1}) n/3]} \|B_s \| \le a_k\psi(n)- b_{k-1,n}^{\hat{\delta}},\quad \|B_{(t_k-t_{k-1})n/3}\| \le b_{k-1,n}^{\hat{\delta}}  \right)\\
& \ge \mathbb{P}\left(\sup_{s \in [0,(t_k - t_{k-1})n/3]}\|B_s\| \le a_k\psi(n)- b_{k-1,n}^{\hat{\delta}} \right) \frac{1}{\left[\frac{2 \sqrt{3}a_k}{2^k \hat{\delta}} + 1\right]^3} 
\ge c \exp\left[-  \frac{t_k -t_{k-1}}{(a_k - 2^k \hat{\delta})^2} \log^{(2)} n \right], 
\end{aligned}
\end{equation*}
where the last inequality holds for $n$ sufficiently large. 

By iterating this estimate in \eqref{eq:inducR}, we see that we can derive the desired estimate,
\begin{equation}\label{second**}
    \mathbb{P}(\mathcal{TG}_{n,\mathcal{T}}^{\mathcal{A},\hat{\delta}}) \ge c^k \exp\left[-\sum_{i=1}^k \frac{t_i - t_{i-1}}{(a_k - 2^{i+1} \hat{\delta})^2} \log^{(2)}n \right].
\end{equation}

Choosing $\hat{\delta}$ significantly small relative to $\delta$, by \eqref{first**} and \eqref{second**}, we can make the upper bound in \eqref{first**} significantly smaller than the lower bound in \eqref{second**}. Thus, for $n$ large enough, we can bound  $\mathbb{P}(F^{\mathcal{A},\hat{\delta},\delta}_{n,\mathcal{T}})$ from below by the right hand side of \eqref{second**}. 
\end{proof}

As a consequence of this estimate, we can derive our final assertion the capacity, after appropriate rescaling, would lie in neighborhoods around functions belonging to $\mathcal{K}$.

\begin{thm}
Choose $f \in \mathcal{K}$. 
Consider any neighborhood around $f$ of the form,
\begin{equation*}
 N^{\mathcal{T},\delta,f}:= \{g: f(t_i) - \delta  \le g(t_i) \le f(t_i) + \delta, \quad  \forall i \text{ s.t. } 1\le  i \le k\}.
\end{equation*}
Then,  $g_n \in  N^{\mathcal{T},\delta,f}$ infinitely often with probability 1.
\end{thm}
\begin{proof}
Since  $f$ is monotone increasing, it must be the case that we have,
\begin{equation*}
1 \ge \int_0^1 \frac{1}{f^2(t)} \text{d}t \ge \sum_{i=1}^n \frac{t_i -t_{i-1}}{f(t_i)^2}.
\end{equation*}

Thus, for any $\hat{\delta}>0$, it must be the case that,
\begin{equation*}
\sum_{i=1}^k \frac{t_i -t_{i-1}}{(f(t_i)+ \hat{\delta})^2}<1.
\end{equation*}

In what proceeds, we will choose $\hat{\delta}$ such that $4 \hat{\delta} < \delta$. 
Fix some $\epsilon>0$ and consider the times $\hat{t}_{2i-1} := t_i - \epsilon$ and $\hat{t}_{2i}:= t_i + \epsilon$ for $i \le k-1$ and set $\hat{t}_{2k-1}:= t_k- \epsilon(= 1-\epsilon)$ and $\hat{t}_{2k}:= t_k(=1).$ 
Furthermore, set the values $a_{2i-1} =f(t_i) + 2\hat{\delta}$, $a_{2i} = f(t_i) + 2\hat{\delta}$ for $i \le k$.
For sufficiently small $\epsilon$ (depending on $\hat{\delta}$), we must also have that
\begin{equation*}
\sum_{i=1}^{2k} \frac{\hat{t}_{i} - \hat{t}_{i-1} }{(a_i- \hat{\delta})^2}  = 1 - \kappa,
\end{equation*}
with $\kappa>0$. 
We define $d_n:=n^n - (n-1)^{n-1}$ and $\hat{\mathcal{T}}=(\hat{t}_1,\ldots, \hat{t}_{2k})$. 
We now consider the shifted random walk 
$\tilde{S}_i=:S_{i+(n-1)^{n-1}}-S_{(n-1)^{n-1}}$ with $0\le i\le  d_n$. 
We now consider the event, 
$
 \tilde{E}^{\mathcal{A}_{2k},\hat{\delta}}_{d_n,\hat{\mathcal{T}}}
$ applied to the walk $\tilde{S}.$ Now, from the results of Lemmas  \ref{lem:lwrbndtildE} and \ref{lem:LwrbndF}, we see that,
\begin{equation*}
\begin{aligned}
\mathbb{P}( \tilde{E}^{\mathcal{A}_{2k},\hat{\delta}}_{d_n,\hat{\mathcal{T}}})& \ge \exp\left[ - \sum_{i=1}^{2k} \frac{\hat{t}_i - \hat{t}_{i-1}}{(a_i-\hat{\delta})^2} \log^{(2)} n^n\right] \ge \exp[-(1-\kappa) [\log n + \log^{(2)} n]],\\
& \ge \exp[-(1- \kappa') \log n] \ge \frac{1}{n^{1-\kappa'}}.
\end{aligned}
\end{equation*}
Here, $\kappa'$ is a positive constant slightly smaller than $\kappa$, and the inequality on the last line would hold for sufficiently large $n$. Now, by applying the second Borel-Cantelli lemma, noting that the walks $\tilde{S}_{[0,d_n]}$ are independent of each other, would imply that the events $ \tilde{E}^{\mathcal{A}_{2k},\hat{\delta}}_{d_n,\hat{\mathcal{T}}}$  would occur infinitely often.

Now, the event $ \tilde{E}^{\mathcal{A}_{2k},\hat{\delta}}_{d_n,\hat{\mathcal{T}}}$ controls the capacity on regions of the form $S_{[(n-1)^{n-1}, \hat{t}_i n^n + (1-\hat{t}_i) (n-1)^{n-1}]}$. We seek to use this information to control the capacity on $S_{[0,t_i n^n]}$. 
For $n$ sufficiently large, we have that,
$$
t_i n^n \ge \hat{t}_{2i-1} n^n + (1- \hat{t}_{2i-1}) (n-1)^{n-1}, \text{  } t_i n^n \le \hat{t}_{2i} n^n + (1- \hat{t}_{2i}) (n-1)^{n-1}. 
$$
Thus, we have that,
\begin{equation*} 
\begin{aligned}
&R_{t_i n^n}\le 
R[1,\hat{t}_{2i} n^n + (1- \hat{t}_{2i}) (n-1)^{n-1}]
\le \ca(\tilde{S}_{[0,\hat{t}_{2i}d_n]}) + R_{(n-1)^{n-1}},\\
& R_{t_i n^n}
\ge R[(n-1)^{(n-1)},\hat{t}_{2i-1} n^n + (1- \hat{t}_{2i-1}) (n-1)^{n-1}]
\ge \ca(\tilde{S}_{[0,\hat{t}_{2i-1}d_n]}). 
\end{aligned}
\end{equation*}
By the bound on the $\limsup$ of the capacity of the range of a random walk, we can show that 
\begin{equation} \label{eq:ordmag}
R_{(n-1)^{n-1}} \le C \sqrt{(n-1)^{n-1} \log^{(2)} (n-1)^{n-1}} \ll \sqrt{\frac{n^n}{\log^{(2)}n^n}}.
\end{equation}
Hence, by applying the bound from $ \tilde{E}^{\mathcal{A}_{2k},\hat{\delta}}_{d_n,\hat{\mathcal{T}}}$, we have that,
\begin{equation*}
\begin{aligned}
R_{t_i n^n}&\le  a_{2i} \sqrt{\frac{d_n}{\log^{(2)} d_n}} +C \sqrt{(n-1)^{n-1} \log^{(2)}(n-1)^{n-1}} \le a_{2i} \sqrt{\frac{n^n}{\log^{(2)} n^n}} + \hat{\delta} \sqrt{\frac{n^n}{\log^{(2)} n^n}},\\
&\le (f(t_i) + 3 \hat{\delta}) \sqrt{\frac{n^n}{\log^{(2)} n^n}}.
\end{aligned}
\end{equation*}

In the second inequality, we used both the fact that $\sqrt{\frac{x}{\log^{(2)} x}}$ is an increasing function and the order of magnitude bound from equation \eqref{eq:ordmag}. The final line uses the definition of $a_{2i}$ as $f(t_i) + 2 \hat{\delta}$. 
For the lower bound, we have,
\begin{equation*}
R_{t_i n^n}\ge  \ca(\tilde{S}_{[0,d_n]}) \ge (a_{2i-1}-\hat{\delta}) \sqrt{\frac{d_n}{\log^{(2)} d_n}} \ge (f(t_i) + \hat{\delta}) \frac{f(t_i)}{f(t_i) +  \hat{\delta}} \sqrt{\frac{n^n}{\log^{(2)} n^n}}.
\end{equation*}
In the last line, we used the fact that $n^n - d_n = (n-1)^{n-1} \ll n^n$. Thus, for any constant $k < 1$, we eventually have that $d_n > k n^n$ when $n$ is large enough. We  set $k$ to be the constant $\left(\frac{f(t_i)}{f(t_i) + \hat{\delta}}\right)^2$. 

Since the event $ \tilde{E}^{\mathcal{A}_{2k},\hat{\delta}}_{d_n,\hat{\mathcal{T}}}$ is satisfied infinitely often, we see that $g_n$ must belong to the neighborhood $N^{\mathcal{T}, \delta, f}$ infinitely often. 
\end{proof}


\subsection{Upper bound estimates on $\mathbb{P}( \tilde{E}^{\mathcal{A},\delta}_{n,\mathcal{T}})$}

To derive upper bound estimates on the probability of the set $\tilde{E}^{\mathcal{A},\delta}_{n,\mathcal{T}}$, 
we consider the event
\begin{equation*}
 V^{\mathcal{A}}_{n,\mathcal{T}}:= \bigcap_{i=1}^k \{|\text{Nbd}(B[t_{i-1}n/3,t_in/3],{ n^{1/2}/(\log^{(2)}n)^{3/2-\kappa}})| \le \omega_3 (a_i \psi(n))^3 \}.
\end{equation*}
We have that,
\begin{lem} \label{lem:Vbnd}
Recall the events $V$ from above and $\tilde{E}$ from equation \eqref{eq:deftildE}. We see that the event $ \tilde{E}^{\mathcal{A},\delta}_{n,\mathcal{T}}$ would imply the event $ V^{\mathcal{A}}_{n,\mathcal{T}}$. Namely,
\begin{equation*}
\mathbb{P}( \tilde{E}^{\mathcal{A},\delta}_{n,\mathcal{T}}) \le \mathbb{P}( V^{\mathcal{A}}_{n,\mathcal{T}}).
\end{equation*}
Furthermore,
\begin{equation*}
\mathbb{P}( V^{\mathcal{A}}_{n,\mathcal{T}}) \le C\exp \left[ - \sum_{i=1}^k \frac{t_i - t_{i-1}}{a_i^2} \log^{(2)} n \right].
\end{equation*}
\end{lem}
\begin{proof}
As was discussed in the proof of Lemma \ref{lem:lwrbndtildE}, if we have that $\ca_B(B[0,t_in/3]) \le  a_i \psi(n)$, then with very high probability, it must have been the case that \begin{equation}\label{eq:voltibnd}|\text{Nbd}(B[ 0,t_in/3],{ n^{1/2}/(\log^{(2)}n)^{3/2-\kappa}})| \le \omega_3 (a_i \psi(n))^3.
\end{equation}

As a consequence of equation \eqref{eq:voltibnd}, we must also have
\begin{equation*}
|\text{Nbd}(B[t_{i-1}n /3, t_i n/3],  n^{1/2}/(\log^{(2)}n)^{3/2-\kappa})| \le \omega_3 (a_i \psi(n))^3.
\end{equation*}
This implies that the event $\tilde{E}^{\mathcal{A}, \delta}_{n,\mathcal{T}}$ would imply the event $ V^{\mathcal{A}}_{n,\mathcal{T}}$.

Now, since the intervals $B[t_{i-1}n/3,t_in/3]$ are independent of each other, we have that,
\begin{equation*}
\begin{aligned}
\mathbb{P}( V^{\mathcal{A}}_{n,\mathcal{T}}) &= \prod_{i=1}^k \mathbb{P}(|\text{Nbd}(B[t_{i-1}n/3,t_in/3 ],{n^{1/2}/(\log^{(2)}n)^{3/2-\kappa}})| \le \omega_3 (a_i \psi(n))^3)\\
& \le C^k \exp \left[-\sum_{i=1}^k  \frac{t_i - t_{i-1}}{a_i^2} \log^{(2)} n \right].
\end{aligned}
\end{equation*}
Again, we used here the large deviation principle for the volume of the Wiener sausage from \cite{BBH01}.
\end{proof}

With the probability estimate from the last paragraph, we can finally establish the following theorem.
\begin{thm}
Let $f$ be a monotone increasing function such that 
\begin{equation*}
\int_0^1 \frac{1}{f^2(x)} \text{d}x >1.
\end{equation*}
We can find some neighborhood around $f$ of the form,
\begin{equation*}
 N^{\mathcal{T},\delta,f}:=\{g: g(t_i)  \le f(t_i) + \delta, \forall i \in \{1,\ldots,k\} \},
\end{equation*}
such that, almost surely, there exists $N(\omega)$ large enough with $g_n \not \in  N^{\mathcal{T},\delta,f}$ for $n \ge N (\omega)$. It is clear that if we show $g_n \not \in N^{\mathcal{T},\delta,f}$, then it is also the case that $g \not \in \tilde{N}^{\mathcal{T},\delta,f}$ with,
$$
\tilde{N}^{\mathcal{T},\delta,f}:=\{g: f(t_i) \le g(t_i) \le f(t_i) + \delta, \forall i \in\{1,\ldots,k \} \}.
$$
\end{thm}
\begin{proof}

 Since we have that,
 \begin{equation*}
\int_0^1 \frac{1}{f^2(x)} \text{d}x >1, 
 \end{equation*}
 we can find  a discretization of the integral, as well as a sufficiently small $\delta$ such that,
 \begin{equation*}
\sum_{j=1}^k \frac{t_j - t_{j-1}}{(f(t_j) +\delta)^2} =1 + \kappa,
 \end{equation*}
 for some $\kappa >0$. 
We now consider the neighborhood $ N^{\mathcal{T},\delta,f}$.  
By the result of Lemma \ref{lem:Vbnd}, we can ensure that,
\begin{equation*}
\mathbb{P}(g_n \in  N^{\mathcal{T},\delta,f}) \le \exp \left[- \log^{(2)} n \sum_{j=1}^k \frac{t_j - t_{j-1}}{(f(t_j) + \delta)^2}\right].
\end{equation*}
Now, we consider values $n = q^i$ for some $q>1$ and $i \in \mathbb{N}$. We see that we have,
\begin{equation*}
\begin{aligned}
\sum_{i=1}^{\infty}
\mathbb{P}(g_{q^i} \in  N^{\mathcal{T},\delta,f}) 
\le \sum_{i=1}^{\infty} \exp\left[ - \log^{(2)} q^i \sum_{j=1}^k \frac{t_j -t_{j-1}}{(f(t_j) + \delta)^2} \right] & \le \sum_{i=1}^{\infty} \exp[-  (1+\kappa)( \log i + \log^{(2)} q)] \\
& \le  \exp[-(1+\kappa) \log^{(2)} q]\sum_{i=1}^{\infty} \frac{1}{i^{1+\kappa}} < \infty.
\end{aligned}
\end{equation*}
By the Borel-Cantelli lemma, we can ensure that for $i$ large enough, it must be the case that $g_{q^i}$ is not an element of the neighborhood $N^{\mathcal{T},\delta,f}$.

Let us consider what happens to values $n$ between $q^i\le n \le q^{i+1}$. 
Notice that we have the monotonicity relationship,
\begin{equation*}
 \frac{R_{q^i t_l}}{\frac{\sqrt{6 \pi^2}}{9} \sqrt{\frac{q^{i+1}}{\log^{(2)} q^{i}}}} \le \frac{R_{n t_l}}{\frac{\sqrt{6 \pi^2}}{9} \sqrt{\frac{n}{\log^{(2)} n}}} \quad \forall l,
\end{equation*}
and, thus,
\begin{equation*}
g_{q^i}(t_l) \sqrt{q}^{-1} \le g_{n}(t_l).
\end{equation*}
Hence, if we know that $g_n(t_l) \le f(t_l)+ \delta/2$, then it must be the case  that,
$$
g_{q^i}(t_l) \le f(t_l) + [\sqrt{q}-1] f(t_l) + \sqrt{q} \delta/2.
$$
For $q$ sufficiently close to 1, one can ensure that the term $[\sqrt{q}-1]f(t_l) + \sqrt{q}\delta/2 $ will be less than $ \delta$. 
Thus, if $g_n \in N^{\mathcal{T},\delta/2,f}$ for some $n$, then it must have been the case that $g_{q^i} \in N^{\mathcal{T},\delta,f}$, for $q$ sufficiently small and $i = \lfloor \frac{ \log n}{\log q} \rfloor$. 

Since we have shown for all $q>1$ that it must be the case that eventually $g_{q^i} \not \in N^{\mathcal{T},\delta,f} $, we must have that eventually $g_n \not \in N^{\mathcal{T},\delta/2,f}$ for all sufficiently large $n$. 
\end{proof}

\section{Proof of Theorem \ref{m2}}

\subsection{Proof of (i) in Theorem \ref{m2}} 


\subsubsection{Preliminaries}
We will first quickly establish that  $\mathbb{P}( A_n^2 \cap B_n \quad \text{i.o.})=0$: 
notice $\ca (\mathcal{B}_{j_3(n)}) \gg h_3(n)$ and, by our $\limsup$ bounds on the capacity of the random walk range from Section \ref{sec:Strassenlimsuplil}, we know that $\limsup R_n / h_3(n)\le 1$ almost surely. Thus, it would be impossible for $R_n \ge \ca (\mathcal{B}_{j_3(n)}) $ infinitely often. 

We will now spend most of our effort in establishing that  $\mathbb{P}( A_n^1 \cap B_n \quad \text{i.o.})=1$. 
In what proceeds, we will assume that $j_3(n)\ll \sqrt{n \log^{(2)} n}$. As observed in the construction of Section \ref{sec:limsup}, we can obtain the $\limsup$ of the LIL of the capacity of the random walk range by allowing the random walk to lie in the unions of finitely many  cylinders of size $O( \sqrt{n \log^{(2)} n})$.  
Now, given a sphere of radius $j_3(n)$, we can embed inside it a cube  $\mathcal{C}_{ Kj_3(n)}$ of side length $K  j_3(n)$ centered at 0 with sides parallel to the coordinate axes, and $K$ is a constant chosen so that the vertices of this cube lie exactly on the sphere.  We define $\tilde{j}_3(n):= Kj_3(n)$ and use the notation $\mathcal{C}_{\tilde{j}_3(n)}$ instead. In what proceeds, we will find a collection of events of the random walk that satisfy our bounds on the diameter and the capacity and, additionally, lie inside the cube $\mathcal{C}_{\tilde{j}_3(n)}$. We will then show that these events will occur infinitely often. 

We will now begin by describing some heuristics behind our construction before giving a formal discussion. Intuitively, we want our walk to wrap around the surface of the cube such that the distinct `segments' of the walk around the cube are as far apart as possible. The distinct `segments' of the walk will be various cross-sections of the cube parallel to the $x-y$ plane and small segments that connect these cross-sections. Now, notice that each square cross-section will have a perimeter of $4 \tilde{j}_3(n)$; now, computations in the previous section \ref{sec:limsup} suggest that, in order to achieve the $\limsup$, the random walk must have a speed of $\sqrt{\frac{2 \log^{(2)} n}{3n}}$ and the total perimeter covered will be $\sqrt{\frac{2}{3}n \log^{(2)} n}$.   Thus, the random walk trace will approximate the union of $\frac{\sqrt{ \frac{2}{3}n \log^{(2)} n}}{4\tilde{j}_3(n)}$ different cross-sections. (This comes from the fact that each cross section will take approximately $4\tilde{j_3}(n)$ perimeter.) Furthermore, the distance between adjacent cross-sections would be $\frac{4 (\tilde{j}_3(n))^2}{\sqrt{ \frac{2}{3}n \log^{(2)} n}}$ (coming from the fact that we have $\frac{\sqrt{ \frac{2}{3}n \log^{(2)} n}}{4\tilde{j}_3(n)}$ cross sections over a length of $\tilde{j}_3(n)$).  

As such, we define the heights 
$$
h_k:= k \frac{4 (\tilde{j}_3(n))^2}{\sqrt{\frac{2}{3}n \log^{(2)} n}}, \quad k \in \left\{-\frac{\sqrt{\frac{2}{3}n \log^{(2)} n}}{8 \tilde{j}_3(n)},\ldots,  \frac{\sqrt{\frac{2}{3}n \log^{(2)} n}}{8 \tilde{j}_3(n)}\right\}.
$$
In what follows, we will denote $\ell:=\frac{\sqrt{\frac{2}{3}n \log^{(2)} n}}{8 \tilde{j}_3(n)}$ 
and consider the vertices,
$$
p_{a,k} = \left(\frac{\tilde{j}_3(n)}{\sqrt{2}} \cos\left( \frac{\pi}{4} + a \frac{\pi}{2} \right) , \frac{\tilde{j}_3(n)}{\sqrt{2}} \sin\left( \frac{\pi}{4} + a \frac{\pi}{2} \right), h_k  \right).
$$
These are the vertices of the cross-sectional squares at the heights $h_k$.
We introduce the following notation to denote the path,
$$
\mathcal{P}_{k}:= p_{0,k} \to p_{1,k} \to p_{2,k} \to p_{3,k} \to p_{0,k},
$$
that traverses the square. 
The full path we consider will be composed of all these subpaths,
\begin{equation*}
\mathcal{P}:= \mathcal{P}_{-\ell} \to \mathcal{P}_{-\ell+1} \to \ldots \to \mathcal{P}_{\ell -1} \to \mathcal{P}_{\ell}.
\end{equation*}

This path traverses all the cross-sectional squares from bottom to top and traverses between each cross-sectional square using the edge $p_{0,k-1} \to p_{0,k}.$ We can alternatively write this path $\mathcal{P}$ as,
\begin{equation*}
\mathcal{P} = v_0 \to v_1 \to v_2 \to \ldots \to v_{f-1} \to v_f, 
\end{equation*}
where each $v_0$ corresponds to one of the $p_{a,k}$. Note that, here, the subindex only indicates the order in which the vertex appears in the path. Namely, it is possible for vertices with two distinct subindices $v_i$ and $v_j$ to correspond to the same vertex. Now, we also denote $d_i$ to be the total perimeter of the path up to the vertex $v_i$. Namely,
$$
d_i = \sum_{j=0}^{i-1} \|v_{j+1} - v_j \|. 
$$
We also choose a parameter $\kappa >0$, and we can define the time $t_i$ as,
$$
t_i := \frac{1}{1-\kappa}\sqrt{\frac{3n}{2\log^{(2)} n}}d_i,
$$
to be the time in which a walk moving at rate $(1-\kappa)\sqrt{ \frac{2 \log^{(2)} n}{3 n}}$ should reach the distance $d_i$. Notice that by the way we chose our path, the total perimeter would be $d_f:= \sqrt{\frac{2}{3}n \log^{(2)} n} + \tilde{j}_3(n). $ Thus, the time it would take to reach the final point would be 
$$t_f:= \frac{1}{1-\kappa} \left[n + \sqrt{\frac{3n}{2\log^{(2)} n}} \tilde{j}_3(n) \right].$$  With the assumption that $\tilde{j}_3(n) \ll \sqrt{n \log^{(2)} n} $, we see that the second term in the right-hand side of $t_f$ would be negligible.

For the convenience of the reader, we will justify these choices as follows. Notice that if one were concerned with the exponential factor, the probability of making the transition from $v_{i+1} \to v_i$ would be,
\begin{equation*}
\exp\left[- \frac{\|v_{i+1} - v_i\|^2}{\frac{2}{3}(t_{i+1}- t_{i})}\right] = \exp\left[-(1-\kappa)\sqrt{\frac{3 \log^{(2)} n}{2n}} \|v_{i+1} - v_i\|\right].
\end{equation*}
Multiplying these probabilities together for all paths would demonstrate that the probability of making all such transitions would be,
\begin{equation*}
\exp\left[ - (1-\kappa) \log^{(2)} n  - (1-\kappa)\sqrt{\frac{3 \log^{(2)} n}{2n} } \tilde{j}_3(n) \right].
\end{equation*}
The factor of $\log^{(2)} n$ is exactly the threshold between whether an event could occur infinitely often or not. We are free to choose $\kappa$ arbitrarily small.

\subsubsection{Necessary Notation}

Let $\delta>0$ be a small parameter. 
In what follows, we consider the following events. Let $v_i$ be a vertex such that the transition $v_i \to v_{i+1}$ lies in the $x-y$ plane; namely, $v_i \to v_{i+1}$ is an edge along a cross-sectional square of the cube. Assume that the edge $v_i \to v_{i+1}$ goes in the coordinate direction $j$.

We construct the following event,
\begin{equation} \label{eq:defFE}
\begin{aligned}
\mathcal{FE}_i: = &\{  |S^j_{t_{i+1}} - v^j_{i+1}| \le \delta \tilde{j}_3(n) \} \\
&\bigcap\bigcap_{k \ne j} \left\{\max_{l \in [t_i,t_{i+1}]} |S^k_l - S^k_{t_i}| \le \sqrt{ \tilde{j}_3(n) \sqrt{\frac{2n}{3\log^{(2)} n}}} \log^{(4)} n \right \}.
\end{aligned}
\end{equation}
In the expression above, we use the superscript $j=\{x,y,z\}$ to denote the coordinate in the given axis of both the walk and the vertex. To simplify notation when used in the future, we use the notation
\begin{equation*}
\mathcal{TE}_i:= \{  |S^j_{t_{i+1}} - v^j_{i+1}| \le \delta \tilde{j}_3(n) \},
\end{equation*}
for the first line in the event above,
\begin{equation*}
\mathcal{BE}_i:= \bigcap_{k \ne j} \left\{\max_{l \in [t_i,t_{i+1}]} |S^k_l - S^k_{t_i}| \le \sqrt{ \tilde{j}_3(n) \sqrt{\frac{2n}{3\log^{(2)} n}}} \log^{(4)} n \right \}.
\end{equation*}

Now, if the edge $v_i \to v_{i+1}$ traverses the vertical direction, we define the event,
\begin{equation} \label{eq:defFE2}
\begin{aligned}
\mathcal{FE}_i:= & \left\{  |S^z_{t_{i+1}} - v^z_{i+1}| \le \delta \frac{4 (\tilde{j}_3(n))^2}{\sqrt{\frac{2}{3}n \log^{(2)} n}} \right\}\\
& \bigcap \bigcap_{k \ne z} \left\{\max_{l \in [t_i,t_{i+1}]} |S^k_l - S^k_{t_i}| \le \sqrt{ \frac{4 (\tilde{j}_3(n))^2}{\sqrt{\frac{2}{3}n \log^{(2)} n}}  \sqrt{\frac{2n}{3\log^{(2)} n}}} \log^{(4)} n \right \}.
\end{aligned}
\end{equation}
As before, we let the event on the first line above be denoted as $\mathcal{TE}_i$ and the second event be denoted as $\mathcal{BE}_i$

Now, we consider the collection of events
\begin{equation} \label{eq:deffecol}
\mathcal{FE}^i := \cap_{j=0}^{i} \mathcal{FE}_j.
\end{equation}
We are especially concerned with $\mathcal{FE}^{f-1}$, which is the intersection of all these events. On this intersection, we can derive various nice properties of the walks that contribute to these events.

\begin{lem} \label{lem:verttranslate}
Recall that for any integer $j$, the vertices $v_{5j}, v_{5j+1}, v_{5j+2}, v_{5j+3}, v_{5j+4}$ correspond to the vertices at a cross-section of height $h_{-\ell+ j}$ for $\ell:=\frac{\sqrt{\frac{2}{3}n \log^{(2)} n}}{8 \tilde{j}_3(n)}$. Under the event $\mathcal{FE}^{f-1}$, for times $[t_{5j},t_{5j+4}]$ and $ [t_{5j+5},t_{5j+9}]$, it must necessarily be the case that the walks $S_{[t_{5j},t_{5j+4}]}$ and $S_{[t_{5j+5},t_{5j+9}]}$ are completely disjoint. Namely, for any $s_1 \in [t_{5j},t_{5j+4}] $ and $s_2 \in [t_{5j+5},t_{5j+9}]$, we have,
\begin{equation} \label{eq:zfluc}
 |S^z_{s_1} - S^z_{s_2}| \ge (1-3 \delta)\frac{4 (\tilde{j}_3(n))^2}{\sqrt{\frac{2}{3}n \log^{(2)} n}}. 
\end{equation}

\end{lem}
\begin{proof}
We combine the estimates about $|S_{t_{i+1}}^z -v_{i+1}^z| $ when $\mathcal{FE}_i$ corresponds to a vertical edge and estimates about $\max|S_l^z - S_{t_i}^z|$ from $\mathcal{FE}_i$ when $i$ does not correspond to a vertical edge. 
We can ensure that for times $t \in [t_{5j}, t_{5j+4}]$, we have
\begin{equation*}\
|S_t^z - h_{-\ell + j}| \le \delta \frac{4 (\tilde{j}_3(n))^2}{\sqrt{\frac{2}{3}n \log^{(2)} n}} + 4  \sqrt{ \tilde{j}_3(n) \sqrt{\frac{2n}{3\log^{(2)} n}}} \log^{(4)} n.
\end{equation*}

Notice that the first term on the right-hand side is of order $\frac{k_n^2 \sqrt{n \log^{(2)} n}}{(\log^{(3)} n)^2}$ while the second term is of order $\sqrt{k_n \frac{n}{\log^{(3)} n}}\log^{(4)} n$; this latter term is significantly smaller than the first. 
In particular, this means that for times $[t_{5j},t_{5j+4}]$ and $ [t_{5j+5},t_{5j+9}]$, it must necessarily be the case that the walks $S_{[t_{5j},t_{5j+4}]}$ and $S_{[t_{5j+5},t_{5j+9}]}$ are completely disjoint. Recall that the closest vertical distance they have from each other can be bounded from below by $(1-3 \delta)\frac{4 (\tilde{j}_3(n))^2}{\sqrt{\frac{2}{3}n \log^{(2)} n}}$. 

\end{proof}

Secondly, we also have some control over the movement of $S_t$ along horizontal edges.
\begin{lem} \label{lem:probFE}
Let $v_i \to v_{i+1}$ be an edge along one of the square cross-sections of the sphere. Let $j\in\{x,y\}$ denote the coordinate direction of this edge.  In the event $\mathcal{FE}^{f-1}$, we must have that,
\begin{equation} \label{eq:flucxy}
(1- 3 \delta) \tilde{j}_3(n) < |S^j_{t_i} - S^j_{t_{i+1}}| < (1+ 3 \delta) \tilde{j}_3(n).
\end{equation}
As a consequence of this as well as equation \eqref{eq:zfluc} we can deduce that,
\begin{equation*}
\mathbb{P}(\mathcal{FE}^{f-1}) \ge \exp\left[-(1-\kappa)(1+4 \delta)^2 \left(\log^{(2)} n + \sqrt{\frac{3 \log^{(2)} n}{2n}} \tilde{j}_3(n) \right)\right].
\end{equation*}
\end{lem}

\begin{proof}
The proof of equation \eqref{eq:flucxy} is the same as that of equation \eqref{eq:zfluc}.

To determine the probability of the event $\mathcal{FE}^{f-1}$,
we appeal to the Markov property of the random walk. First, notice that,
\begin{equation} \label{eq:MarkovProp}
\mathbb{P}(\mathcal{FE}^i) = \sum_{p\in \mathbb{Z}^3} \mathbb{P}(\mathcal{FE}_i|S_{t_i}=p,\mathcal{FE}^{i-1}) \mathbb{P}(\mathcal{FE}^{i-1},S_{t_i} =p).
\end{equation}
Due to the equations \eqref{eq:flucxy} or \eqref{eq:zfluc},  (conditional on $S_{t_i}=p$ and $\mathcal{FE}^{i-1})$, we must have that $(1- 3 \delta)\|v_{i+1} -v_i\| \le |S^e_{t_{i+1}} - S^e_{t_i}| \le (1+3 \delta)\|v_{i+1} -v_i\|$ for the appropriate coordinate direction $e$. Thus, for all choices of the point $p$ where the probability $\mathbb{P}(\mathcal{FE}^{i-1},S_{t_i} =p)$ is non-zero, we must have that the conditional probability of ensuring the events $\mathcal{TE}_i$ is at least,
$$
\mathbb{P}(\mathcal{TE}_i| S_{t_i}=p, \mathcal{FE}^{i-1}) \ge \exp\left[-(1-\kappa) (1+ 3 \delta)^2 \|v_{i+1} - v_i\|{ \sqrt{\frac{3 \log^{(2)} n}{2n}}}\right].
$$

Now,  we follow the logic of the first part of the proof of Lemma \ref{lem:TransverseFluc} so that we can deal with the transverse fluctuations of $\max|S^k_l - S^k_{t_i}|$ in directions orthogonal to the edge $v_i \to v_{i+1}$. Namely, we can apply the reflection principle to see that the probability can be bounded as,
\begin{equation*}
\mathbb{P}\left(\mathcal{BE}_i^c | \mathcal{TE}_i, S_{t_i}=p, \mathcal{FE}^{i-1} \right) \le C \exp[-(1-\kappa) (\log^{(4)} n)^2].
\end{equation*}

Thus,
\begin{equation} \label{eq:removbad}
\begin{aligned}
\mathbb{P}(\mathcal{FE}_i| S_{t_i}=p, \mathcal{FE}^{i-1})& = \mathbb{P}(\mathcal{TE}_i |S_{t_i}=p, \mathcal{FE}^{i-1})(1- \mathcal{P}(\mathcal{BE}_i^c|\mathcal{TE}_i,S_{t_i}=p, \mathcal{FE}^{i-1}))\\& \ge \exp\left[-(1-\kappa)(1+4 \delta)^2\|v_{i+1} - v_i\|{ \sqrt{\frac{3 \log^{(2)} n}{2n}}}\right], 
\end{aligned}
\end{equation}
whenever $\mathbb{P}(\mathcal{FE}^{i-1}, S_{t_i}=p)$ is non-zero.

Iterating the bound in equation \eqref{eq:MarkovProp} will give us the desired bound on $\mathbb{P}(\mathcal{FE}^{f-1})$.
\end{proof}

Now, given that the event $\mathcal{FE}^{f-1}$ holds, we can now detail further properties of the Green functions on each of the segments.  Consider each $i$ that corresponds to some path $v_i \to v_{i+1}$ that is  part of a cross-sectional square; this path is associated with times $t_i \to t_{i+1}$. Let  $e_i$ be the unit vector with orientation in the direction of the path $v_i \to v_{i+1}$.  We say that a time $l \in [t_i,t_{i+1}]$ is a good time if the following is true.
We have the event,
\begin{equation*}
\mathcal{GF}_{i,l}:= \left\{ \sum_{t=t_{i}}^{t_{i+1}} G(S_t-S_{l}) \le \frac{(1+ 20\delta)}{1-\kappa} \frac{n}{h_3(n)}  \right\}.
\end{equation*}
This event ensures that the sum of the Green's function estimates along the segment containing $S_i$ are not very large.
We also define the event,
\begin{equation*}
\begin{aligned}
\mathcal{GL}_{i,l}:= & \bigcap_{t \in \mathcal{T}} \bigg\{  (1+ 10\delta)  \frac{|\langle S_{t_{i+1}} - S_{t_i}, e_i\rangle| }{t_{i+1} -t_i}  \ge \frac{\langle S_t - S_l,  e_i\rangle }{t-l} \ge (1- 10\delta) \frac{|\langle S_{t_{i+1}}-S_{t_i},e_i \rangle|}{t_{i+1} -t_i} \bigg\}, \\
& \mathcal{T}:=   [t -  (\log^{(2)} n)^{-1}(t_{i+1} -t_i), t+  (\log^{(2)} n)^{-1}(t_{i+1} -t_i)]^c \cap  [t_i,t_{i+1}].
\end{aligned}
\end{equation*}
This event ensures the random walk grows linearly in the coordinate direction $e_i$  around the point $S_l$. Ideally, we would want both of the events $\mathcal{GF}_{i,l}$ and $\mathcal{GL}_{i,l}$ to hold at all times $l$ that belong to the segment $[t_i,t_{i+1}]$. Unfortunately, it is not possible to ensure that these events are held at all times; instead, we will try to ensure that these events hold for most of the segments. Indeed, we define the event,
\begin{equation*}
\mathcal{SG}_i:=\{|\{l :(\mathcal{GL}_{i,l} \cap \mathcal{GF}_{i,l})^c \text{ occurs} \} | \le \delta |t_{i+1} -t_i| \}.
\end{equation*}

In other words, if the event $\mathcal{SG}_i$ holds, at most a $\delta$ fraction of the number of points along a segment is `bad' in the sense that either $\mathcal{GL}_{i,l}^c$ or $\mathcal{GF}_{i,l}^c$ holds.

We have the following lemma.
\begin{lem} \label{lem:bndprobFESG}
Let $i$ be an index such that $v_i \to v_{i+1}$ is an edge that traverses along a cross-sectional square of the  cube. 
Consider $p$ such that $\mathbb{P}(\mathcal{FE}^{i-1},S_{t_i}=p)$ is non-zero. 
There is some $\zeta_n$ with $\lim_{n \to \infty} \zeta_n =0$ such that, 
\begin{equation*}
\mathbb{P}(\mathcal{SG}_i^c|\mathcal{FE}_i, S_{t_i}=p, \mathcal{FE}^{i-1}) \le \zeta_n.
\end{equation*}

As a consequence,  we can deduce that,
\begin{equation}\label{eq:coversquare}
\mathbb{P}\left(\mathcal{FE}^{f-1} \bigcap_{i=0}^{f-1} \mathcal{SG}_i\right) \ge \exp\left[-(1-\kappa)(1+5 \delta)^2 \left(\log^{(2)} n  + \sqrt{\frac{3 \log^{(2)} n}{2n}} \tilde{j}_3(n) \right)\right],
\end{equation}
for $n$ sufficiently large.
\end{lem}
\begin{proof}
By appealing to the logic of Corollary \ref{cor:Inmiddle}, we can determine that the probability of $\mathcal{GL}_{i,l}^c $ conditional on $\mathcal{FE}_i$, $S_{t_i}=p$, $\mathcal{FE}^{i-1}$ is less than $\frac{\zeta^1_n}{2}$  for some $\zeta^1_n$ such that $\lim_{n \to \infty}\zeta^1_n =0 $. As in Proposition \ref{prop:upperboundcap}, we can appeal to Markov's inequality  to deduce that with probability at most $\frac{\zeta^1_n}{\delta}$, that the number of vertices $l$ that satisfy $\mathcal{GL}_{i,l}^c$ is greater than $\frac{\delta}{2}|t_{i+1} - t_i|$.

 \cite[Lemma 3.1]{DemboOkada} allows us to deduce that $\mathbb{P}(\mathcal{GF}_{i,l}^c|\mathcal{FE}_i, S_{t_i}=p,\mathcal{FE}^{i-1}) \le \frac{\zeta^2_n}{2}$ for some $\zeta^2_n$ such that $\lim_{n \to \infty} \zeta^2_n = 0$. Again, by appealing to the Markov Property, we can deduce that with probability at most $\frac{\zeta_n^2}{\delta}$, the number of vertices $i$ that satisfy $\mathcal{GF}_{i,l}^c$  is greater than $\frac{\delta}{2}|t_{i+1} -t_i|$. 

By setting $\zeta_n= \frac{\zeta_n^1 + \zeta_n^2}{\delta}$, we see that,
\begin{equation*}
\mathbb{P}(\mathcal{SG}_i^c|\mathcal{FE}_i, S_{t_i}=p, \mathcal{FE}^{i-1}) \le \zeta_n.
\end{equation*}

We can then follow the method of Lemma \ref{lem:probFE} by iterating the Markov property as in equation \eqref{eq:MarkovProp} and the subsequent equation \eqref{eq:removbad} to deduce equation \eqref{eq:coversquare}.  
\end{proof}

If the event $\mathcal{FE}^{f-1} \bigcap_{i=0}^{f-1} \mathcal{SG}_i$ holds, then we have good estimates of the capacity if we restrict to the subset of those vertices that satisfy the nice properties outlined in $\mathcal{GF}_{i,l}$ and $\mathcal{GL}_{i,l}$.  Indeed, we have the following theorem.

\begin{thm}\label{thm:lowerbndicap}
Consider a walk that satisfies the event $\mathcal{FE}^{f-1} \bigcap_{i=0}^{f-1} \mathcal{SG}_i$. In this case, we see that there is some constant $C_{\delta}$ with $C_{\delta} \to 0$ as $\delta \to 0$, such that,
\begin{equation*}
\ca(S_{[0,t_f]}) \ge (1-\kappa) (1- C_{\delta}) h_3(n).
\end{equation*}

\end{thm}
\begin{proof}
 As stated in Lemma \ref{lem:uprlwrcap}, it must be the case that for any set $S$ of points, we have that,
\begin{equation} \label{eq:caplwrbnd}
\ca(S) \ge \frac{|S|}{\max_{s \in S} \sum_{s' \in S} G(s-s')}.
\end{equation}
As we have mentioned before stating the theorem, we will apply the inequality above when we consider the subset of those points $S_{l}$ that satisfy the events $\mathcal{GF}_{i,{ l}}$ and $\mathcal{GL}_{i,{l}}$. We call the collection of such indices $\mathcal{L}.$ Now, since the event $\mathcal{SG}_i$ holds for each $i$, we must have that the set $\mathcal{L}$ contains at least $(1-\delta) t_f$ elements. Thus, the size of $|\mathcal{L}|$ is at least $(1- \delta) n$ on $\mathcal{GF}_{i,{ l}}$. We remark here that the set $\mathcal{L}$ does not include times that correspond to a traversal along a vertical segment; we only care about the traversal along one of the segments on a cross-sectional plane.  

Now, we bound the denominator. The sum $\sum_{{ l'} \in \mathcal{L}} G(S_{l}-S_{l'})$ comes from two parts. The first is the local contribution along the same segment that contains the point $S_l$. The other contribution is from different segments.  Because of the event $\mathcal{GF}_{i,l}$, we can ensure that the contribution of the Green's function from the local segment can be bounded by $\frac{1 + 20 \delta}{1-\kappa} \frac{n}{h_3(n)}$. 
The contribution to the Green's function sum from points of distance of order $j_3(n)$ from  $S_l$  can be bounded by $\frac{n}{j_3(n)} \ll \frac{\sqrt{n \log^{(2)} n}}{\log^{(3)} n}$.  

It suffices to control the contribution to $\sum_{l' \in \mathcal{L}} G(S_l-S_{ l'})$ from the points $S_{l'}$ that are not in the same segment as $S_{l}$ but are not of a distance of order $j_3(n)$. Suppose that the point $S_{l}$  corresponds to an edge $v_{5(m+\ell)+i} \to v_{5(m+\ell)+i+1}$ that is on the cross section at height $h_m$. Then, the points $S_{l'}$ that we are most concerned about are those that correspond to segments $v_{5(m' + \ell) +i} \to v_{5(m'+\ell) + i+1}$ that belong to height $h_{m'}$ for $m' \ne m$  and are obtained as a vertical translation of the original segment $v_i \to v_{i+1}$. 

Fix a point $S_l$ on segment $v_{5(m+\ell)+i} \to v_{5(m+\ell)+(i+1)}$;  call the coordinate direction of this segment $d$, and call the unit vector that has the orientation of this edge $e_i$ (note that this edge does not depend on the value of $m$). We now define $S_{l^+}$ and $S_{l^-}$ as follows. We let $S_{l^+}$ be the point in $\mathcal{L}$ and on the translated segment  $v_{5(m'+\ell)+i} \to v_{5(m'+\ell)+i+1}$ at height $h_{m'}$  such that $\langle S_{l^+} - S_{l}, e_i \rangle $ is smallest but still positive; we let $S_{l^-}$ be corresponding point  on the same segment such that $\langle S_{l^-} - S_{l}, e_i \rangle $ is largest but still negative. Let $\tilde{l}$ be another index in $\mathcal{L}$ that still belongs to this segment.  
Lemma \ref{lem:verttranslate} ensures that 
$$|S_l^z - S_{\tilde{l}}^z |\ge (m - m' - 3 \delta) \frac{4 (\tilde{j}_3 (n))^2}{\sqrt{\frac{2}{3} n \log^{(2)} n}}.$$  Furthermore, the condition on $\mathcal{GL}_{5(m'+\ell)+i , l^{\pm}}$ would also imply the linear growth,
$$
|S_l^d - S^d_{l^+ + j}|  \ge j(1-\kappa) \sqrt{\frac{2\log^{(2)} n}{3n}}(1-10 \delta)(1-3 \delta), |j| \ge (\log^{(2)} n)^{-1} \frac{1}{1-\kappa}\sqrt{\frac{3n}{2\log^{(2)} n}} \tilde{j}_3(n). 
$$
We have similar estimates if we consider times of the form $l^{-} - j$. 

Recall that $\tilde{j}_3(n):= K k_n \sqrt{n \log^{(2)} n}(\log^{(3)} n)^{-1}$. 
Thus, we have that,
\begin{equation} \label{eq:compsum}
\begin{aligned}
&\sum_{\substack{l' \in \mathcal{L}\\ l' \in [t_{5(m'+\ell) + i}, t_{5(m'+ \ell)+(i+1)}]}} G(S_l-S_{l'}) \le  \frac{3}{2\pi}\frac{(\log^{(2)} n)^{-1} \frac{2}{1-\kappa}\sqrt{\frac{3n}{2\log^{(2)} n}} \tilde{j}_3(n)}{(m - m' - 3 \delta) \frac{4 (\tilde{j}_3 (n))^2}{\sqrt{\frac{2}{3} n \log^{(2)} n}}}\\
&+\frac{3}{\pi} \sum_{j=0}^{ \frac{1}{1-\kappa}\sqrt{\frac{3n}{2\log^{(2)} n}} \tilde{j}_3(n)} \frac{1}{\sqrt{\left|(m - m' - 3 \delta) \frac{4 (\tilde{j}_3 (n))^2}{\sqrt{\frac{2}{3} n \log^{(2)} n}} \right|^2 +\left|j(1-\kappa) \sqrt{\frac{2\log^{(2)} n}{3n}}(1-10 \delta)(1-3 \delta)\right|^2}}\\
& \le C \frac{\sqrt{n}\log^{(3)} n}{(m-m')  k_n(\log^{(2)} n)^{3/2}} + C  \sqrt{\frac{n}{\log^{(2)} n}}\int_{0}^1 \frac{\text{d}j}{\sqrt{j^2 +  \frac{k_n^2 (m-m')^2 }{(\log^{(3)} n)^2}}}\\
&\le  C \frac{\sqrt{n}\log^{(3)} n}{(m-m')k_n (\log^{(2)} n)^{3/2}} + C  \sqrt{\frac{n}{\log^{(2)} n}} \log\left| \frac{\log^{(3)} n}{ k_n|m-m'|} + \frac{\sqrt{(\log^{(3)} n)^2 + k_n^2(m-m')^2}}{k_n(m-m')} \right|.
\end{aligned}
\end{equation}
In the third line, we approximated the discrete sum by an integral, and  $C$ is some universal constant.

We now need to sum up the quantity on the last line of \eqref{eq:compsum} over all $|m-m'| \le 2 \frac{\log^{(3)} n}{k_n}$. 
Now, if we sum up the first term on the last line of \eqref{eq:compsum} in $m$, we see that the sum will be bounded by $$ \frac{\sqrt{n} \log^{(3)} n \log^{(4)} n}{k_n (\log^{(2)} n)^{3/2}} \ll \frac{n}{h_3(n)}.$$  For the second term on the last line, we can divide it into two regions based on whether $|m-m'| \le \frac{\log^{(3)} n}{k_n}$ or $|m-m'| \ge \frac{\log^{(3)} n}{ k_n}$.  In the first case, we can bound the term inside the logarithm by $\log\left| \frac{3 \log^{(3)} n}{ k_n|m- m'|} \right|$.
We also see that,
\begin{equation*}
\begin{aligned}
&\sum_{|m-m'|=1}^{\frac{\log^{(3)} n}{k_n}} \log \left| \frac{3 \log^{(3)} n}{k_n|m-m'|} \right| \le \frac{\log^{(3)} n}{k_n} \log\left| \frac{3 \log^{(3)} n}{k_n} \right|  - \sum_{|m-m'|=1}^{\frac{\log^{(3)} n}{k_n}} \log|m-m'|\\
&\le \frac{\log^{(3)} n}{k_n} \log\left| \frac{3 \log^{(3)} n}{k_n} \right| - \frac{\log^{(3)} n}{k_n} \log \frac{\log^{(3)} n}{k_n} + \frac{\log^{(3)} n}{k_n} \le C \frac{\log^{(3)} n}{k_n}.
\end{aligned}
\end{equation*}

We also know that $\sqrt{1 + x^2} \le 1+ 2x$ when $0 \le x \le 1$ and $\log|1+x| \le Cx$ when $0 \le x \le 3$. As a consequence, when $|m-m'| \ge \frac{\log^{(3)} n}{k_n}$, we can bound the term inside the logarithm of the second term in the last line of \eqref{eq:compsum} by $ C \frac{\log^{(3)} n}{k_n |m- m'|}$. Summing up, we get,
$$
\sum_{|m-m'| = \frac{\log^{(3)} n}{k_n}}^{2\frac{\log^{(3)} n}{k_n}} C \frac{\log^{(3)} n}{k_n|m-m'|} \le C \frac{\log^{(3)} n}{k_n}. 
$$
Since we know that,
\begin{equation*}
\sqrt{\frac{n}{\log^{(2)} n}}\frac{\log^{(3)} n}{k_n}  \ll  \frac{n}{h_3(n)},
\end{equation*}
we must have that,
\begin{equation*}
\sum_{l' \in \mathcal{L}} G(S_l-S_{l'}) \le \left[\frac{1 + 20 \delta}{1-\kappa} + o(1) \right] \frac{n}{h_3(n)}, \quad \forall l \in \mathcal{L}.
\end{equation*}
We can now apply equation \eqref{eq:caplwrbnd} to get the desired lower bound on the capacity.
\end{proof}

We can now finish the proof of (i) in Theorem \ref{m2}.
\begin{proof}[Proof of (i) in Theorem \ref{m2}]

Due to the estimate of Lemma \ref{lem:bndprobFESG}, we can ensure that the event $\mathcal{FE}^{f-1} \cap_{i=0}^{f-1} \mathcal{SG}_i$ would occur infinitely often when considering the collection of times $n^n$ by the same argument as in Theorem \ref{thm:neinfoft}. By Theorem \ref{thm:lowerbndicap}, we would have the desired lower bound on the capacity of the random walk as we set $\delta\to 0$ and $\kappa \to 0$. This completes the proof.

\end{proof}

\subsection{Proof of (ii) in Theorem \ref{m2}}

\begin{proof}[ Proof that $\mathbb{P}(A_n^1 \cap B_n \quad \text{i.o.})=0$]
First, we show $\mathbb{P}(A_n^1 \cap B_n \quad \text{i.o.})=0$.
 To do this, we will start by considering the following. Define $t_n$ to be the collection of times $q^n$ with $q>1$. Our first claim is that if there were an infinite sequence of times $n_1,n_2,\ldots, n_i,\ldots$ such that for all $n_i$, we have 
\begin{equation} \label{eq:notA1}
\begin{aligned}
&R_{n_i} \ge (1-\epsilon/2) h_3(n_i),\\
&(1-\epsilon/2) j_3(n_i)\le D_{n_i} \le j_3(n_i),
\end{aligned}
\end{equation}
then, for $q$ sufficiently small, we would be able to find an infinite sequence of times $t_{k_1}, t_{k_2},\ldots, t_{k_j},\ldots$ such that,
\begin{equation}\label{eq:notA1spec}
\begin{aligned}
&R_{t_{k_j}} \ge (1-\epsilon/2) h_3(t_{k_j}),\\
& D_{t_{k_j}} \le j_3(t_{k_j}).
\end{aligned}
\end{equation}

Notice first that if $t_{j}\le n_i \le t_{j+1}$ for some $n_i$, then  the monotonicity and subadditivity of the capacity implies that,
\begin{equation*}
R_{n_i} -R_{t_{j}} \le R_{t_{j+1}} - R_{t_j} \le R_{[t_j,t_{j+1}]}.
\end{equation*}

From our $\limsup$ bounds on the maximum increase of the capacity from Section \ref{sec:Strassenlimsuplil}, namely, an adaptation of Lemma \ref{lem:Uuprbnd}, we must have that eventually almost surely,  all $R_{[t_{j},t_{j+1}]} \le (1+\kappa)  h_3(t_{j+1} -t_j) \le (1+\kappa)  \sqrt{q-1} h_3(t_j)$. Here, $\kappa>0$ is a small parameter that does not need to stay the same from line to line. Note that as we take $q \to 1$, the relative increase in the capacity becomes smaller and smaller relative to $h_3(t_j)$.

Thus, for $q$ sufficiently small (depending only on our choice of $\epsilon$ from equation \eqref{eq:notA1}), we notice that if we let $t_{k_i}$ be the value of $t_k$ that is closest to but still less than $n_i$, we have,
\begin{equation*}
R_{t_{k_i}} \ge R_{n_i} - (1+\kappa) ( \sqrt{q-1}) h_3(t_{k_i}) \ge (1- \epsilon/2) h_3(n_i) - (1+\kappa)(\sqrt{q-1}) h_3(t_{k_i}) \ge (1-\epsilon) h_3(t_{k_i}).
\end{equation*}
On the line above, we took $q$ sufficiently small, and we used that $h_3(t_{k_i}) \le h_3(n_i)$.

Thus, it suffices to show that the equations in \eqref{eq:notA1spec} do not occur infinitely often. Let $\delta>0$ be some small parameter and $\eta>0$ be another parameter chosen to be small based on the value of $\delta$.  We define the radius $r_n:= \phi(t_n)/(2 \eta(1+\delta)^2 \log^{(2)}t_n)$ and $b_n:= 2 \eta(1+\delta)^6 \log^{(2)}t_n $. As we have done in the proof of Lemma \ref{lem:Uuprbnd} and in \cite[Section 3.2]{DemboOkada}, almost surely, eventually all the ranges $[S_0,S_{t_n}]$ will be contained in the union $\mathcal{C}^{*}_{t_n}(r_n)$ of $b_n$ balls with radius $r_n$ such that the distance between the centers of any two successive balls is less than $r_n$. Namely, we can define the series of stopping times 
$$
\hat{T}_{i+1}:= \inf \{t> \hat{T_i}: \|S_{t} - S_{\hat{T}_i}\| \ge  r_n-1 \}.
$$
Then,
$$
\mathcal{C}^{*}_{t_n}(r_n) = \cup_{i=0}^{b_n-1} \mathcal{B}(S_{\hat{T}_i}, r_n),
$$
where we recall that $\mathcal{B}(x,r)$ represents the ball of radius $r$ centered at $x$.

We must have that $R_{t_n} \le \ca(\mathcal{C}^*_{t_n}(r_n))$. Furthermore, observe that under  the condition that $D_{t_n} \le j_3(t_n)$, we see that all of these balls have centers inside the ball of radius $j_3(t_n)$.  We will show that constraining all of these balls to have centers inside the ball of radius $j_3(t_n)$ will cause the capacity to be less than $h_3(t_n)$ by at least a constant factor. This will give us a proof that the inequalities in equation \eqref{eq:notA1spec} cannot occur infinitely often and, thus, $\mathbb{P}(A_n^1 \cap B_n \quad\text{i.o.}) =0$, as desired.

To bound the capacity of $\mathcal{C}^*_{t_n}(r_n)$, we need to derive a lower bound on $\min_{x \in \mathcal{C}^*_{t_n}(r_n)} \sum_{y \in \mathcal{C}^*_{t_n}(r_n)} G(x- y)$. (The sum over the union $\mathcal{C}^*_{t_n}(r_n)$ should be understood as first a sum over $0 \le i \le b_n$ and then a sum of $y \in \mathcal{B}(S_{\hat{T}_i},r_n)$. Thus, points can be repeated.)   Let $\delta'>0$ be some small parameter. First, let us consider $ \delta' b_n \le i \le (1-\delta') b_n$ and consider $x \in \mathcal{B}(S_{\hat{T}_i},r_n)$. We can show that,
\begin{equation*}
\begin{aligned}
&\sum_{|j-i| \le \delta' b_n} \sum_{y \in \mathcal{B}(S_{\hat{T}_j},r_n)} G(x-y) \ge 2  \frac{3}{2\pi}(1+o(1)) \sum_{m= 3}^{\delta' b_n} |\mathcal{B}_{r_n}| \frac{1}{m r_n}\\
&\ge \frac{3}{\pi} (1+ o(1))|\mathcal{B}_{r_n}| \frac{1}{r_n} \log \frac{\delta' b_n}{3} \ge \frac{3}{\pi}(1+o(1)) |\mathcal{B}_{r_n}| \frac{1}{r_n} \log b_n \ge (1- \epsilon') |\mathcal{B}_{r_n}| b_n h_3(n)^{-1}.
\end{aligned}
\end{equation*}
$\epsilon'$ in the last line can be taken to be arbitrarily close to $0$.
Here, we used the fact that the maximum distance between any point of $\mathcal{B}(S_{\hat{T}_i},r_n)$ and $\mathcal{B}(S_{\hat{T}_j},r_n)$ is less than $|j - i+2| r_n$ and that the Green's function is $G(x) = \frac{3}{2\pi}(1+o(1)) \frac{1}{\|x\|}$. In the last inequality, we used that $\log \frac{\delta' b_n}{3} \ge (1+ o(1)) \log b_n$.

One can observe that by comparison to the Green's function bound obtained in Section \ref{sec:limsup} that the main contribution to the Green's function around a point $x \in \mathcal{B}(S_{\hat{T}_i}, r_n)$ when obtaining the standard $\limsup$ behavior comes from these close intervals and that the contribution from balls that are further away are much smaller.

In the case that we have bounded $D_n$ for the range of the random walk, it is no longer the case that the balls that are further away have a negligible contribution to the Green's function. Indeed, we now see that,
\begin{equation*}
\sum_{|j-i| \ge \delta' b_n} \sum_{y \in \mathcal{B}(S_{\hat{T}_j},r_n)} G(x - y) \ge  c b_n |\mathcal{B}_{r_n}| (j_3(n))^{-1}.
\end{equation*}
Here, we used the fact that the maximum distance any two points in $\mathcal{C}^*_{t_n}(r_n)$, given that $D_n \le j_3(n)$, is $j_3(n)^{-1}$.
Thus, we see that,
\begin{equation*}
\min_{x \in \bigcup_{i = \delta' b_n}^{1-\delta' b_n} \mathcal{B}(S_{\hat{T}_i},r_n)} \sum_{y \in \mathcal{C}^*_{t_n}(r_n)} G(x- y) \ge (1- \epsilon') |\mathcal{B}_{r_n}| b_n h_3(n)^{-1} + c b_n |\mathcal{B}_{r_n}|(j_3(n))^{-1}.
\end{equation*}
We remark additionally that, by similar computations, we can argue that,
\begin{equation*}
\min_{x \in \bigcup_{i=0}^{\delta' b_n} \mathcal{B}(S_{\hat{T}_i},r_n)} \sum_{y \in \bigcup_{j=0}^{\delta' b_n} \mathcal{B}(S_{\hat{T}_j},r_n)} G(x-y) \ge \frac{(1-\epsilon')}{2}|\mathcal{B}_{r_n}| b_n h_3(n)^{-1}.
\end{equation*}
Thus, we see that we can bound the capacity of the `ends' of the chain of balls by using Lemma \ref{lem:uprlwrcap} to derive, 
\begin{equation*}
\ca\left( \bigcup_{i=0}^{\delta' b_n} \mathcal{B}(S_{\hat{T}_i},r_n) \cup \bigcup_{i=(1-\delta')b_n}^{b_n} \mathcal{B}(S_{\hat{T}_i},r_n) \right) \le \frac{2 \delta' b_n |\mathcal{B}_{r_n}|}{\frac{(1-\epsilon')}{2}|\mathcal{B}_{r_n}| b_n h_3(n)^{-1}} \le C \delta' h_3(n).
\end{equation*}

Now, we can use Lemma \ref{lem:uprsumcap} to finally derive,
\begin{equation*}
\ca(\mathcal{C}^*_{t_n}(r_n)) \le  C \delta' h_3(n) + \frac{b_n|\mathcal{B}_{r_n}|}{ (1- \epsilon') |\mathcal{B}_{r_n}| b_n h_3(n)^{-1} + c b_n |\mathcal{B}_{r_n}| (j_3(n))^{-1}} \le (1-\epsilon/2) h_3(n),
\end{equation*}
where the right hand side is attained if one chooses $\delta'$ and $\epsilon'$ sufficiently close to $0$. Thus, we have determined that the collection of events from equation \eqref{eq:notA1spec} cannot happen infinitely often. Thus, we have established our initial assertion that $\mathbb{P}(A_n^1 \cap B_n \quad \text{i.o.}) = 0$. 
\end{proof}

We turn to showing $\mathbb{P}(A_n^2 \cap B_n \quad \text{i.o.})=0$.  This requires a more involved argument. We will first start by presenting the heuristics. In order for the random walk to have the full capacity of the ball, it must be the case that the  random walk must cover the full boundary of the ball. Thus, in order to show that the random walk cannot obtain the full capacity of the ball, we will show that there is some part of the boundary that cannot be fully covered by the random walk.

More rigorously, we will let $E^1(x, \eta j_3(n))$ with $\|x\| = j_3(n)$ denote the caps of the sphere of radius $\mathcal{B}_{j_3(n)}$ for small $\eta>0$. Namely, $E^1(x,\eta j_3(n)) = \mathcal{B}_{j_3(n)} \cap \mathcal{B}(x,\eta j_3(n)) $. The capacity of such a cap is of order $\eta j_3(n)$. By the pigeonhole principle, we will show that there is some cap along the boundary of $\mathcal{B}(x,j_3(n))$ that contains a limited fraction of the random walk. Furthermore, we will also show that $\ca(S_{[1,n]} \cap E^1(x,\eta j_3(n)) \le O(\eta^2 j_3(n))$, which is too small to get the full capacity of the cap. Thus, the random walk inside this cap does not attain the full capacity of the cap.

We will argue that if the random walk does not obtain the full capacity inside of this cap, then the random walk cannot obtain the full capacity of the ball of radius $j_3(n)$. To do this, we develop the following lemmas. In Lemma \ref{cap1}, we derive an explicit formula for how the capacity increases from $\ca(A)$ to $\ca(A \cup B)$ as we add the set $B$ to the set $A$. We then use this formula in Lemma \ref{cap2} to explicitly compute the potential increase of the capacity when we add the complement of the cap $E^1(x,\eta j_3(n))$ inside the ball $\mathcal{B}_{j_3(n)}$, ($V_2 = E^1(x,\eta j_3(n))^c \cap \mathcal{B}_{j_3(n)} $) to the part of the random walk that lies inside the cap $E^1(x,\eta j_3(n))$ ($V_1 = E^1(x,\eta j_3(n)) \cap S[1,n])$ under the condition that the capacity of this latter part, $V_1$, is very small. Lemma \ref{cap2} will show that the capacity of $V_1 \cup V_2$ is strictly smaller by a constant factor than the capacity of $\mathcal{B}_{j_3(n)}$, and thus the capacity of the random walk (which lies inside $V_1 \cup V_2$) must be too small. The following Lemma \ref{cap3} contains a useful probability estimate that is necessary to prove Lemma \ref{cap2}. 

To finally conclude our argument, we will show that there is a cap such that the capacity of the random walk inside this cap will be very small. Then, we can apply Lemma \ref{cap2} and obtain our desired contradiction. In what follows, since the computations do not specifically depend on the radius $j_3(n)$, we will simplify notation by considering the computations on a ball of radius $n$, instead.

\begin{lem}\label{cap1}
For any disjoint finite sets $A$, $B \subset \Z^3$, 
\begin{align*}
    \ca (A \cup B) - \ca (A)
    =\sum_{x\in B} \mathbb{P}^x(T_A=\infty)\mathbb{P}^x(T_{A\cup B}=\infty). 
    \end{align*}
\end{lem}
\begin{proof}
    Let $A=\{x_1, \ldots , x_j\}$ and $B=\{x_{j+1}, \ldots ,x_{j+k}\}$. We assume that $A$ and $B$ are disjoint. 
    Let $z_i:=\mathbb{P}^{x_i}(T_A=\infty)$ and 
    $y_i:=\mathbb{P}^{x_i}(T_{A\cup B}=\infty)$. 
    Note that for $1\le i \le j$, 
    \begin{align*}
        \sum_{l=1}^j G(x_i-x_l) y_l
        = 1- \sum_{l=j+1}^{j+k} G(x_i-x_l)y_l.
    \end{align*}
We have that,
    \begin{equation*}
    \begin{aligned}
    &\sum_{i=1}^j z_i(1 - \sum_{l=j+1}^{j+k}G(x_i-x_l)y_l) = \sum_{i=1}^j z_i \sum_{l=1}^j G(x_i-x_l) y_l
    \\&= \sum_{l=1}^j y_l \sum_{i=1}^j z_iG(x_i-x_l) = \sum_{l=1}^j y_l.
    \end{aligned}
    \end{equation*}
    Hence, we have $\sum_{i=1}^j y_i 
    = \sum_{i=1}^j z_i (1- \sum_{l=j+1}^{j+k} G(x_1-x_l)y_l)$. 
    Therefore, we have
    \begin{align*}
    \ca (A \cup B) - \ca (A)
    =&\sum_{i=1}^{j+k} y_i - \sum_{i=1}^j z_i \\
    = &\sum_{i=1}^j z_i (1- \sum_{l=j+1}^{j+k} G(x_{i}-x_l)y_l)
    + \sum_{i=j+1}^{j+k} y_i - \sum_{i=1}^j z_i \\
    =& \sum_{i=j+1}^{j+k} y_i (1-\sum_{l=1}^j z_l G(x_l-x_i))\\
    =&\sum_{x\in B} \mathbb{P}^x(T_A=\infty)\mathbb{P}^x(T_{A\cup B}=\infty). 
    \end{align*}
\end{proof}

Pick a point $x \in \partial \mathcal{B}_n$, where we recall that $\mathcal{B}_n$ is a ball with a radius $n$. We consider the small ball with a radius of $\eta n$ and center $x$. 
Let $E_n^1(x,\eta n)=E_n^1$ be the intersection between that small ball and $\mathcal{B}_n$. 

\begin{lem}\label{cap2}
Assume that $V_1$ is contained in the subset $E_n^1$ of the ball $\mathcal{B}_n$ and $V_2 = \partial \mathcal{B}_n \cap E_{n}^1$. 
Let $E_n^2:= \mathcal{B}_n \cap (E_n^1)^c$.  
If 
\begin{align}\label{assum:cap}
\ca (V_1) \le (c\eta^2+\xi_n) \ca (\mathcal{B}_n)  
\end{align}
with $\xi_n \to 0$, there exists $f(\eta)>0$ (which will have some dependence on $\eta$) such that
\begin{align*}
    \ca (E_n^2  \cup V_1) \le (1-f(\eta) ) \ca (\mathcal{B}_n).  
    \end{align*}
\end{lem}

\begin{rem}\label{dis*}
There is a point, which occurs in the proof of Lemma \ref{cap3}, where it would be necessary for $\text{dist}(V_1,\partial \mathcal{B}_n)$ to be greater than $\theta n$ for some $\theta >0$. Namely, this is equation \eqref{eq:eisfarfromd}. We remark that this inequality does not depend on the value of $\theta$, only that it is greater than $0$.

From this logic, we then claim that it is sufficient to consider the case that all the points of $S[0,n]$ lie inside the ball $\mathcal{B}_{n(1-g(\eta))}$, where $g(\eta)$ is a constant such that $g(\eta) >0$ and $\ca(\mathcal{B}_n) - \ca(\mathcal{B}_{n(1-g(\eta))}) \le \frac{f(\eta)}{2} \ca(\mathcal{B}_n)$.  Indeed, if we assume that the random walk $S_{[0,n]}$ lies inside the ball $\mathcal{B}_{n(1- g(\eta))}$, then all the points of $S_{[0,n]}$ would be of distance at least $g(\eta)n$ from the boundary of $\mathcal{B}_n$. If furthermore, there was a cap $E^1_n$ in the boundary of $\mathcal{B}_n$ that satisfied the conditions of Lemma \ref{cap2} with $V_1 = E^1_n \cap S[1,n]$, then we would have that $R_n \le \ca(E_n^2 \cup V_1) \le (1- f(\eta)) \ca(\mathcal{B}_n) \le \frac{1 - f(\eta)}{1- \frac{f(\eta)}{2}} \ca(\mathcal{B}_{n(1-g(\eta)}).$ Then, there would still be a gap between $R_n$ and the capacity of the ball $\mathcal{B}_{n(1-g(\eta))}$ that it is contained inside. This will complete our desired assertion.

\end{rem}

\begin{proof}
In the following proof, $C$ and $c$ are constants that need not stay the same from line to line.

By Lemma \ref{cap1}, 
\begin{align*}
    &\ca (V_1 \cup V_2 )-\ca (V_1)\\
    =& \sum_{x \in V_2} \mathbb{P}^x(T_{V_1}=\infty) \mathbb{P}^x (T_{V_1 \cup V_2}=\infty)\\
    =& \sum_{x \in V_2} \mathbb{P}^x(T_{V_1}=\infty) \mathbb{P}^x (T_{V_1 \cup V_2}=\infty, \text{ the first step of the random walk exits the ball }\mathcal{B}_n)\\
    +&\sum_{x \in V_2} \mathbb{P}^x(T_{V_1}=\infty) \mathbb{P}^x (T_{V_1 \cup V_2}=\infty, \text{ the first step of the random walk enters inside of the ball }\mathcal{B}_n).
    \end{align*}
We consider a cylinder such as
\begin{align*}
    \{(x,y,z) \in \RR^d (\text{ or }\Z^d): 
     x^2+y^2 \le c, -a<z<b \}
\end{align*}
for $a,b>0$. 
Note that by the simple computation of the estimate of the hitting time of the Brownian motion, 
\begin{align*}
    \mathbb{P}(\text{A Brownian motion does not exit the cylinder before $T_{CL}$ }) \ge C, 
\end{align*}
where $CL:=\{(x,y,z) \in \RR^d: 
     x^2+y^2 \le c, z=b \}$. 
Hence, by the Skorokhod Embedding (see Lemma 3.1 in \cite{LA91}), we can easily obtain
\begin{align}\label{cylinder}
    \mathbb{P}(\text{ A random walk does not exit the cylinder before hitting at the area } CL \cap \{z: z =b \}) \ge C.
\end{align}

    Let $\mathcal{S}'$ be the surface of the ball whose center is the center of $V_2$, namely, the point $x$, and whose radius is twice of the diameter of $V_2$ (so the radius of $\mathcal{S}'$ would be $ 4 \eta n$).  
    Also, let $\mathcal{S}$ be the surface of the ball whose center is $x$ and whose radius is $ 6\eta n$. 
    With the help of \eqref{cylinder}, 
\begin{align*}
        \mathbb{P}^y(T_{V_2}<T_{\mathcal{S}})
        \le C \mathbb{P}^y(T_{V_2}<T_{\mathcal{S}}, \text{ the random walk hits } V_2 \text{ outside of }\mathcal{B}_n)
    \end{align*}
holds uniformly for $y$ in ${\mathcal{S}'}$. Thus, we have that, 
    \begin{align*}
        \mathbb{P}^z(T_{V_2}<T_{\mathcal{S}})
        \le &\sum_{y\in S'}\mathbb{P}^z(T_{V_2}<T_{\mathcal{S}}, S_{T_{\mathcal{S}'}}=y)\mathbb{P}^y(T_{V_2}<T_{\mathcal{S}})\\
        \le &C \sum_{y\in S'}\mathbb{P}^z(T_{V_2}<T_{\mathcal{S}}, S_{T_{\mathcal{S}'}}=y)\mathbb{P}^y(T_{V_2}<T_{\mathcal{S}}, \text{ the random walk hits } V_2 \text{ outside of }\mathcal{B}_n)\\
        \le & C \mathbb{P}^z(T_{V_2}<T_{\mathcal{S}}, \text{ the random walk hits } V_2 \text{ outside of }\mathcal{B}_n)
    \end{align*}
   holds uniformly for all $z$ in $\mathcal{S}$. 
Indeed, a random walk that hits $V_2$ before hitting ${\mathcal{S}}$ starting from a point in $\mathcal{S}$, must hit a point in $\mathcal{S}'$ before hitting $V_2$. 
    
Hence, by reversing the random walk, we get,
    \begin{equation*} 
    \begin{aligned}
        &\sum_{x\in V_2}  \mathbb{P}^x(T_{V_2}=\infty, S_{T_{\mathcal{S}}}=y) \\
        \le & C \sum_{x\in V_2} \mathbb{P}^x(T_{V_2}=\infty, \text{ the first step of the random walk exits the ball }\mathcal{B}_n, S_{T_{\mathcal{S}}}=y). 
    \end{aligned}
    \end{equation*}
    In addition, if we sum over $y$, we get as a consequence that,
    \begin{equation*}
    \begin{aligned}
        &\sum_{x \in V_2} \mathbb{P}^x(T_{V_2} =\infty)\\
        &\le C \sum_{x \in V_2} \mathbb{P}^x(T_{V_2} = \infty, \text{the first step of the random walk exits the ball } \mathcal{B}_n).
    \end{aligned}
    \end{equation*}

    Therefore, we have 
    \begin{align*}
    &\sum_{x \in V_2} \mathbb{P}^x(T_{V_1}=\infty) \mathbb{P}^x (T_{V_1 \cup V_2}=\infty, \text{ the first step of the random walk enters inside of the ball }\mathcal{B}_n)  \\
    \le &\sum_{x \in V_2}  \mathbb{P}^x (T_{V_2}=\infty, \text{ the first step of the random walk enters inside of the ball }\mathcal{B}_n)\\
    \le & (1-C^{-1}) \ca (V_2). 
    \end{align*}
To derive the last line, we used the following decomposition:
\begin{equation*}
\sum_{x \in V_2} \mathbb{P}^x(T_{V_2} = \infty) = \sum_{x \in V_2} \mathbb{P}^x(T_{V_2} = \infty, \text{first step exits}) + \sum_{ x \in V_2} \mathbb{P}^x(T_{V_2 } = \infty, \text{first step enters}),
\end{equation*}
and hence noting $\ca (V_1 \cup V_2 )-\ca (V_1) \ge (1-
(c  \eta+\xi_n)) \ca(V_1 \cup V_2) \ge (1-c \eta+\xi_n) \ca(V_2)$, 
\begin{equation} \label{eq:addcap1}
     \begin{aligned}
    \sum_{x \in V_2} \mathbb{P}^x (T_{V_1}=\infty) \mathbb{P}^x (T_{V_1 \cup V_2}=\infty, \text{ the first step of the random walk exits the ball }\mathcal{B}_n) 
   \\ \ge  [(1- (c\eta+\xi_n))-(1-C^{-1})] \ca(V_2) \ge c \ca (V_2). 
    \end{aligned}
\end{equation}
We remark that at this point, the constant $c$ that appears above is still a universal constant that does not depend on the value of $\eta$ as long as $\eta$ is small enough.

  In addition, there is a universal constant $K$  and a subset  $V'_2 \subset  V_2$ such that dist$(V'_2,  \partial V_2)\ge \eta n/ K$ and 
 \begin{equation} \label{eq:defV2'}
  \begin{aligned}
    \sum_{x \in V'_2} \mathbb{P}^x(T_{V_1}=\infty) 
    \mathbb{P}^x (T_{V_1 \cup V_2}=\infty, \text{ the first step of the random walk exits the ball }\mathcal{B}_n) 
    \ge  c/2 \ca (V_2). 
    \end{aligned}
\end{equation}
To see this, we let $V_2'$ consist of all of the points of distance at least $\frac{n \eta}{K}$ from the boundary of $V_2$, and see that 
$$
\sum_{x \in V_2 \setminus V_2'} \mathbb{P}^x(T_{V_1} = \infty) \mathbb{P}^x(T_{V_1 \cup V_2} = \infty) \le \sum_{x \in V_2 \setminus V_2'} \mathbb{P}^x(T_{V_2 \setminus V_2'} = \infty) = \ca(V_2 \setminus V_2').
$$
This latter capacity will scale in the form $\frac{n \eta}{h(K)}$ for some function $h(K)$ that will go to $0$ as $K$ goes to $0$ (to the leading order approximation, $V_2 \setminus V_2'$ is like an annulus with outer radius $\eta n$ and inner radius $\eta(1- K^{-1})n$. It suffices to choose $K$ large enough so that $h(K) \le \frac{c}{2}$, where $c$ was the constant from the last inequality of equation \eqref{eq:addcap1}. Thus, this $K$ is a universal constant.

We remark that the point of considering only the points in $V_2'$ since for each point $x \in V_2'$, we can find a ball of radius $O(n \eta)$ such that there will be no intersection between this ball around the point $x$ and the set $E_n^2$. 

    Now, we turn to considering 
    \begin{align*}
    \ca (E_n^2 \cup V_1 \cup V_2 )-\ca (E_n^2 \cup V_1)
    \ge & \sum_{x \in V'_2} \mathbb{P}^x(T_{E_n^2 \cup V_1}=\infty) \mathbb{P}^x (T_{E_n^2 \cup V_1 \cup V_2}=\infty).
    \end{align*}
Combining the following equation for $x\in V'_2$,

    \begin{align*}
    \mathbb{P}^x (T_{V_1 \cup V_2}=\infty, \text{ the first step of the random walk exits the ball }\mathcal{B}_n) 
    \le  h_1(\eta) \mathbb{P}^x (T_{E_n^2 \cup V_1 \cup V_2}=\infty).
    \end{align*}
We remark here that the constant on the right-hand side will have some dependence on $\eta$. We will give a few remarks on how the equation above is derived.

Given that we are considering a random walk  with $T_{V_1 \cup V_2} = \infty$ whose first step exits the ball $\mathcal{B}_n$, we see that the first point that the walk hits the ball $\mathcal{B}(x,\frac{\eta n}{K})$ must be on its hemisphere that lies outside the ball $\mathcal{B}_n$. Let $L$ denote the plane through $x$ that is tangent to the ball $\mathcal{B}_n$; we also let $H_1$ denote the hemisphere of $\mathcal{B}(x,\frac{\eta n}{K})$ that is outside of the ball $\mathcal{B}_n$ and $H_2$ be the hemisphere that intersects $\mathcal{B}_n$.

Notice that the distribution of this first point that hits the ball $\mathcal{B}(x, \frac{\eta n}{K})$ would be the same as the one generated if we considered a random walk with $T_{V_2} = \infty$ whose first step exits the ball $\mathcal{B}_n$.
Thus, at least a $p_1(\eta)$ proportion of the walks such that $T_{V_1 \cup V_2} = \infty$ and whose first step exists the ball $\mathcal{B}_n$ will hit the hemisphere $H_1$ at a point $y$ of distance at least $\frac{n\eta}{4K}$ away from $L$. Of these, walks that hit $y$, there is at least a $p_2(\eta)$ proportion that satisfy $T_{E_n^2 \cup V_1 \cup V_2}= \infty$, and this bound is uniform in $y$.  Thus,
\begin{equation*}
\mathbb{P}^x (T_{V_1 \cup V_2}=\infty, \text{ the first step of the random walk exits the ball }\mathcal{B}_n)  \le \frac{1}{p_1(\eta) p_2(\eta)} \mathbb{P}^x(T_{E_n^2 \cup V_1 \cup V_2}=\infty).
\end{equation*}
Letting $h_1 (\eta):= \frac{1}{p_1(\eta) p_2(\eta)}$, we complete our result.


From Lemma \ref{cap3}, we see that there is some other function $h_2(\eta)$ such that we can obtain $\forall x \in V_2'$, that,
    $
    \mathbb{P}^x(T_{V_1 \cup E_n^2} = \infty) \ge h_2(\eta) \mathbb{P}^x(T_{V_1}= \infty).
    $
We can combine these estimates together to see that,
\begin{equation*}
\begin{aligned}
&\sum_{x \in V'_2} \mathbb{P}^x(T_{E_n^2 \cup V_1}=\infty) \mathbb{P}^x (T_{E_n^2 \cup V_1 \cup V_2}=\infty)\\& \ge h_1(\eta) h_2(\eta) \sum_{x \in V'_2} \mathbb{P}^x(T_{V_1}=\infty) 
    \mathbb{P}^x (T_{V_1 \cup V_2}=\infty, \text{ the first step of the random walk exits the ball }\mathcal{B}_n) \\
    &\ge  \frac{c}{2}h_1(\eta) h_2(\eta) \ca(V_2).
\end{aligned}
\end{equation*}
The last equation was derived from \eqref{eq:addcap1}.
Thus, we have
\begin{equation*}
\ca(E_n^2 \cup V_1 \cup V_2) - \ca(E_n^2 \cup V_1) \ge \sum_{x \in V'_2} \mathbb{P}^x(T_{E_n^2 \cup V_1}=\infty) \mathbb{P}^x (T_{E_n^2 \cup V_1 \cup V_2}=\infty) \ge f(\eta) \ca(\mathcal{B}_n).
\end{equation*}

Here, $f(\eta):= \frac{c}{2} h_1(\eta) h_2(\eta) \frac{\ca(V_2)}{\ca(\mathcal{B}_n)}$. We observe that $\frac{\ca(V_2)}{\ca(\mathcal{B}_n)}$ scales linearly in $\eta$.  Now, since $\ca(E_n^2 \cup V_1 \cup V_2)= \ca(\mathcal{B}_n)$, we have that $\ca(E_n^2 \cup V_1)  \le (1- f(\eta)) \ca(\mathcal{B}_n)$, as desired. 
\end{proof}

\begin{lem}\label{cap3}
Recall the set $V_2'$ from equation \eqref{eq:defV2'}. Let $x \in V_2'$ and $L$ be the tangent plane to $\mathcal{B}_n$ through $x$. Furthermore, assume that there exists some $\theta >0$ such  that $\text{dist}(L,V_1) \ge \theta n$. Then, we have 
\begin{align*}
    \mathbb{P}^x (T_{V_1\cup E_n^2 }=\infty)
    \ge  h(\eta) \mathbb{P}^x(T_{V_1}=\infty ). 
\end{align*}
Here, $h(\eta)$ is a function that could depend on $\eta$ and the universal constant $K$ from the definition of \eqref{eq:defV2'}, but not on $\theta$. 
\end{lem}
\begin{proof}

Consider the ball $\mathcal{B}( x, a_n)$ with $a_n=\eta n/2K$ and $K$ is the universal constant from the definition of \eqref{eq:defV2'} and a plane $L$ through the point $x$ that is tangent to the sphere $\mathcal{B}_n$. This plane divides $\mathcal{B}(x,a_n)$ into two halves $H_1$ and $H_2$. Let $H_1$ be the half of the sphere $\mathcal{B}(x,a_n)$ that intersects the ball $\mathcal{B}_n$. 
Let $D_1$ denote $D_1 :=\partial \mathcal{B}_{a_n} \cap H_1$ and $D_2:= \partial \mathcal{B}_{a_n} \cap H_2$ and $D= D_1 \cup D_2$.  Now, let $E$ be any set that $E \subset H_1$ and that there exists $\theta>0$ such that $\forall e \in E$, $\text{dist}(e, L) \ge \theta n$ (as mentioned in Remark \ref{dis*}). 

We will demonstrate that
\begin{equation} \label{eq:hitD2moreD1}
\mathbb{P}^{x }(T_{\partial \mathcal{B}_{a_n}} = T_{D_2}|T_{E} > T_{\partial {\mathcal{B}_{a_n}}})  
\ge \mathbb{P}^{x}(T_{\partial \mathcal{B}_{a_n}} = T_{D_1}| T_{E} > T_{\partial{\mathcal{B}_{a_n}}}).
\end{equation} 

Note that it suffices to show that,
\begin{equation*}
\mathbb{P}^{ x}(\{T_{\partial \mathcal{B}_{a_n}} = T_{D_2} \} \cap \{ T_{E} > T_{\partial \mathcal{B}_{a_n}} \}) \ge \mathbb{P}^{x}(\{T_{\partial \mathcal{B}_{a_n}}  = T_{D_1} \} \cap\{ T_E > T_{\partial B_{a_n}} \}).
\end{equation*}

We can reflect the walk around the center $x$ to argue that,
\begin{equation*}
\mathbb{P}^{x}(T_{\partial \mathcal{B}_{a_n}} = T_{D_2}) 
= \mathbb{P}^{x}(T_{\partial \mathcal{B}_{a_n}} = T_{D_1}).
\end{equation*}
We only need to argue that,
\begin{equation} \label{eq:uprbndpotential}
\mathbb{P}^{x}(\{T_{\partial \mathcal{B}_{a_n}} = T_{D_2}\} \cap \{ T_{E} < T_{\partial \mathcal{B}_{a_n}} \}) \le \mathbb{P}^{x}(\{ T_{\partial \mathcal{B}_{a_n}} = T_{D_1} \} \cap \{ T_{E}  < T_{\partial \mathcal{B}_{a_n}} \}).
\end{equation}

Since $T_{E}< T_{\partial \mathcal{B}_{a_n}}$, we can decompose both of these walks based on the first point where they enter the set $E$. Thus, we have that,
\begin{equation*}
\mathbb{P}^{ x}(\{T_{\partial \mathcal{B}_{a_n}} = T_{D_1}\} \cap \{ T_{E} < T_{\partial \mathcal{B}_{a_n} }\}) =\sum_{e \in E} \mathbb{P}^{x}\left( \{S_{T_E } =e\} \cap \bigcap_{t \le T_E}\{ S_t \not \in \partial \mathcal{B}_{a_n}\}\right) \mathbb{P}^e(T_{\partial \mathcal{B}_{a_n}} = T_{D_1}),
\end{equation*}
and,
\begin{equation*}
\mathbb{P}^{x}(\{T_{\partial \mathcal{B}_{a_n}} = T_{D_2}\} \cap \{ T_{E} < T_{\partial \mathcal{B}_{a_n}} \}) =\sum_{e \in E} \mathbb{P}^{x}\left( \{ S_{T_E } =e \}  \cap \bigcap_{t\le T_{E}} \{S_t \not \in \partial \mathcal{B}_{a_n}\}  \right) \mathbb{P}^e(T_{\partial \mathcal{B}_{a_n}} = T_{D_2}).
\end{equation*}
Since $e \in  H_1$ and these points $e \in H_1$ are of distance at least $\theta n$ from the dividing plane $L$, it is clear that 
\begin{equation} \label{eq:eisfarfromd}
\mathbb{P}^e(T_{\partial \mathcal{B}_{a_n}}  = T_{D_1})  \ge \mathbb{P}^e(T_{\partial \mathcal{B}_{a_n}} = T_{D_2}) .
\end{equation}
This proves the desired inequality \eqref{eq:uprbndpotential}. 

Finally, 
\begin{align*}
    \mathbb{P}^x(T_{V_1\cup E_n^2 }=\infty)
    \ge p(\eta) \mathbb{P}^x (T_D =T_{D_2}, T_{V_1} >T_D)
    \ge 
     \frac{p(\eta)}{2} \mathbb{P}^x (T_{V_1}>T_D)
    \ge \frac{p(\eta)}{2} \mathbb{P}^x (T_{V_1}=\infty ).
\end{align*}
To get the first inequality, we used the fact that every point on the boundary of $H_2$ must be of distance $d_K(\eta)$ away from the boundary of $\mathcal{B}_n$, where $d_K(\eta)$ is some function of $\eta$ that depends on the universal constant $K$. If we now let $y$ be the the point in which the random walk hits $D_2$, there must be some probability $p(\eta)$ (that is a function of this $d_{K}(\eta)$ and ultimately a function of $\eta$) such that  a walk that starts from $y$ will never again hit the ball $\mathcal{B}_n$. The second inequality comes from equation \eqref{eq:hitD2moreD1}. The last is derived from the fact that in three-dimensions $T_D$ is finite. This completes the desired assertion. Furthermore, we remark that $p(\eta)$ does not depend on the choice of $\theta$ from earlier. 
\end{proof}

\begin{proof}[Proof of $\mathbb{P}(A_n^2 \cap B_n \quad \text{i.o.})=0$]
As in the proof of $\mathbb{P}(A_n^1 \cap B_n \quad \text{i.o.})=0$, it suffices to only show that $\mathbb{P}(A_{t_n}^2 \cap B_{t_n} \quad\text{i.o.} ) =0$  along an appropriate sequence $t_n = q^n$ with $q>1$. Accordingly, we will also assume that the random walk $S[0,t_n]$ can be covered by the union $\mathcal{C}^*_{t_n}$, with the radii $r_n$, number of balls $b_n$, and stopping times $\hat{T}_i$ defined in the same way. We will show \eqref{assum:cap} to use Lemma \ref{cap2}. 

Now, consider some cap $E^1(x,\eta j_3(t_n))$. This cap will have surface area $O(\eta^2 j_3(t_n)^2)$ while the ball with radius $j_3(t_n)$ will have surface area $O(j_3(t_n)^2)$. We can cover the surface area of the sphere with $O(\eta^{-2})$ disjoint caps. Let $x_1,\ldots, x_{K}$ denote the center of these disjoint caps.

We will say that  a sphere $\mathcal{B}(S_{\hat{T}_i}, r_n)$ is contained in the cap $E^1(x_k,\eta j_3(t_n))$ if there exists some index $j$ satisfying $|j-i| \le  \frac{\log^{(2)} t_n}{(\log^{(3)} t_n)^2}$ and $\mathcal{B}(S_{\hat{T}_j}, r_n)$ intersects $E^1(x_k, \eta j_3(t_n))$. First, observe that the distance between the sphere $\mathcal{B}(S_{\hat{T}_j}, r_n)$ and $\mathcal{B}(S_{\hat{T}_i},r_n)$ will be no more than $ r_n(2  \frac{\log^{(2)}n}{(\log^{(3)} n)^2} +1) \ll  j_3(n)$. Thus, with this definition, any sphere $\mathcal{B}(S_{\hat{T}_i},r_n)$ belongs to at at most $4$ different caps.

By the pigeonhole principle, we can argue that there must be a cap that contains no more than $C \eta^2 $ balls $\mathcal{B}(S_{\hat{T}_j}, r_n)$  with our definition of containment. 

Now, consider a cap $E^1(x_k,\eta j_3(t_n))$ that contains no more than $C \eta^2$ balls. Let $\mathcal{X}_{x_k}$ denote the union of balls that are contained inside $E^1(x_k,\eta j_3(t_n))$ If the ball $\mathcal{B}(S_{\hat{T}_i},r_n)$ is contained inside $E^1(x_k,\eta j_3(t_n))$, then there must be some $j$ with $|j-i| \ge  \frac{\log^{(2)}t_n}{(\log ^{(3)} t_n)^2}$ such that all the balls $\mathcal{B}(S_{\hat{T}_k},r_n)$ with $k \in [i,j]$ are all contained inside $E^1(x_k,\eta j_3(t_n))$.  Thus, a Green's function lower bound for any $z$ belonging to a ball contained in $E^1(x_k,\eta j_3(t_n))$ would be,
\begin{equation*}
\sum_{i \in \mathcal{X}_{x_k}} \sum_{y \in\mathcal{B}(S_{\hat{T}_i},r_n)} G(z-y) \ge c\frac{|\mathcal{B}_{r_n}|}{r_n} \sum_{i=1}^{\epsilon\frac{\log^{(2)} t_n}{(\log^{(3)} t_n)^2} } \frac{1}{i} \le c \frac{|\mathcal{B}_{r_n}|}{r_n} \log^{(3)}t_n 
\end{equation*}
and hence
\begin{align*}
    \ca (S[1,t_n]\cap E^1(x_k,\eta j_3(t_n))) 
    \le& \ca(\cup_{i \in \mathcal{X}_{x_k}} \mathcal{B}(S_{\hat{T}_i},r_n))  \\
    \le & \frac{|\mathcal{B}_{r_n}||\mathcal{X}_{x_k}|}{c\frac{|\mathcal{B}_{r_n|} \log^{(3)}t_n}{r_n}} \le O(1) \eta^2 \frac{b_n r_n}{\log^{(3)}t_n} = O(1) \eta^2 h_3(t_n).
\end{align*}
We let $V_1$ be the set $S[1,t_n]\cap E^1(x_k,\eta j_3(t_n))$.  We see that this satisfies the conditions of Lemma \ref{cap2}, that is, \eqref{assum:cap} and thus, we see that $R_{t_n} \le \ca(E_{t_n}^2 \cup V_1) \le (1- f(\eta)) \ca(\mathcal{B}_{t_n})$.    
\end{proof}

\subsection{Proof of (iii) in Theorem \ref{m2}} 
\subsubsection{Preliminaries and Heuristics}  
Concerning $\mathbb{P}( A_n^1 \cap B_n \quad \text{i.o.})=0$: 
since $\ca (\mathcal{B}_{j_3(n)}) \ll h_3(n)$, it is trivial to show it.  
Then, we show $\mathbb{P}( A_n^2 \cap B_n \quad \text{i.o.})=1$. 
 As in the proof of (i), we can maximize the capacity by wrapping it uniformly around the sphere;  and, as before, we describe `wrapping uniformly around the sphere' as traveling around the $x-y$ cross sections of the sphere that are spaced at equidistant heights along the sphere. We will also derive lower bounds on the capacity by establishing upper bounds on the Green's functions sums $\max_{l}\sum_{l'} G(S_l-S_{l'})$. 
 However, we first remark that the main contribution to the Green's function sums would be from different locations. In (i), the main contribution to the sum of the Green's function was from the local contribution along each edge. 
 
 However, in the case of (iii), the main contribution must come from points that are of the distance scale of the same size as the sphere. To see this, one must notice that the global contribution is roughly the size of $\frac{n}{j_3(n)}$, while the local contribution is $\frac{n}{h_3(n)}$. Since $j_3(n) \ll h_3(n)$, the global contribution must be larger. One advantage that this gives us is that we can obtain the maximum value of the capacity using only a subset of the vertices of the random walk.  If we only consider a subset of $|T|$ vertices, it is only necessary for $\frac{|T|}{j_3(n)} \gg \frac{n}{h_3(n)}$. 
 Namely, $|T| = n k_n t(n)$. Here, $t(n)$ is a function such that $\lim_{n \to \infty} t(n) = \infty$, but should be thought of as doing so at a very slow rate. 

 Now, a cross-section will have a perimeter of order $j_3(n)$ and, if the random walk traverses the perimeter at a rate $\sqrt{\frac{\log^{(2)} n}{n}}$, then each cross-section will contain $ \text{O} \bigg(\sqrt{\frac{n}{\log^{(2)} n}}j_3(n)\bigg)  =  \text{O} \bigg(\frac{k_n n}{\log^{(3)} n}\bigg)$ points. 
 Since we need only $|T| = n k_n t(n)$ points, we need $ O((\log^{(3)} n) t(n))$ cross-sectional slices of the sphere.  Thus, the distance between the slices would be $ \frac{j_3(n)}{(\log^{(3)} n) t(n)}= \frac{k_n \sqrt{n \log^{(2)} n}}{(\log^{(3)} n)^2 t(n)}$. With this intuition in hand, we can now start introducing our appropriate discretization of the sphere.

We consider the following heights,
\begin{equation} \label{eq:heightsX}
  h_{\pm l} = \pm l  \frac{j_3(n)}{(\log^{(3)} n) t(n)}, \quad l\in \left \{0,\ldots, [1- \epsilon]   (\log^{(3)} n)t(n)\right\}.
\end{equation}

At each of these heights $h_l$, we will consider an $m$-gon at said height $h_l$. Namely, we consider the points 
\begin{equation*}
p_{a,l}:=\left(\sqrt{j_3(n)^2 - h_l^2} \cos(2\pi a m^{-1}), \sqrt{j_3(n)^2 - h_l^2} \sin(2 \pi a m^{-1})   , h_l\right).
\end{equation*}
We let $\mathcal{P}_k$ denote the cycle,
\begin{equation*}
\mathcal{P}^m_k:p_{0,k} \to p_{1,k} \to p_{2,k} \to \ldots \to p_{m-1,k} \to p_{m,k}.
\end{equation*}
Then, we consider a walk that makes the following sequence of moves,
\begin{equation}\label{eq:pathconst}
\mathcal{P}:=\mathcal{P}^m_{-(\log^{(3)} n)t(n)} \to \mathcal{P}^m_{-(\log^{(3)} n) t(n)+1}\to \ldots \mathcal{P}^m_{(\log^{(3)} n) t(n)}.
\end{equation}
We can alternatively denote the vertices as,
\begin{equation} \label{eq:path3}
\mathcal{P}:= v_0 \to v_1 \to v_2 \to \ldots \to v_f.
\end{equation}
The vertices $v_{j(m+1)},v_{j(m+1)+1},\ldots, v_{j(m+1)+m}$ correspond to vertices on the same cross-sectional level. 

Furthermore, it will traverse the edges at the rate $\sqrt{\frac{2 \log^{(2)} n}{3n}}$. Namely, if it is at vertex $v_{i}$ at time $t_i$ and vertex $v_{i+1}$ at time $t_{i+1}$, then we would have $\|v_{i+1} - v_i \| \approx \sqrt{\frac{\log^{(2)} n}{n}}(t_{i+1} -t_i) $. We can now give the necessary rigorous notation.

\subsubsection{Rigorous Notation}

Given the path, $\mathcal{P}$, from equation \eqref{eq:pathconst} and labeling from equation \eqref{eq:path3}, we define the distances,
\begin{equation*}
d_i:= \sum_{j=0}^{i-1} \|v_{j+1} - v_j \|, 
\end{equation*}
as well as the times,
\begin{equation*}
t_i := \frac{1}{1-\kappa} \sum_{j=0}^{i-1} \|v_{j+1} -v_j\| \sqrt{\frac{n}{\log^{(2)} n}}.
\end{equation*}
Now, $t_f= O(n k_n t(n)) \ll n$.

We now need to define appropriate events about controlling the path of the random walk. We need to introduce the following notational convention. If we write the superscript of a coordinate directions on top of the random walk, this means that we consider the projection of the random walk onto this coordinate direction. In addition, if $v_i \to v_{i+1}$ is an edge that lies along the cross-section of the sphere, we can define the edge $e_i$ to be the unit vector in the same direction as $v_i \to v_{i+1}$ and $o_i$ to be the unit vector that is orthogonal to the edge.  For $j \in \{x,y\}$, we define,
\begin{equation*} 
\begin{aligned}
\mathcal{FE}_i:= &\{  |S^{j}_{t_{i+1}} - v^{j}_{i+1}| \le \delta  j_3(n) \} \\
&\bigcap \left\{\max_{l \in [t_i,t_{i+1}]} |S^z_l - S^z_{t_i}| \le \sqrt{ j_3(n) \sqrt{\frac{2n}{3\log^{(2)} n}}} g(n) \right \}\\
&\bigcap \left\{\max_{l \in [t_i,t_{i+1}]}|\langle S_l -S_{t_i}, o_i \rangle| \le  \sqrt{ j_3(n) \sqrt{\frac{2n}{3\log^{(2)} n}}} g(n) \right\},
\end{aligned}
\end{equation*}
where $g(n)$ is some function such that $\lim_{n \to \infty} g(n) = \infty$. Note that, in contrast to the similar definition that was presented in equation  \eqref{eq:defFE}, we have to be more careful in how we specify transverse fluctuations. In addition to this, we bound the transverse fluctuations by the transverse factor $g(n)$ rather than $\log^{(4)} n$; adding this fluctuation parameter, $g(n)$, allows us to optimize how small we can take $j_3(n)$ while the condition $g(n) \to \infty$ still ensures that the probability of this fluctuation would be sufficiently rare.

For edges $v_i \to v_{i+1}$ that traverse cross sections, we can make them purely vertical and consider the event,
\begin{equation*}
\begin{aligned}
\mathcal{FE}_i:= & \left\{  |S^z_{t_{i+1}} - v^z_{i+1}| \le \delta \frac{ j_3(n)}{(\log^{(3)} n) t(n)} \right\}\\
& \bigcap \bigcap_{k \ne z} \left\{\max_{l \in [t_i,t_{i+1}]} |S^k_l - S^k_{t_i}| \le \sqrt{ \frac{ j_3(n)}{(\log^{(3)} n) t(n)} \sqrt{\frac{2n}{3\log^{(2)} n}}} g(n) \right \}.
\end{aligned}
\end{equation*}
In addition, we will define the event,
$$
\mathcal{FE}^i = \bigcap_{j=0}^{i} \mathcal{FE}_j,
$$
and the event $\mathcal{FE}^{f-1}$ would be the intersection of all of these events.

In what proceeds, we will present a collection of lemmas that are analogous to those that appear in the proof of part (i). Since the proofs are similar, we omit the details. 

\begin{lem} \label{lem:verttranslateiii}
In  the following, assume that the growth parameter $k_n \gg \frac{(\log^{(3)} n)^3}{\log^{(2)} n}t(n)g(n)$.
Further, recall that for any integer $j$, the vertices $v_{(m+1)j}, v_{(m+1)j+1},\ldots, v_{(m+1)j + m}$ correspond to the vertices at a cross-section of height $h_{-(\log^{(3)} n)t(n)+ j}$.

Under the event $\mathcal{FE}^{f-1}$, for times $[t_{(m+1)j},t_{(m+1)j+m}]$ and $ [t_{(m+1)(j+1)},t_{(m+1)(j+1)+m}]$, it must necessarily be the case that the walks $S_{[t_{(m+1)j},t_{(m+1)j+m}]}$ and $S_{ [t_{(m+1)(j+1)},t_{(m+1)(j+1)+m}]}$ are completely disjoint. Namely, for any $s_1 \in [t_{(m+1)j},t_{(m+1)j+m}] $ and $s_2 \in [t_{(m+1)(j+1)},t_{(m+1)(j+1)+m}]$ we have,
\begin{equation} \label{eq:zfluciii}
 |S^z_{s_1} - S^z_{s_2}| \ge (1-3 \delta)\frac{j_3(n)}{(\log^{(3)} n) t(n)}. 
\end{equation}

\end{lem}

\begin{proof}
This is analogous to Lemma \ref{lem:verttranslate}.
\end{proof}

\begin{lem} \label{lem:probFEiii}

Let $v_i \to v_{i+1}$ an edge along one of the $m$-gon cross-sections of the sphere.  In the event, $\mathcal{FE}^{f-1}$, we must have that,
\begin{equation} \label{eq:flucxyiii}
(1- 3 \delta) \| v_{i+1} - v_i \| < |\langle e_i, S_{t_i} - S_{t_{i+1}}\rangle| < (1+ 3 \delta) \|v_{i+1} -v_i\|.
\end{equation}
As a consequence of this as well as equation \eqref{eq:zfluc} we can deduce that,
\begin{equation*}
\mathbb{P}(\mathcal{FE}^{f-1}) \ge \exp\left[- C (1-\kappa)(1+4 \delta)^2 Ck_n t(n) \log^{(2)} n \right].
\end{equation*}
\end{lem}
\begin{proof}
This proof is analogous to that of Lemma \ref{lem:probFEiii}.
\end{proof}

We now define events that imply good properties of the Green's functions. Consider each $i$ that corresponds to some path $v_i \to v_{i+1}$ with the coordinate direction $j$ that is part of a cross-sectional $m$-gon; this path is associated with times $t_i \to t_{i+1}$. We say that a time $l \in [t_i,t_{i+1}]$ is a good time if the following is true.
We have the event,
\begin{equation*}
\mathcal{GF}_{i,l}:= \left\{ \sum_{t=t_{i}}^{t_{i+1}} G(S_t- S_{l}) \le \frac{(1+ 20\delta)}{1-\kappa} \frac{n}{h_3(n)}  \right\}.
\end{equation*}
As before, this makes sure the local contribution to the Green's function is not too large.
The event
\begin{equation*}
\begin{aligned}
\mathcal{GL}_{i,l}:= &\bigcap_{t \in \mathcal{T} }\bigg\{  (1+ 10\delta) \frac{|t-l|}{t_{i+1} -t_i} |\langle S_{t_{i+1}} - S_{t_i}, e_i \rangle| \ge |\langle S_t - S_l, e_i \rangle| \ge (1- 10\delta) \frac{|t-l|}{t_{i+1} -t_i}|\langle S_{t_{i+1}}-S _{t_i},e_i \rangle| \bigg \},\\
& \mathcal{T}:= [t - (\log^{(2)}n)^{-1}(t_{i+1} -t_i), t + (\log^{(2)} n)^{-1}(t_{i+1} -t_i)]^c \cap [t_i,t_{i+1}] . 
\end{aligned}
\end{equation*}
This event ensures the random walk grows linearly in the direction $v_i \to v_{i+1}$. The event,
\begin{equation} \label{eq:SGiii}
\mathcal{SG}_i:=\{|\{l :(\mathcal{GL}_{i,l} \cap \mathcal{GF}_{i,l})^c \text{ occurs} \} | \le \sqrt{\zeta_n} |t_{i+1} -t_i| \}
\end{equation}
ensures that there are very few bad points in each segment. Note that the difference here is that we impose that a vanishing fraction of points along each segment is allowed to be bad rather than a small but not decreasing fraction. The function $\zeta_n$ will be chosen later and is equal to the probability that the event $(\mathcal{GF}_{i,l})^c \cup (\mathcal{GL}_{i,l})^c$ occurs.

\begin{lem}\label{lem:inter}
Let $i$ be an index such that $v_i \to v_{i+1}$ is a path that traverses along an $m$-gon cross section of the sphere. 
Consider $p$ such that $\mathbb{P}(\mathcal{FE}^{i-1},S_{t_i}=p)$ is non-zero. 
There exists some $\zeta_n$ (that does not depend on $i$ or $p$) such that $\lim_{n \to \infty} \zeta_n =0$, and, 
$$
\mathbb{P}(\mathcal{GF}_{i,l}^c \cup \mathcal{GL}_{i,l}^c| \mathcal{FE}_i, S_{t_i}=p, \mathcal{FE}^{i-1}) \le \zeta_n .
$$

We use this $\zeta_n$ to define the parameter in equation \eqref{eq:SGiii} and can deduce that,
\begin{equation*}
\mathbb{P}(\mathcal{SG}_i^c|\mathcal{FE}_i, S_{t_i}=p, \mathcal{FE}^{i-1}) \le \sqrt{\zeta_n}.
\end{equation*}
As a consequence,  we can deduce that,
\begin{equation}\label{eq:coversquareiii}
\mathbb{P}\left(\mathcal{FE}^{f-1} \bigcap_{i=0}^{f-1} \mathcal{SG}_i\right) \ge \exp\left[-(1-\kappa)(1+5 \delta)^2 Ck_n t(n) \log^{(2)} n\right],
\end{equation}
for $n$ sufficiently large.
\end{lem}

\begin{proof}
One can follow the proof of Lemma \ref{lem:bndprobFESG} to find some $\delta_i$ such that the probability of $\mathbb{P}(\mathcal{GF}_{i,l}^c \cup \mathcal{GL}_{i,l}^c| \mathcal{FE}_i, S_{t_i}=p, \mathcal{FE}^{i-1}) \le \zeta_n $ for some $\zeta_n \to  0$. Then, we can apply Markov's inequality to argue that $\mathbb{P}(\mathcal{SG}_i^c|\mathcal{FE}_i, S_{t_i}=p, \mathcal{FE}^{i-1}) \le \sqrt{\zeta_n}$. 

After that, the proof of equation \eqref{eq:coversquareiii} follows the proof of equation \eqref{eq:coversquare}. 

\end{proof}

We can now introduce our final theorem on the computation of the capacity.
\begin{thm}\label{thm:lowerbndicapiii}
Consider those $j_3(n)$ that satisfy $k_n \gg \frac{(\log^{(3)} n)^3}{\log^{(2)} n}t(n)g(n)$.
Also, consider a walk that satisfies the event $\mathcal{FE}^{f-1} \bigcap_{i=0}^{f-1} \mathcal{SG}_i$. In this case, we can find some constant $C_{\delta}$ such that $C_{\delta} \to 0$ as $\delta \to 0$ and another constant $C_m$ such that $C_m \to 0$ as $m \to \infty$, in order to satisfy
\begin{equation*}
R_{t_f} \ge (1-\kappa) (1- C_{\delta}) (1-C_m)\ca(\mathcal{B}_{j_3(n)}).
\end{equation*}

\end{thm}

\begin{proof}
As in the proof of Theorem \ref{thm:lowerbndicap}, we can bound the capacity of a set by the inequality,
\begin{equation*}
\ca(S) \ge \frac{|S|}{\max_s\sum_{s'} G(s-s')}.
\end{equation*}
As before, we let $\mathcal{L}$ be the good times that satisfy the events $\mathcal{GL}_{i,l}$ and $\mathcal{GF}_{i,l}$. 
We can now compute the Green's functions in a similar way as in Theorem \ref{thm:lowerbndicap}. We remark here that the main contribution is no longer from the local segment, but from the global contribution. We also remark that choosing heights uniformly, as well as the linear rate of traversal along each segment from Lemma \ref{lem:probFEiii}, ensure that the global contribution to $\sum_{l' \in \mathcal{L}} G(S_l- S_{l'})$ approaches the integral $\int\frac{1}{\|x-y\|} \text{d}\mu(x) \text{d} \mu(y)$, where $\text{d}\mu$ is a uniform measure on the surface of a sphere of radius $j_3(n)$ as one takes $m$ larger and larger. 
\end{proof}

We can now finish the proof of part (iii) of Theorem \ref{m2}.
\begin{proof}[Part (iii) of Theorem \ref{m2}]
Using Lemma \ref{lem:inter}, we can ensure that the event $\mathcal{FE}^{f-1} \cap_{i=0}^{f-1} \mathcal{SG}_i$ can occur infinitely often. (Note that this controls the beginning of the walk up to time $n k_n t(n)$ rather than up to time $n$. The end of the walk can be arbitrary as long as it says within the sphere of radius $j_3(n)$. This is undemanding and can be imposed by having the walk go back and forth from the terminal point at time $n k_n t(n)$ and the center.) By taking $m \to \infty$ and $\delta, \kappa \to 0$, we can ensure the lower bound of $\ca(\mathcal{B}_{j_3(n)})$ on the random walk. 
\end{proof}

\appendix 

\section{Assorted Estimates }
This appendix will contain assorted estimates on the capacity that will appear periodically throughout the paper. 

\begin{lem} \label{lem:uprlwrcap}
Given a set $X \in \mathbb{Z}^3$, we have that
\begin{equation*}
 \frac{|X|}{\max_{x \in X} \sum_{y \in X} G(x-y)}\le \ca(X) \le \frac{|X|}{\min_{x \in X} \sum_{y \in X} G(x-y)}.
\end{equation*}
In the formula above, $X$ is allowed to contain multiple points. 
That is, we count the overlap in $\sum_{y \in X}$ or $|X|$. 

\end{lem}
\begin{proof}
For simplicity, we will present the proof when $X$ only contains distinct points. 
From the last hitting time decomposition, we can write, for $x \in X$
\begin{equation*}
1=  \sum_{y \in X} G(x-y) \mathbb{P}^y(T_X = \infty).
\end{equation*}
By adding all of these equations, we see that
\begin{equation*}
|X| =  \sum_{x \in X} \mathbb{P}^x(T_X = \infty) \sum_{y \in X} G(x-y).
\end{equation*}
We can either bound $ \sum_{y \in X} G(x-y)$ from above by the maximum, or from below by the minimum. Thus, we see that,
\begin{equation*}
 \frac{|X|}{\max_{x \in X} \sum_{y \in X} G(x-y)}\le\sum_{x \in X} \mathbb{P}^x(T_X = \infty) \le \frac{|X|}{ \min_{x \in X} \sum_{y \in X} G(x-y)}.
\end{equation*}

\end{proof}

\begin{lem} \label{lem:uprsumcap}
Let $X$ and $Y$ be two set, not necessarily distinct and each allowing multiple points. We have that,
\begin{equation*}
\ca(X \cup Y) \le \ca(X) + \frac{|X \cup Y|}{ \min_{y \in Y} \sum_{z \in X \cup Y} G(z-y)}.
\end{equation*}
\end{lem}
\begin{proof}
Again, for simplicity, we will only present the proof when $X$ and $Y$ only consist of distinct points without multiplicity. 
We see that by the first hitting time decomposition, we have, for $z_1 \in X \cup Y$
\begin{equation*}
1 = \sum_{z_2 \in X \cup Y} G(z_1 -z_2) \mathbb{P}^{z_2}(T_{X \cup Y} = \infty).
\end{equation*}
We sum up this quantity over all $y$ in $Y$ and observe that,
\begin{equation*}
\begin{aligned}
& |X| + |Y| = \sum_{z_1 \in X \cup Y} \mathbb{P}^{z_1}(T_{X \cup Y} = \infty) \sum_{z_2 \in X \cup Y} G(z_1-z_2)  \ge  \sum_{y \in Y} \mathbb{P}^{y}(T_{X \cup Y } =\infty) \sum_{z \in X \cup Y} G(y-z)\\
& \ge \sum_{ y \in Y} \mathbb{P}^y(T_{X \cup Y} = \infty) \min_{y \in Y} \sum_{z \in X \cup Y} G(y-z).
\end{aligned}
\end{equation*}
Thus, 
\begin{equation*}
\sum_{ y \in Y} \mathbb{P}^y(T_{X \cup Y} = \infty) \le \frac{|X|+|Y|}{\min_{y \in Y} \sum_{z \in X \cup Y} G(y-z)}.
\end{equation*}
Now,
\begin{equation*}
\begin{aligned}
    &\ca(X \cup Y) = \sum_{x \in X} \mathbb{P}^x(T_{X \cup Y} = \infty) + \sum_{y \in Y} \mathbb{P}^y(T_{X \cup Y} = \infty)  \le \sum_{x \in X} \mathbb{P}^x(T_X = \infty) + \sum_{y \in Y} \mathbb{P}^{y}(T_{X \cup Y} = \infty) \\
    &\le \ca(X) + \frac{|X|+|Y|}{ \min_{y \in Y} \sum_{z \in X \cup Y} G(y-z)}.
\end{aligned}
\end{equation*}
\end{proof}


\begin{thebibliography}{9}

\bibitem{AdhikariOkada} Adhikari A. and Okada I. Moderate Deviations for the Capacity of the Random Walk range in dimension four, to appear in {\it Ann. Probab.}

\bibitem{BBH01} van den Berg M., Bolthausen E. and den Hollander F. Moderate deviations for the volume of the Wiener sausage. {\it Ann. Math.} (2) {\bf 153} (2001), no. 2, 355-406.

\bibitem{Bo78} Bolthausen E. On the speed of convergence in Strassen’s law of the iterated logarithm. {\it Ann. Probab.} {\bf 6} (1978), 668–672.

\bibitem{Ch17} Chang Y. Two observations on the capacity of the range of simple random walk on $\Z^3$ and $\Z^4$. {\it Electron. Com. Probab.} {\bf 22} (2017), No. 25, 9 pp.

\bibitem{CKW} Chen X., Kuelbs J. and  Li W. A Functional LIL for Symmetric Stable Processes. {\it Ann. Probab.} {\bf 28}, 1 (2000), 258-276.

\bibitem{DemboOkada} Dembo A. and Okada I. Capacity of the range of random walk: The law of the iterated logarithm. {\it Ann. Probab.} {\bf 52}, 5  (2024), 1954-1991. 


\bibitem{DemboOkada+} Dembo A. and Okada I. Errata for: Capacity of the range of random walk: The law of the iterated logarithm. 

\bibitem{GL99} Gorn N. and Lifshits M. Chung’s law and the Cs\'{a}ki function, {\it J. Theor. Probab.} {\bf 12} (1999), 399–420.

\bibitem{Grill92} Grill K. Exact rate of convergence in Strassen’s law of the iterated logarithm. {\it J. Theor. Probab.} {\bf 5} (1992), 197–205.



\bibitem{HS20} Hutchcroft T. and Sousi P.
Logarithmic corrections to scaling in the four-dimensional uniform spanning tree. {\it Comm. Math. Phys.} {\bf 401} (2023), no.2, 2115-2191.

\bibitem{Kent80}Kent J. T. Eigenvalue expansions for diffusion hitting times. Z. Wahrsch. Verw. Gebiete, 52 (1980), 309-319


\bibitem{LA91} Lawler G. F. Intersections of random walks. Second edition, Birkhauser, 1996.



\bibitem{PS51} P\'{o}lya G. and Szeg\"{o} G. Isoperimetric inequalities in mathematical physics. {\it Annals of Mathematics Studies}, no. {\bf 27}, Princeton University Press, Princeton, N.J., (1951).


\bibitem{Strassen} Strassen V. An invariance principle for the law of the iterated logarithm. {\it Z. Wahrsh. Verw. Geb.} {\bf 3} (1964), 211–226. 


\bibitem{SZ10} Sznitman A.-S. Vacant set of random interlacements and percolation. {\it Ann. Math.} (2) {\bf 171} (2010), no. 3, 2039-2087. 


\bibitem{Tala92} Talagrand M. On the rate of clustering in Strassen’s LIL for Brownian motion. {\it In Probability in Banach Spaces}, Birkh\"{a}user, Boston {\bf 8} (1992), 339–347.


\end{thebibliography}
\end{document}